%% file: invhilb-RHP-final-rev.tex
\documentclass[10pt]{article}

\usepackage{epsfig,cancel,amsmath,amsthm,amssymb}
\usepackage{color,soul}

\usepackage{wrapfig}
\usepackage{color}
\usepackage{graphicx,framed}
\usepackage[colorlinks=true, pdfstartview=FitV, linkcolor=black, citecolor=black, urlcolor=blue]{hyperref}
\definecolor{shadecolor}{rgb}{0.95, 0.95, 0.86}

\def \gg{{\mathfrak g}}
\def \wt{\widetilde}
\def\ra{\rightarrow}
\newcommand{\al}{\alpha}
\def\Id{ \mathrm{Id}} 
\newcommand{\bt}{\beta}

\newcommand{\G}{\Gamma} 
\renewcommand{\O}{\Omega} 
\renewcommand{\k}{\varkappa} 
\renewcommand{\d}{\delta}
\newcommand{\D}{\Delta} 
\newcommand{\e}{\epsilon}
\renewcommand{\o}{\omega}  
\newcommand{\g}{\gamma} 
\newcommand{\la}{\lambda}

\newcommand{\pa}{\partial}

\newcommand{\ts}{\text{supp}\,} 
 
\newcommand{\br}{{\mathbb R}}

\newcommand{\ioi}{\int_0^{\infty}}

\newcommand{\CH}{{\mathcal H}}

\newtheorem{theorem}{Theorem}[section]
\newtheorem{example}[theorem]{Example}
\newtheorem{exercise}[theorem]{Exercise}

\newtheorem{lemma}[theorem]{Lemma}
\newtheorem{remark}[theorem]{Remark}
\newtheorem{problem}[theorem]{Riemann-Hilbert Problem}

\newtheorem{proposition}[theorem]{Proposition} 
\newtheorem{corollary}[theorem]{Corollary} 
\newtheorem{definition}[theorem]{Definition}
\def\le{\left}
\def\ri{\right}
\def\ds{\displaystyle}

\def\res{\mathop{\mathrm {res}}\limits_}

\def\bth{\begin{theorem}}
\def\et{\end{theorem}}
\def\bc{\begin{corollary}}
\def\ec{\end{corollary}}
\def\bx{\begin{example}\small}
\def\ex{\end{example}}
\def\bxr{\begin{exercise}\small}
\def\exr{\end{exercise}}
\def\bl{\begin{lemma}}
\def\el{\end{lemma}}
\def\bd{\begin{definition}}
\def\ed{\end{definition}}
\def\bp{\begin{proposition}}
\def\ep{\end{proposition}}

\def\br{\begin{remark}}
\def\er{\end{remark}}

\def\be{\begin{equation}}
\def\ee{\end{equation}}
\def\ov {\overline}

\def\&{\hspace{-15pt}&}
\def\bea{\begin{eqnarray}}
\def\eea{\end{eqnarray}}
\def\beas{\begin{eqnarray*}}
\def\eeas{\end{eqnarray*}}

\def \pa{\partial}
\def\C{{\mathbb C}}

\def\R{{\mathbb R}}
\def\N{{\mathbb N}}

\def\wh{\widehat}

\def\Z{{\mathbb Z}}
\def\u{\mathfrak u}

\def\K{\mathcal K}
\def\l{\lambda}

\def\1{{\bf 1}}

\def\s{ {\sigma}} 
\def\t{ {\tau}} 
 
\def\Th{ {\Theta}} 

\def\z{\zeta}

\def\hf{\frac{1}{2}}

\newcommand{\Rscr}{\mathcal R}
\numberwithin{equation}{section}

\begin{document}
\baselineskip 18pt plus 1pt minus 1pt

\begin{center}
\begin{large}

\textbf{\sc Singular value decomposition of a finite Hilbert transform defined on several intervals and  the interior problem of tomography: the Riemann-Hilbert problem approach} 
\end{large}
\bigskip

M. Bertola$^\dagger$\footnote{The work was supported  in part by the Natural Sciences and Engineering Research Council of Canada. }
A. Katsevich$^\star$\footnote{The work  was supported in part by NSF grants DMS-0806304 and DMS-1211164.}
 and 
 A. Tovbis $^\star$\footnote{The work  was supported in part by NSF grant DMS-1211164.}
 \bigskip
 
\begin{small}
$^{\dagger}$ {\em Centre de recherches math\'ematiques,
Universit\'e de Montr\'eal\\ C.~P.~6128, succ. centre ville, Montr\'eal,
Qu\'ebec, Canada H3C 3J7} and \\
 {\em  Department of Mathematics and
Statistics, Concordia University\\ 1455 de Maisonneuve W., Montr\'eal, Qu\'ebec,
Canada H3G 1M8} \\
\smallskip
$^{\star}$ {\em  University of Central Florida
	Department of Mathematics\\
	4000 Central Florida Blvd.
	P.O. Box 161364
	Orlando, FL 32816-1364
} \\
\end{small}
\end{center}
{\it E-mail:} bertola@mathstat.concordia.ca, Alexander.Katsevich@ucf.edu, Alexander.Tovbis@ucf.edu
\begin{center}{\bf Abstract}\\
\end{center} 
\baselineskip 12pt

We study the asymptotics of singular values and singular functions of a Finite Hilbert
transform (FHT), which is defined on several intervals. Transforms of this kind arise in the study of the
interior problem of tomography. We suggest a novel approach based on the technique of the
matrix Riemann-Hilbert problem and the steepest descent method of Deift-Zhou. We obtain a family
of matrix RHPs depending on the spectral parameter $\lambda$ and show that the singular values of the FHT 
coincide with the values of $\lambda$ for which the RHP is not solvable. Expressing the leading
order solution as $\lambda\to 0$ of the RHP in terms of the Riemann Theta functions, we prove that the asymptotics
of the singular values can be obtained by studying the intersections of the locus of zeroes of a certain Theta
function with a straight line. This line can be calculated explicitly, and it depends on the geometry of the intervals that
define the FHT. The leading order asymptotics of the singular functions and singular values are explicitly expressed
in terms of the Riemann Theta functions and of the period matrix of the corresponding normalized differentials,
respectively. We also obtain the error estimates for our asymptotic results.

\baselineskip 12pt
\tableofcontents

\section{Introduction}\label{math-intro}

Fix any $2g+2$, $g\in\N$, distinct points $a_i$ on the real line $a_i<a_{i+1}$, $i=1,2,\dots,2g+1$. Consider the Finite Hilbert Transform (FHT)
\begin{equation}\label{def-hilb}
(\CH f)(x):= \frac1\pi \int_{a_1}^{a_{2g+2}} \frac{f(y)}{y-x}dy,\ f\in L^2([a_1,a_{2g+2}]).
\end{equation}
Here and throughout the paper singular integrals are understood in the principal value sense.
It is well-known that if $f$ is supported on $[a_1,a_{2g+2}]$ and $\CH f$ is known on $[a_1,a_{2g+2}]$, then one can stably 
reconstruct $f$ using classical FHT inversion formulas \cite{tric} (properties of the FHT in various $L^p,~p>1,$ spaces can be found in 
\cite{oe91}). 
In some applications, {for example, in tomography,} there arise problems with incomplete data, 
where $\CH f$ is known only on a subinterval of $[a_1,a_{2g+2}]$ 
(see Section~\ref{motivation}). 
Since any singularity of $f$ located outside of the interval where $\CH f$ is given is smoothed out and not visible from the data, we conclude that stable recovery of $f$ may be possible only on the interval where $\CH f$ is known. Thus we suppose that our data are
\begin{equation}\label{hilb-dataintro}
(\CH f)(x)=\varphi(x),\ x\in [a_2,a_{2g+1}],
\end{equation}
and we want to find $f$ on $[a_2,a_{2g+1}]$. Consider the operator 
$$\CH:\, L^2([a_1,a_{2g+2}])\to L^2([a_2,a_{2g+1}]).$$ Unique recovery 
of $f$ on $[a_2,a_{2g+1}]$ is impossible since $\CH$ has a non-trivial kernel (see \cite{kt12} for its complete description). 
Therefore, to achieve  unique recovery the data $\varphi$ should be augmented by some additional information. One type of 
information that guarantees uniqueness is the knowledge of $f$ on some {interval or}
intervals inside $[a_2,a_{2g+1}]$. Let us assume that $f$ is known on the interior intervals 
\begin{equation}\label{int-int}
I_i:=[a_3,a_4]\cup [a_5,a_6]\cup\dots\cup [a_{2g-1},a_{2g}].
\end{equation}
Denote by $I_e:=[a_1,a_2]\cup [a_{2g+1},a_{2g+2}]$ the remaining ``exterior'' intervals. Applying the FHT inversion formula (see e.g. \cite{oe91})
to $\varphi(x)=(\CH f)(x)\in L^2([a_1,a_{2g+2}])$, we get
\begin{equation}\label{hilb-inv}
\begin{split}
f(y)&=-\frac{w(y)}{\pi}\left(\int_{a_1}^{a_2}+\int_{a_{2g+1}}^{a_{2g+2}}\right) \frac{\varphi(x)}{w(x)(x-y)}dx-\frac{w(y)}{\pi}\int_{a_2}^{a_{2g+1}}\frac{\varphi(x)}{w(x)(x-y)}dx,\\
&
{\rm where}~~ a_1<y<a_{2g+2}~~{\rm and}
~~w(x):=\sqrt{(a_{2g+2}-x)(x-a_1)}
.
\end{split}
\end{equation}
The left side of (\ref{hilb-inv}) is known on $I_i$. The last integral on the right is known everywhere. Combining these known quantities we get an integral equation:
\begin{equation}\label{int-eq}
(\CH^{-1}_e\varphi)(y):=-\frac{w(y)}\pi \int_{I_e} \frac{\varphi(x)}{w(x)(x-y)}dx=\psi(y),\ y\in I_i,
\end{equation}
where 
\be\label{psi}
\psi(y)= f(y)+\frac{w(y)}{\pi}\int_{a_2}^{a_{2g+1}}\frac{\varphi(x)}{w(x)(x-y)}dx,\ y\in I_i,
\ee
is a known function. 
Here and throughout the paper the symbol $\mathcal H_e^{-1}$  denotes  only the restriction of the 
inverse of $\mathcal H$ to the set  $I_e$  but not the inverse operator itself.
The problem of finding $f$ can be solved in two steps. In step 1 we solve equation (\ref{int-eq}) for $\varphi(x)$ on  $I_e$. In step 2 we substitute the computed $\varphi(x)$ into (\ref{hilb-inv}) and recover $f$. 
\begin{figure}
\centerline{
\resizebox{0.5\textwidth}{!}{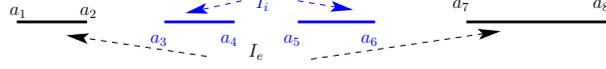}}
\caption{
A schematic arrangement of the multi intervals $I_e$ (external) and $I_i$ (internal), here with $g=3$.}
\vspace{-10pt}
\end{figure}
It is clear that solving (\ref{int-eq}), i.e. inverting $\CH^{-1}_e$, is the most unstable step. The study of this step is the main motivation for this paper.  We consider the operator $\CH^{-1}_e$ in (\ref{int-eq}) as a map between two weighted $L^2$-spaces:
\begin{equation}\label{map}
\CH^{-1}_e:\ L^2(I_e,1/w)\to L^2(I_i,1/w).
\end{equation}
Then its adjoint is the Hilbert transform:
\begin{equation}\label{hilb-adj}
(\CH_i\psi)(x):= \frac1\pi \int_{I_i} \frac{\psi(y)}{y-x}dy,\ x\in I_e.
\end{equation}
The weighted  spaces for the operator $\CH^{-1}_e$ in (\ref{map}) are  naturally determined by  the structure of \eqref{int-eq}. In particular,  as follows from inequality (1.7) of \cite{eo04}, $\CH^{-1}_e$ is a continuous operator from  $L^2(I_e,1/w)$ to $L^2(I_i,1/w)$, but it is not continuous from $L^2(I_e)$ to $L^2(I_i)$.

Our aim is to study the singular value decomposition (SVD) for the operator  $\CH^{-1}_e$.
Namely, we are interested in the singular values  $2\l=2\l_n>0$, $n\in\N$, and the corresponding left and right singular functions $f=f_n,~h=h_n$,  satisfying 
\begin{equation}\label{svd-def}
\begin{split}
(\CH^{-1}_e h)(y)=-\frac{w(y)}{\pi}&\int_{I_e} \frac{h(x)}{w(x)(x-y)}dx={2}\la f(y),\ y\in I_i,\\
(\CH_i f)(x)=\frac1\pi &\int_{I_i} \frac{f(y)}{y-x}dy={2}\la h(x),\ x\in I_e.
\end{split}
\end{equation}
Note that both integrals in \eqref{svd-def}
are  nonsingular.

It is well known that the rate at which the $\l_n$'s approach zero is related with the ill-posedness of inverting $\CH^{-1}_e$.
Because of the symmetry $(\l,f,h)\Leftrightarrow (-\l,-f,h)$ of \eqref{svd-def}, we are interested
only in positive $\l_n$. Thus, 
equation \eqref{svd-def} is the main object of the present work, and obtaining {\bf the large $n$ asymptotics of $\l_n$, $f_n$ and $h_n$ is our main goal}.

Our approach relies upon a reformulation of SVD problem for \eqref{svd-def} in terms of a 
{\em matrix Riemann Hilbert Problem} (RHP) depending on a 
large parameter $\k=-\ln \l$  and, subsequently, the use of the 
 steepest descent method
of Deift and Zhou for the large $\k$ asymptotics of this RHP. 
In the last two decades, the steepest descent method for asymptotic solution of matrix RHPs
has found an 
  increasingly wide scope of applications, such as, for example: 
universality in  random  matrix theory (as the size of the matrix becomes large)  and closely related  asymptotic problems for
 large degree  orthogonal polynomials; the long time and semiclassical asymptotics of integrable nonlinear equations; 
 connection formul\ae\ for Painlev\'e transcendents; approximation theory; behavior of gap-formation probabilities in certain 
random point processes, etc. This list  can certainly be continued. A good introduction to  matrix Riemann-Hilbert problems
and their applications to random matrices and orthogonal polynomials can be found in the book  \cite{Deift} of P. Deift;
some more recent results and perspectives can be also found in the \cite{Deift60volume}.

The above examples require the {leading order approximate} solution of an RHP 
in the appropriate asymptotic regime; the existence of such leading order solutions
is a general feature of the above examples.

More recently, there appeared problems 
 where the failure of solvability of the corresponding RHP was of particular interest; relevant examples include 
the description of the first oscillations behind the point of gradient catastrophe in the 
focusing nonlinear Schr\"odinger equation \cite{BertolaTovbisNLS2}, the poles of the  Painlev\'e transcendents  (see \cite{ItsKapaevFokasbook} and references therein) etc.

The present paper is a further step in this general direction; the twist, however,  
is that our interest is now focused on the non-generic situation where our RHP {\em is not solvable}.
This situation leads to the detailed study of 
a particular object which is known in the literature on Riemann surfaces as the ``theta divisor''. Without entering into details now, the theta divisor is the locus of zeros of a particular analytic function (Riemann Theta function $\Theta$) of several variables. These variables encode the parameters of the problem we  study (the points $a_j$, $j=1,\dots,2g+2$) as well as the spectral variable $\k$. When these variables are on the theta divisor,  the spectral parameter $2\l=2e^{-\k}$ approximates a singular value of $\CH^{-1}_e$.
This is somewhat similar  to the eigenvalue problem for the finite spherical well in quantum mechanics, where the zeroes of a special function (Bessel function) correspond to the eigenvalues of the problem.

In order to achieve this analysis we need to use certain properties of Theta functions, some of which
can be found in Appendix \ref{thetaapp}.
 It seems that even in the specialized literature on Theta functions we could not find  the results that we need and thus a good deal of our effort goes in that direction.

Let us now return to the system \eqref{svd-def}. The goals of the paper are:  
\begin{itemize}
\item to describe the asymptotics  of singular values (including their multiplicity) $2\l_n$ of $\CH^{-1}_e$;
\item to describe the  asymptotics of the left and right singular functions $f_n$ and $h_n$.
\end{itemize}
 As it was mentioned above, we will associate  the singular-value problem for ${\CH^{-1}_e}$ with  an 
RHP \ref{RHPGamma} 
that depends on a parameter $\l$ and prove that: 
 \begin{enumerate}
 \item the singular values $2\l_n$ of ${\CH_e^{-1}}$ are simple and correspond exactly to the values of $\l$ for which the RHP \ref{RHPGamma}
is {\em not} solvable;

 \item in the regime of small $\l$ the RHP \ref{RHPGamma} --and its non solvability-- can be approximated by another "model" RHP
\ref{modelRHP}. The solvability of the latter depends entirely on whether a given value of the spectral parameter $\k = -\ln \l$
will turn 
a certain Riemann Theta function $\Theta$ into zero, that is, whether $\k$ brings the (vector) argument of  $\Theta$
on the theta-divisor.
 \end{enumerate}

The description of our findings for the eigenfunctions would require the introduction of many notations related to Theta functions; 
thus, we found it expedient to  refer the reader  to  Corollary \ref{cor-sing-func} in Section \ref{secdivisor}.
It is however possible to describe here the asymptotics of the singular values $2\l_n$ of ${\CH^{-1}_e}$ from Theorem \ref{theo-ass}; namely
\be
\l_n  = {\rm e}^{-\frac {n  i \pi}{\tau_{11}} + \mathcal O(1)}\ ,\ \ n\to \infty, \label{in1}
\ee
where
 $\tau_{11}$ is a purely imaginary number with positive imaginary part\footnote{The positivity follows from Riemann's Theorem  \ref{Riemann1}.}. Specifically, it is the $(1,1)$ entry of
the normalized matrix of periods $\t$  associated to a double-sheeted covering of the plane, slit along the segments constituting $I_i, I_e$ (i.e., a hyperelliptic surface). If we introduce a $g\times g$ matrix $\mathbb A$ by 
\be
(\mathbb A)_{kj}=2\int_{a_{2k}}^{a_{2k+1}}\frac{z^{j-1} dz}{R(z)}, ~~~k=1,\dots,g-1~~~~{\rm and}~~~
(\mathbb A)_{gj}=2\int_{a_{1}}^{a_{2g+2}}\frac{z^{j-1} dz}{R_+(z)},~~~~j=1,\dots,g,
\ee
where $R(z)  = \prod_{j=1}^{2g+2}(z-a_j)^\hf$ is an analytic function on  $\C \setminus (I_e \cup I_i)$  behaving as $z^{g+1}$ at infinity, then
\be\label{tau11intro}
\t_{11}=-2\sum_{j=1}^g (\mathbb A^{-1})_{j1}\int_{I_e}\frac{z^{j-1} dz}{R_+(z)}.
\ee
Here and throughout the paper the subscripts $\pm$ routinely denote limiting values of  functions (vectors, matrices) from the left/right
side of  corresponding oriented arcs. In particular,
$R_+$ means the limiting value of $R$ on $I= I_e\cup I_i$ from $\Im z>0$. (The matrices $\mathbb A,\t$ are defined  by equations \eqref{1stkind} and \eqref{taumatrix} respectively.)
The analysis behind \eqref{in1}
requires describing the set of zeroes of $\Theta$ (the theta-divisor)
and ensuring that,  as $\l\to 0_+$, the RHP \ref{modelRHP} becomes unsolvable infinitely many times.
Locating the values for which this happens leads to \eqref{in1}.
The asymptotics \eqref{in1} is illustrated by  Figure  \ref{asympteignum}, left,  where the first 22  numerically simulated $\l_n$ 
are compared with  the asymptotic formula  \eqref{in1}.
\begin{figure}
\label{asympteignum}
\begin{tabular}{ccc}
\vspace{-10pt}\includegraphics[width=0.33\textwidth]{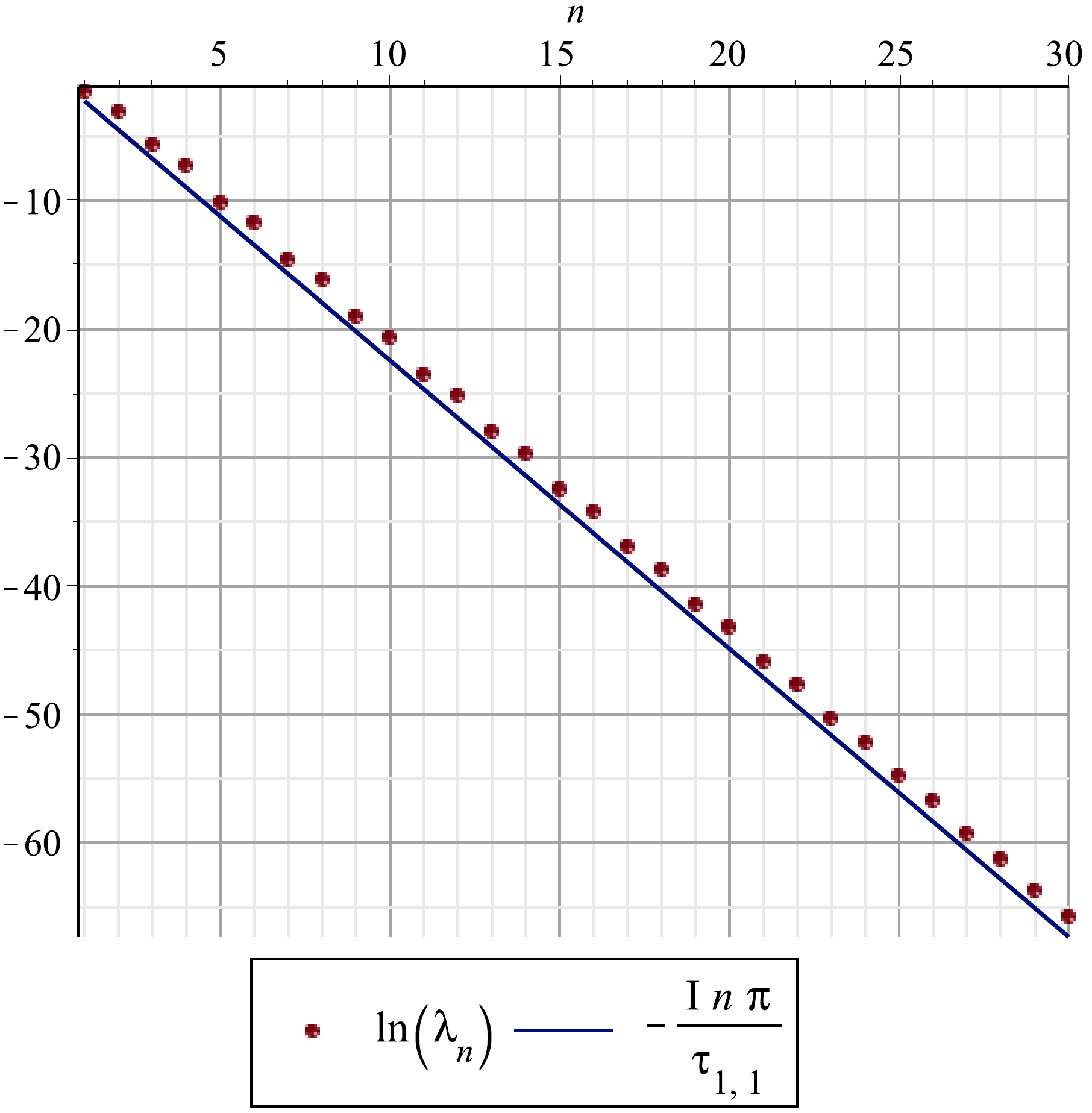}
\includegraphics[width=0.3\textwidth]{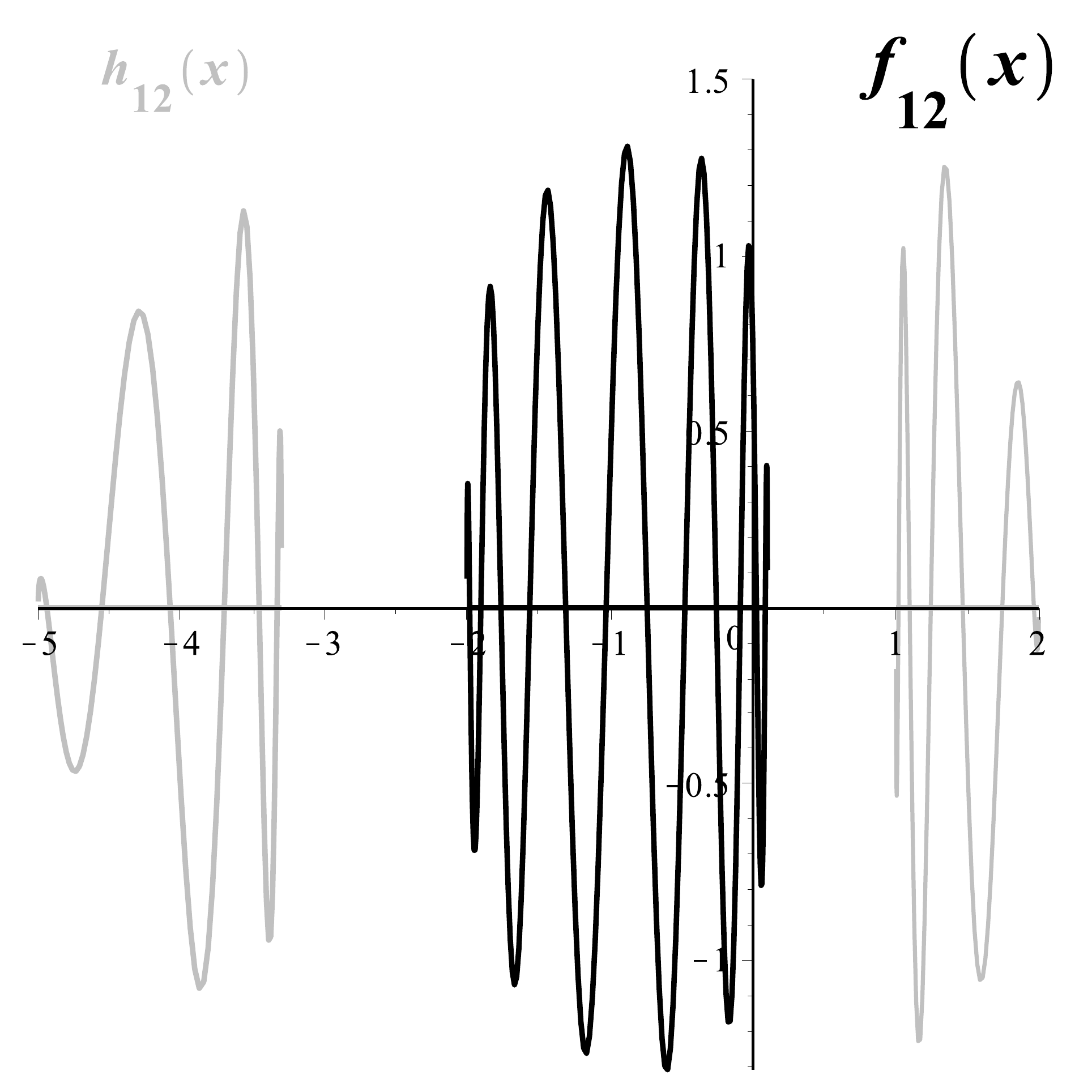}
\includegraphics[width=0.3\textwidth]{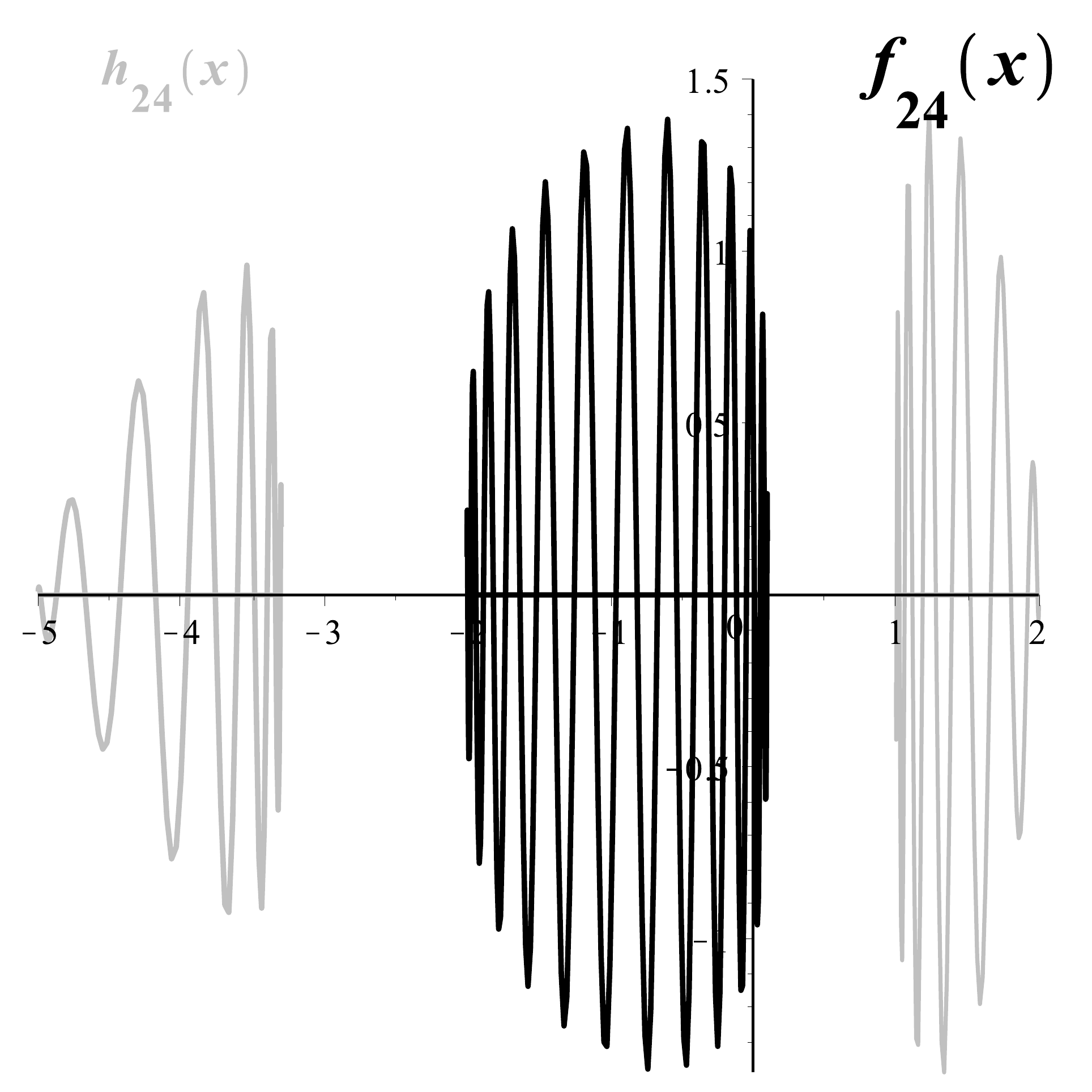}
\end{tabular}
\caption{
Plotted in the figure are $\ln (\lambda_n)$ versus the straight line $n\mapsto -\frac {i n \pi}{\tau_{11}}$ for the  choice of the
successive endpoints $\{-5,-3.3,-2,0.1,1,2\}$ (the case of $g=2$). The slope matches perfectly;
the shift of the line is due to the $\mathcal O(1)$ term in \eqref{in1}.
 The small fluctuations of $\l_n$ away from  linear behavior   are explained by the periodicity of the Theta divisor and the fact that the line shown in  Figure \ref{ThetaDivisor} has, in general, an irrational slope, so that the intersections are quasi-periodic. Observe that the log-linearity of the eigenvalues seems to be correct not only asymptotically in $n$, but even for the first singular values.
 Also plotted are  two pairs of the corresponding singular functions $(f_{12},h_{12})$ and  $(f_{24},h_{24})$, obtained numerically
simultaneously with $\l_n$. Note, the envelope of the oscillations is already visibly the same, as expected from the asymptotic description in Corollary \ref{cor-sing-func}, Remark \ref{signchange2}.   }
\end{figure}

The outline of our paper is the following.
 In Section \ref{motivation} we discuss practical motivations of the SVD problem \eqref{svd-def} coming from  tomography.
In Section \ref{secKernel} we 
 reformulate the SVD problem as an eigenvalue problem for an appropriate self-adjoint operator $\wh K$ defined on $L^2(I)$. 
We next  associate with $\wh K$ a matrix RHP \ref{RHPGamma} in terms of which we  can construct the resolvent operator $\wh R$ of $\wh K$
 and hence, the eigenfunctions of $\wh K$. 
We show that $2\l$ is a singular value of $\CH^{-1}_e$ if and only if $\l$ is a positive eigenvalue of $\wh K$, which is equivalent to the assertion that the corresponding  RHP \ref{RHPGamma}  does not have a solution. 
We also prove that the singular values $2\l_n$ of $\CH^{-1}_e$ are simple
 and that the singular functions $f_n(z)=\sqrt{w(z)}\phi_n(z)\chi_i(z)$,~$h_n(z)=\sqrt{w(z)}\phi_n(z)\chi_e(z)$, where 
$\wh K\phi_n=\l_n\phi_n$ and $\chi_i(z), \chi_e(z)$ denote the characteristic functions of $I_i,I_e$ respectively.

To obtain the large $n$ asymptotics
of $\l_n$ we need to study the asymptotic limit as $\l\ra 0$ of the solution
 $\G=\G(z;\l)$ of RHP \ref{RHPGamma}. That is done in Section \ref{sec-solGam},
where the RHP \ref{RHPGamma} is asymptotically reduced to the model RHP
\ref{modelRHP}. The latter can be solved explicitly in terms of Riemann Theta functions,
 see Section \ref{secmodelRHPM}, Theorem \ref{theoremPsi}. 
The accuracy of replacing  the RHP \ref{RHPGamma} by the model RHP
\ref{modelRHP} is evaluated in Section \ref{erroranalysis}.
The asymptotics \eqref{in1} for the singular values, as well as the
large $n$ asymptotics of the singular functions  $f_n,h_n$, see Theorem \ref{asympteigf}, 
 are obtained in Section \ref{secdivisor}.
Finally, some basic facts about Riemann Theta functions are provided in Appendix \ref{thetaapp}.

\section{Practical motivation for the problem: relationship with tomography}\label{motivation}

In 1991 Gelfand and Graev derived a formula, which gives the Hilbert transform of a function $f$ from a collection of line integrals of $f$ \cite{gegr-91}. Let $f$ be sufficiently smooth and compactly supported. We start with a practically relevant 3D case. Let $D_f(y,\bt)$ be the collection of integrals of $f$ along lines intersecting a fixed piecewise smooth curve $\Gamma$:
\begin{equation}\label{cb-def}
D_f(y,\bt)= \ioi f(y+t\bt) dt,\ y\in\Gamma,
\end{equation}
where $\bt$ is a unit vector. The map $f\to D_f$ is known as the cone-beam transform of $f$. The step of going back from $D_f$ to $f$, where $f$ represents the attenuation coefficient of the object being scanned, is the main mathematical principle on which a vast majority of CT scanners are based today. 

Let $y(s)$ be a parametrization of $\Gamma$. We assume that $\Gamma$ does not self-intersect and is traversed in one direction as $s$ varies over some interval $I$. Pick any two values $s_1,s_2\in I$, $s_1\not=s_2$. Let $\al$ be a unit vector along the chord $y(s_1),y(s_2)$. Then one has \cite{gegr-91}:  
\begin{equation}\label{main-id}
\frac12\int_{s_1}^{s_2} \frac1{|x-y(s)|}\left.\frac{\pa}{\pa\la}D_f\left(y(\la),\frac{x-y(s)}{|x-y(s)|}\right)\right|_{\la=s}ds
=\int \frac{f(x+t\al)}{t} dt,
\end{equation}
where $x$ is located on the chord between $y(s_1)$ and $y(s_2)$.
Equation (\ref{main-id}) implies that knowing the cone beam transform of $f$ one can compute the Hilbert transform of $f$ on the chords of $\Gamma$. An analogous result holds in 2D as well, where the corresponding collection of line integrals of $f$ is known as the fan beam transform of $f$. 

\begin{figure}[h]
{\centerline
{\epsfig{file=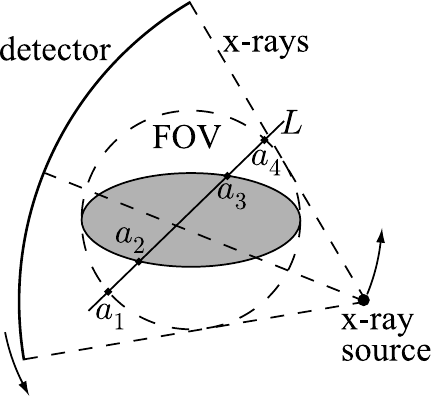, width=6cm}}
}
\caption{The geometry of the fan-beam transform. }
\label{fan-beam}
\end{figure}

See Figure~\ref{fan-beam}, which illustrates the fan-beam transform and equation (\ref{main-id}). In practice, the x-ray source and the detector are located on opposite sides of an object being scanned. The source emits multiple x-ray beams, which pass through the object and are registered by the detector. The source-detector assembly rotates around the object, and this way one collects line integral data for all lines intersecting the circular Field of View, or FOV for short (see the dashed circle in Figure~\ref{fan-beam}). If the object is contained completely inside the FOV, then we know the integrals of $f$ along all lines intersecting the object (i.e., $\ts f$). Pick any line $L$ intersecting $\ts f$. Formula (\ref{main-id}) implies that we can compute the 1D Hilbert transform of $f|_L$ for all points of $L$ inside the FOV, i.e. on the interval $[a_1,a_4]$. Let $[a_2,a_3]$ denote the support of $f|_L$. By assumption, $[a_2,a_3]\subset[a_1,a_4]$. Applying the Finite Hilbert transform inversion formula to the Hilbert data on $[a_1,a_4]$, we can recover $f$ on the line. Repeating this procedure for a collection of lines $L$ that cover the support of $f$, we reconstruct all $f$. 

After the importance of the Gelfand-Graev formula for image reconstruction in CT became clear in the middle 2000s, it led to a number of important advances \cite{ncp-04, dnck, zps-05, yyww-07, yyw-07b, yyw-08, kcnd, cndk-08}. In particular, new tools for investigating image reconstruction from incomplete (or, truncated) tomographic data have been developed. Suppose that instead of reconstructing all of $f$, one is interested in reconstructing only a small subset of $\ts f$, called the Region of Interest (ROI). In this case it would be natural to reduce the x-ray exposure by blocking the x-rays that do not pass through the ROI, i.e. the FOV and ROI will coincide (see Figure~\ref{trunc-setup}, left panel). In the figure the reduced x-ray exposure is illustrated by a smaller detector. Now, the Hilbert transform of $f|_L$ is known only on $[a_2,a_3]$, and $[a_2,a_3]\subset \ts f|_L=[a_1,a_4]$.

\begin{figure}[h]
{\centerline
{\hbox{
{\epsfig{file=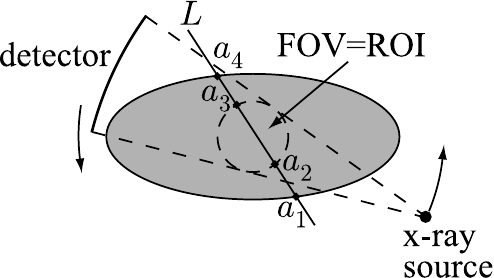, width=6.5cm}}
{\epsfig{file=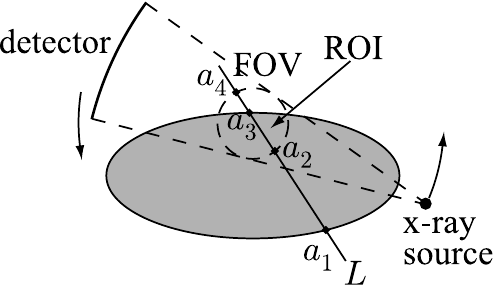, width=6.5cm}}
}}}
\caption{The interior problem (left panel) and the overlap problem (right panel). }
\label{trunc-setup}
\end{figure}

An alternative configuration arises when a part of the FOV is outside of $\ts f$, see Figure~\ref{trunc-setup}, right panel. 
In this case the ROI is a subset of the FOV, and $\ts f|_L=[a_1,a_3]$ overlaps the interval $[a_2,a_4]$ where the Hilbert transform 
is known. 

Our discussion shows that each line $L$, along which the Hilbert transform of $f$ is computed, can be considered separately from 
the others. Thus, we will assume in what follows that $f$ is a function of a one-dimensional argument and is defined on a line. 
A number of interesting results concerning the inversion of the Hilbert transform from incomplete data have been obtained. In 
the case of an overlap, the uniqueness and stability of computing $f$ from input data was established in \cite{dnck}. Uniqueness and stability 
for a different configuration was established in \cite{kcnd}. The main tool for establishing stability in \cite{dnck, kcnd} was the 
Nevanlinna principle. In \cite{kat10c, kat_11} a new approach for the study of the Hilbert transform with incomplete data was initiated. Let $I_1$ and $I_2$ be two intervals on the real line with distinct endpoints. All of the above problems are equivalent to solving the equation $\CH f=g$ for $f$ knowing $g$, where $\CH$ is the FHT that integrates over $I_1$ and the result is evaluated on $I_2$. Different relative positions of $I_1$ and $I_2$ give different problems. Hence we would like to study the operator $\CH: L^2(I_1)\to L^2(I_2)$. 
The $L^2$ spaces are natural here because these are Hilbert spaces (i.e., it makes sense to talk about SVD), square-integrable functions is a very common class that is used in practice, and $\CH: L^2(I_1)\to L^2(I_2)$ is continuous.

It was shown in \cite{kat10c, kat_11} that there exists a second order linear differential 
operator that commutes with $\CH$. This allowed the computation of the singular functions for the operator $\CH$ as solutions of 
certain singular Sturm-Liouville problems. The cases $I_1\cap I_2=\varnothing$ and $I_2\subset I_1$ have been considered in 
\cite{kat10c, kat_11}. The asymptotics of the singular functions and singular values of $\CH$ in these two cases have been 
obtained in \cite{kt12}. In  \cite{aak13} the authors considered the case of an overlap, i.e. $I_1\cap I_2\not=\varnothing$, 
but none of the intervals is inside the other. In particular, it was shown that the spectrum of the operator 
$\CH^*\CH:\,L^2(I_1)\to L^2(I_1)$ is discrete and has two accumulation points: 0 and 1.

From the practical point of view, the interior problem illustrated in Figure~\ref{trunc-setup}, left panel, is most important, so we study it in more detail. Pick six points on the real axis $a_1<a_2<\dots<a_6$ (the reason for having two extra points will become clear shortly). Suppose that  $f$ is supported on $[a_1,a_6]$, and the data $\CH f$ are known on a subinterval $[a_2,a_5]\subset [a_1,a_6]$, i.e. the equation to be solved is
\begin{equation}\label{hilb-data}
(\CH f)(x)=\frac1\pi \int_{a_1}^{a_6} \frac{f(y)}{y-x}dy=
\varphi(x),\ x\in [a_2,a_5].
\end{equation}

\begin{figure}[h]
\centerline{\epsfig{file=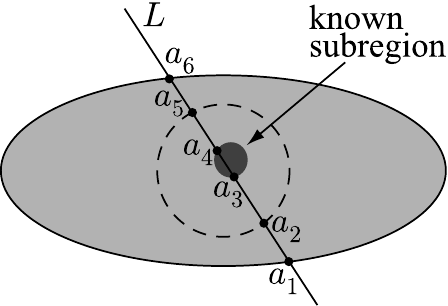, width=6.5cm}}
\caption{The interior problem with prior data. $f$ is assumed to be known inside the ``known subregion''. }
\label{prior-knowl}
\end{figure}

The goal is to recover $f$ on $[a_2,a_5]$, i.e. only where the data are available. Consider $\CH:\, L^2([a_1,a_6])\to L^2([a_2,a_5])$.  As was mentioned in the previous section, this operator has a non-trivial kernel. Hence the unique recovery of $f$ on $[a_2,a_5]$ is 
impossible. Therefore, following the suggestion of several authors \cite{yyww-07, kcnd}, we assume prior knowledge. 
One possibility is to assume that the 2D (or 3D) function $f$ is known on a small open set inside the ROI  
see Figure~\ref{prior-knowl}). Choosing $L$ that intersects the known subregion, we can assume that the 
restriction of $f$ to $L$ is known on $[a_3,a_4]$ (these are the two extra points mentioned above), which leads 
us to the problem at the beginning of Section \ref{math-intro} with $g=2$. In this case the unique recovery of $f$ 
on $[a_2,a_5]$ is theoretically possible \cite{yyww-07, kcnd}.

The assumption about prior knowledge is realistic, because quite frequently there are regions with known values of the attenuation 
coefficient $f$ inside the object being scanned. For instance, in medical applications of CT, $f=0$ for points inside the lungs 
when scanning the chest cavity of a patient. In this paper we consider a more general case $g\ge 2$, which means that there can
 be several intervals on $L$ where $f$ is known. Again, this is a practical assumption, since $L$ can intersect several regions 
inside the FOV with known $f$ (e.g., two lungs).

Theoretical analysis of the interior problem with known subregion was initiated in \cite{cndk-08}, where it was proven that finding $f$ on $(a_2,a_5)$ is stable in the appropriate sense. In this paper we approach the problem from a different  angle. Our main task is to find the asymptotics of the singular values and singular functions of the operators $\CH_e^{-1} $ and $\CH_i$ (see section~\ref{math-intro}) involved in the problem. The exponential decay of the singular values in (\ref{in1}) shows that finding $\varphi$ on the exterior intervals $I_e:=[a_1,a_2]\cup  [a_{2g+1},a_{2g+2}]$ from the data is severely unstable. This, however, does not  contradict the findings in 
\cite{cndk-08}, since the exterior intervals $I_e$ are not covered by the stability estimates in \cite{cndk-08}. 

The most common approach to solving interior problems numerically is iterative. By their nature, iterative algorithms reconstruct $f$ both inside the ROI (i.e., on $[a_2,a_5]$) and outside (i.e., on $I_e$). In some sense, the recovery of $f$ on $I_e$ means that the data $\varphi$ are recovered on $I_e$ as well, which is equivalent to inverting $\CH_e^{-1}$. Hence our findings are relevant for the analysis of stability of such algorithm. Besides, we hope that in the future the results obtained in this paper will lead to novel stability estimates of the recovery of $f$ on $[a_2,a_5]$.

Finally, we note that the approach to the study of the FHT with incomplete data developed in \cite{kat10c, kat_11, kt12} does not apply since now there are six (or more) points $a_j$, $1\leq j \leq 2g+2$, instead of four, and there seems to be no differential operator that commutes with $\CH_e^{-1}$ in this case. Hence a novel approach based on the matrix RHP is developed in this paper.
\section{The integral operator \texorpdfstring{$\hat K$}{hK} and the RHP}
\label{secKernel}

We first reformulate the SVD problem \eqref{svd-def} for the operator $\CH^{-1}_e$.
It is obvious that a triple $(2\l,f,h)$ represents  a singular value and the corresponding singular functions for the operator $\CH^{-1}_e$
if and only if the triple $(\l,\wh f,\wh h)$ represents a singular value and the corresponding singular functions for the operator $H^{-1}_e$, where
\begin{equation}
\begin{split}
(H^{-1}_e\wh h)(y)
&:=\frac{\sqrt{w(y)}}{2\pi i }\int_{I_e} \frac{\wh h(x)}{\sqrt{w(x)}(x-y)}dx
 = \la \wh f(y),\ y\in I_i, \\
(H_i\wh f)(x)&:=\frac1{2\pi i} \frac 1{\sqrt{w(x)}} \int_{I_i} \frac{\wh f(y)\sqrt{w(y)}}{(y-x)}dy=
\la \wh h(x),\ x\in I_e.
\end{split}
\label{svd-def2}
\end{equation}
Here  $\wh h   = \frac {h}{\sqrt{w}}\in L^2(I_e)$, $\wh f = \frac { i f}{\sqrt{w}}\in L^2(I_i)$, and the operators $H^{-1}_e$, $H_i$ act on the corresponding unweighted $L^2$ spaces.  It will
be convenient for us to work with the system \eqref{svd-def2} instead of \eqref{svd-def} in the remaining part of the paper.

The singular values of the system \eqref{svd-def2} coincide with the positive eigenvalues of the
 integral operator $\wh K$ that will be introduced 
in the following Subsection \ref{sec-hatK}. Moreover, the eigenfunctions of $\wh K$ correspond to singular functions of \eqref{svd-def2}.
We also prove there that  $\wh K$ is a self-adjoint Hilbert--Schmidt operator with simple eigenvalues.
In Subsection \ref{sec-resolv} we introduce a matrix RHP for  $\G=\G(z;\l)$, $\l\neq 0$, and express the resolvent operator for $\wh K$
in terms of $\G$. We prove that the solution $\G=\G(z;\l)$ of this RHP exists if and only if $\l$ is not an eigenvalue of $\wh K$.
In Subsection \ref{sect-eigenf} we express the eigenfunctions and the logarithmic derivative of the (regularized) determinant
of $\wh K$ in terms of $\G$. These expressions will be used in Section \ref{secdivisor} to approximate the singular functions
of the system \eqref{svd-def} and to prove the accuracy of the singular values in \eqref{in1}.

\subsection{Definition and properties  of \texorpdfstring{$\wh K$}{rhK} }\label{sec-hatK}

Let us define the integral operator $\wh K:L^2(I) \to L^2(I)$, where $L^2(I)  = L^2(I_e \sqcup I_i) \simeq L^2(I_e) \oplus L^2(I_i)$,
by the requirements
\be
\wh K\big|_{L^2(I_i)}  = H_i\ ,\ \ \ \ \wh K\big|_{L^2(I_e)}  = H_e^{-1}.\label{straight}
\ee
In terms of $\wh K$, the SVD system  \eqref{svd-def2} can be simply written as 
\be\label{svd-K}
 \wh K \phi = \lambda \phi, ~~~{\rm where}~~~~  \phi =\wh f(z)  \chi_i(z)+\wh h(z)\chi_e(z) \in L^2(I).
\ee
Here and henceforth $ \chi_i(z), \chi_e(z)$ denote the  characteristic (indicator) functions of the sets $I_i, I_e$, respectively.
Equation \eqref{svd-K} makes it clear that $(\l,\phi)$ is an eigenvalue/eigenfunction of $\wh K$ if and only if $(\l,\wh f,\wh h)$ satisfies the system \eqref{svd-def2}.
We also point out that, similarly to \eqref{svd-def}, operator $H_i$ is the adjoint of  $H_e^{-1}$.

\begin{theorem}
\label{hatKkernel}
The integral operator $(\hat K \phi)(z) =\int_I K(z,x)\phi(x)dx$ from $L^2(I)$ to $L^2(I)$,  where
\be
\label{kernK}
K(z,x)=
\frac{w^{\frac 1 2}(x) w^{-\frac 1 2}(z)  \chi_e(z)\chi_i(x)  + w^{\frac 1 2}(z) w^{-\frac 1 2}(x)  
\chi_i(z)\chi_e(x)}{2i\pi (x-z)}, 
\ee
is a self-adjoint and a Hilbert--Schmidt operator satisfying \eqref{straight}. Moreover, the 
eigenvalues of $\wh K$ coincide  with 
the singular values of $H_e^{-1}$.
\end{theorem}
\begin{proof} The last statement has been established above.
It is easy to check that $K(z,x)=\overline{K(x,z)}$, so that $\hat K$  is self-adjoint. $\hat K$  is also
a  Hilbert--Schmidt operator because
\be
\int_{I\times I} |K(x,y)|^2 d  x d y =  \frac{1}{2\pi^2}\int_{I_i} d x \int_{I_e} d y \frac { w(x)}{w(y) (x-y)^2} < + \infty.\label{ineq}
\ee
Equations \eqref{straight} follow directly from \eqref{kernK}, \eqref{svd-def2}.
\end{proof}

\br\label{rem-H-HS}
Incidentally, \eqref{ineq} implies  that both $H_i, H_e^{-1}$ are Hilbert--Schmidt.
\er

\bc\label{cor-spectK} The operator $\wh K$ has a real discrete set of eigenvalues $\l_n, n\in\N$, that can accumulate only to $\l=0$.
\ec

As it was mentioned in Section \ref{math-intro}, due to the symmetry $(\l, \wh f, \wh h) \leftrightarrow (-\l, -\wh f, \wh h)$ of the system 
\eqref{svd-def} (and \eqref{svd-def2}), we can consider only nonnegative eigenvalues of $\wh K$. 
Let $\wh L = H_e^{-1} H_i:L^2(I_i) \to L^2(I_i)$. Thus we have immediately the following corollary.
\bc\label{cor-spec-L}
$\l_n$ is an eigenvalue of $\wh K$ if and only if $\l^2_n$ is an eigenvalue of $\wh L$.
Moreover, if $\phi_n = \wh f_n \chi_{i} + \wh h_n \chi_{e}$ is  an eigenfunction of $\wh K$ with eigenvalue $\l_n$, then $\wh f_n$ is an  eigenfunction of $\wh L$ corresponding to $\lambda_n^2$.
\ec

Direct calculations show that  $(\wh Lf)(x)=\int_{I_i}L(x,y)f(y)dy$, where the kernel is given by
\be\label{Lxy}
L(x,y) = \frac{\sqrt{w(x) w(y)} }{4\pi^2}\int_{I_e} \frac {d z}{w(z)} \frac 1{(z-x)(z-y)}.
\ee
Aiming now at analyzing the spectrum of $\wh L$ we point out that $\wh L$ is {\em strictly Totally Positive}, according to the definition below.

\bd
An integral operator  $\wh L :L^2(I) \to L^2(I)$ with a continuous kernel $L(x,y)$, where $I\subset \R$ is a finite union of segments, is called strictly totally positive (sTP) if for any $n\in \N$, and for any choices of a pair of ordered  $n$-tuples $x_1<x_2<\dots<x_n,\ y_1<y_2<\dots<y_n$ in $I$ 
\be\label{def-sTP}
\det\le[L(x_\ell, y_k)\ri]_{1\leq\ell,k\leq n} >0.
\ee
\ed

\bl\label{propSTP} 
The operator  $\wh L = H_e^{-1} H_i$ is strictly totally positive.  
\el
\begin{proof}
The proof is a straightforward computation using Andreief's identity that relates determinants of single integrals and multiple integrals of determinants \cite{Andreief} 
\begin{equation}
\begin{split}
\det&\le[L(x_\ell, y_k)\ri]_{1\leq\ell,k\leq n}  \\
&= \prod_{s=1}^n \sqrt{w(x_s) w(y_s)} \det \le[
\int_{I_e} \frac { {\rm d}z}{w(z)} \frac 1{(z-x_\ell)(z-y_k)}
\ri]_{1\leq\ell,k\leq n} \\ 
&\mathop{=}^{ \hbox{\tiny Andreief}}
 \frac {\prod_{s=1}^n \sqrt{w(x_s) w(y_s)}}{n!} \int_{(I_e)^n} \prod_{r=1}^n \frac {{\rm d} z_r}{w(z_r)}
 \det \le[
 \frac 1{z_s-x_\ell}
 \ri]_{\ell,s\leq n}  
 \det \le[
 \frac 1{z_s-y_k}
 \ri]_{s,k\leq n}\\ 
&\mathop{=}^{\text{\tiny Cauchy det}}
\frac {\Delta(x) \Delta(y) \prod_{s=1}^n \sqrt{w(x_s) w(y_s)}}{n!} \int_{(I_e)^n} \prod_{r=1}^n \frac {{\rm d} z_r}{w(z_r)}
\frac{[\Delta(z)]^2}{\prod_{\ell,s} (z_s-x_\ell) \prod _{r,k}(z_r-y_k)} \\
&= \frac {\Delta(x) \Delta(y) \prod_{s=1}^n \sqrt{w(x_s) w(y_s)}}{n!} \int_{(I_e)^n} \prod_{r=1}^n \frac {{\rm d} z_r}{w(z_r)}
\frac{[\Delta(z)]^2}{\prod_{\ell,s} (z_s-x_\ell) (z_s-y_\ell )}.
\end{split}
\end{equation}
Here $\Delta(x) = \prod_{i<j} (x_i-x_j)$ is the Vandermonde determinant.
Now it is apparent that the integrand of the last expression here above is strictly positive for any $x_j, y_j\in I_i$ and $z_r\in I_e$ because the interval $I_i$ is inside the external one.
\end{proof}

\begin{theorem}\label{theo-simple-spec}
The integral operator $\wh L = H_e^{-1} H_i$ has simple, positive eigenvalues.
\end{theorem}

This theorem was, in fact, proved in Kellogg \cite{Kellogg1918} and can also be found in \cite{Pinkus-rev}. The only difference is that \cite{Kellogg1918, Pinkus-rev} consider a kernel on $L^2([0,1])$, whereas we have a union of disjoint intervals. However the arguments used in their proof apply verbatim. We highlight some details in Appendix \ref{PinkusApp}.

\bc \label{simple-K} 
The eigenvalues of the integral operator $\wh K$ are simple. \ec
\begin{proof}
Indeed, $\wh K$ is Hilbert-Schmidt and self-adjoint and, therefore, it has a complete basis of eigenfunctions. If an eigenvalue $\l_n$ of 
 $\wh K$ is not simple, then there are at least two linearly independent eigenfunctions of $\wh K$ corresponding to $\l_n$.
Then, according to \eqref{svd-K}, $\l^2_n$ is not a simple eigenvalue of $\wh L$. The obtained contradiction with Theorem \ref{theo-simple-spec}
completes the proof. 
\end{proof}
\br
\label{signchange}
It is also known   (\cite{Pinkus-rev}) that the $n$-th eigenfunction $\wh f_n$ of $\wh L$ corresponding to the ordered eigenvalues $\l_0^2> \l_1^2>\dots>0$ changes sign $n$ times within $I_i$ (and that the zeroes in the interior of $I_i$ are simple), see Figure \ref{asympteignum}. 
It will be shown in Remark \ref{signchange2} that the approximation of $\wh f_n$ we are going to obtain is 
asymptotically consistent with this property. 
\er

%

\subsection{Resolvent of \texorpdfstring{$\wh K$}{rhK} and the Riemann--Hilbert problem}\label{sec-resolv}
The operator $\wh K$ falls within the class of ``integrable kernels'' \cite{ItsIzerginKorepinSlavnov} (see also the introduction of \cite{BertolaCafasso1}) and it is known that its spectral properties are intimately related to a suitable Riemann--Hilbert problem. 
In particular, the kernel of the resolvent integral operator $\wh R = \wh R(\l)  :L^2(I) \to L^2(I)$, defined by
\be\label{whR}
(\Id + \wh R)(\Id -\frac 1 \l \wh K) = \Id,
\ee
can be expressed through the solution $\G$ of the following RHP (as explained  in Lemma \ref{lem-kernR} below).

\begin{problem}
\label{RHPGamma}
Find a $2\times 2$ matrix-function $\G=\G(z;\la)$, $\l\in\C\setminus\{0\}$, which is  analytic in 
$\overline{\C}\setminus I$, where $I=I_i\cup I_e$, admits non-tangential boundary values   from the upper/lower half-planes that belong to $L^2_{loc}$ in the interior points of $I$, and satisfies 
\bea
\label{rhpGam}
&&\G_+(z;\l)=\G_-(z;\l) \left[\begin{matrix} 1 & 0 \\ \frac{iw}{\la} & 1 \end{matrix}\right], \ \ z\in  I_i;\qquad 
\G_+(z;\l)=\G_-(z;\l) \left[\begin{matrix} 1 & -\frac{i}{\la w} \\ 0 & 1 \end{matrix}\right],\ \ z\in  I_e,
\\
\label{assGam}
&&\G(z;\l)=\1+O(z^{-1})~~~~{\rm as} ~~z\ra\infty, \\
\label{endpcond-out}
&&\G(z;\l)=\le[\mathcal O(1), \mathcal O((z-a_j)^{-\hf})\ri],~~z\ra a_j,~~j=1,2g+2,\\
\label{endpcond-out-inn}
&&\G(z;\l)=\le[\mathcal O(1), \mathcal O(\ln(z-a_j))\ri],~~~~z\ra a_j,~~j=2,2g+1,\\
\label{endpcond-inn}
&&\G(z;\l)=\le[\mathcal O(\ln (z-a_j)),\mathcal O(1)\ri], ~~~~~ z\ra a_j,~~j=3, \dots, 2g.
\eea
Here the endpoint  behavior of $\Gamma$ is described column-wise. We will frequently omit the dependence on $\l$ from notation for  convenience.
\end{problem}
{We will refer informally to the conditions \eqref{rhpGam}, as well as similar conditions to be introduced later, as ``jumps''.}
\br Since the jump matrices in RHP \ref{RHPGamma} are analytic at all points in the interior of $I$, the solution  of the RHP can be easily shown to admit  {\em analytic} boundary values. A similar observation applies to all the subsequent RHPs.
\er

\bp\label{prop-RHPG}
If a  solution to the RHP \ref{RHPGamma} exists,  then it is unique.
\ep 
\begin{proof} 

Let $\G_{1,2}$ be two solutions of the RHP \ref{RHPGamma}. Then it is promptly seen that  $(\G_2\G_1^{-1})(z)$ has no jumps on $I_i, I_e$.
If $\G(z)$ satisfies RHP \ref{RHPGamma}, $\det \G$ is analytic and single-valued in $\bar \C\setminus\{a_1,a_2,\dots,a_{2g+2}\}$
and $\det \G(\infty)=1$. Then, in view of the endpoint behavior \eqref{endpcond-out}-\eqref{endpcond-inn}, we conclude
that  $\det\G\equiv 1$.
Thus, since $\det\G_{1,2}\equiv 1$, the matrix  $\G_2\G_1^{-1}$ has no more than logarithmic
growth at $a_j$, $j=2,\dots ,2g+1$. Since only the first row of $\G_1^{-1}(z)$ has $O((z-a_j)^{-\hf})$
behavior near $a_j,~j=1,2g+2,$ we conclude that $\G_2\G_1^{-1}= O((z-a_j)^{-\hf})$ near $a_j,~j=1,2g+2,$.
Thus,  $\G_2(z)\G_1^{-1}(z)$ is analytic  in $\bar\C$ and approaches $\1$ as $z\ra\infty$.
So, by Liouville's theorem, $\G_2(z)\equiv \G_1(z)$. 
\end{proof}
For convenience of matrix calculations below, throughout the paper we use the Pauli matrices
$$
\s_1= \begin{bmatrix}
0 & 1\\
1&0
\end{bmatrix},~~~
\s_2= \begin{bmatrix}
0 & -i\\
i&0
\end{bmatrix},~~~
\s_3= \begin{bmatrix}
1 & 0\\
0&-1
\end{bmatrix}.
$$
Remarks \ref{SchwartzGamma}, \ref{symmetrylambda} below follow from   Proposition \ref{prop-RHPG}.
\br
\label{SchwartzGamma}
For any $\l \in \C$ the solution $\Gamma(z;\l)$ enjoys the Schwarz symmetry 
\be
\ov {\Gamma(\ov z;\ov\l)} = \Gamma(z;\l)
\ee
because (as the reader may verify) the matrix  $\ov {\Gamma(\ov z;\ov\l)}$ solves the same RHP.
\er
\br
\label{symmetrylambda}
The function $\Gamma(z;\l)$ has the symmetry 
\be
\Gamma(z;-\l) = \s_3 \Gamma(z;\l)\s_3, 
\ee
which follows by noticing that the jumps have the same symmetry. In particular the RHP for $\Gamma(z;\l)$ is solvable if and only if the one for $\Gamma (z;-\l)$ is. This is a reflection of the symmetry of the spectrum of the problem.
\er

\br\label{rem-vec-AB}
If $\Gamma(z;\l)  = [\vec A(z), \vec B(z)]$, where $\vec A,\vec B$ are  the columns of $\G$,	
is the solution of the RHP \ref{RHPGamma}, then $A(z)$ is analytic on $I_e$, and $B(z)$ is analytic on $I_i$.
Moreover, in terms of $\vec A,\vec B$, the jump and normalization conditions 
 \eqref{rhpGam} and \eqref{assGam}, respectively, can be written as a system
\begin{equation}\label{intsysGam}
\begin{split}
\vec A(z) = \le[\begin{array}{c} 1 \\ 0 \end{array} 
\ri] + \frac {1} {2\pi \la} \int_{I_i}\frac{\vec B(\z) w(\z) d\z}{ \z-z},\ z\not\in  I_i,\\
\vec B(z) = \le[\begin{array}{c} 0 \\ 1 \end{array} 
\ri] - \frac {1} {2\pi\la} \int_{I_e}\frac{\vec A(\z) d \z}{ w(\z)(\z-z)},\ z\not\in I_e. 
\end{split}
\end{equation}
On the other hand, the system of integral equations \eqref{intsysGam} is equivalent to the RHP \ref{RHPGamma}. Indeed, the jump \eqref{rhpGam} and normalization condition \eqref{assGam} immediately follow from \eqref{intsysGam}. As solutions of   \eqref{intsysGam}, $\vec A(z)$ is analytic in $\bar\C\setminus I_i$ and 
$\vec B(z)$ is analytic in $\bar\C\setminus I_e$. Then the endpoint behavior \eqref{endpcond-out}-\eqref{endpcond-inn}
follow from properties of Cauchy operators, see \cite{Gakhov}, Section 8.3.
\er

Lemma \ref{lem-kernR} below is a particular case of the resolvent formula derived in \cite{ItsIzerginKorepinSlavnov}.
In the interest of self-containedness, we present it with a proof. 
\bl\label{lem-kernR}
If $\l$ is such that the solution $\G(z;\l)$ of the RHP \ref{RHPGamma} exists, then the kernel $R$ of the resolvent $\wh R$ defined by \eqref{whR} is given by 
\bea
\label{resolvent}
R(z,x;\lambda) =   
\frac{
\vec g^t(x) \Gamma^{-1}(x;\lambda) \Gamma(z;\lambda)\vec f(z)} {2i\pi \l  (z-x)},\\
\vec f(z):= \le[
\begin{matrix}
{ \frac{i \chi_e(z)}{\sqrt{ w(z)}}} \\
\sqrt{ w(z)} \chi_i(z)
 \end{matrix}\ri], \ 
 \vec g(x):= \le[
 \begin{matrix}
-i\sqrt{ w(x)} \chi_i(x)\\
\frac{\chi_e(x)} {\sqrt{ w(x)}} 
\end{matrix}
\ri],
\eea
where $\vec g^t$ denotes the transposition of $\vec g$.
\el
\begin{proof}
The jumps of the RHP \ref{RHPGamma} can be written in the form
\be
\Gamma_+(z;\l) = \Gamma_-(z;\l) \le(\1 - \frac 1 \lambda \vec f (z) \vec g^t(z)\ri),\  z\in I\ ,
\label{RHPIIKS}
\ee
and equations  \eqref{intsysGam}  can be written compactly as 
\be
\Gamma(z;\l)  = \1  - \int_{I} \frac  {  \Gamma_- (\z;\l)\vec f(\z) \vec g^t(\z)d \z}{2i\pi \l (\z-z)}. 
\label{integ1}
\ee
Note that 
\be\label{IIKs}
K(z,x)=\frac{\vec f^t(z)\vec g(x)}{2\pi i(z-x)}~~~{\rm and}~~~\vec f^t(z)\vec g(z) \equiv 0,\  z\in I.
\ee
The latter equation implies that the boundary value in the integrand of \eqref{integ1} is irrelevant because  
$\Gamma_+(\z;\l)\vec f(\z) \vec g^t(\z) =\Gamma_-(\z;\l) \le(\1 - \frac {\vec f(\z)\vec g^t(\z)}\l\ri) \vec f(\z) \vec g^t(\z) = \Gamma_- (\z;\l)\vec f(\z) \vec g^t(\z) $.

Equation \eqref{whR} can be written as 
\be\label{whR1}
 \wh K  \wh R = \l \wh R -  \wh K.
\ee
 To complete the proof, it is sufficient to show that
the kernel of  $ \wh K  \wh R $ is equal to $\l R(x,y;\l)  - K(x,y)$, where $K,R$ are given by \eqref{resolvent}, \eqref{kernK} respectively.
Indeed, taking into account \eqref{IIKs} and \eqref{integ1},
we calculate the  kernel of  $ \wh K  \wh R $ as
\bea\label{kern-calc}
&\& \int_{I} d \z \frac{
\vec g^t(y) \Gamma^{-1}(y;\lambda) \Gamma(\z;\lambda)\vec f(\z)} {2i\pi \l (\z-y)} \frac { \vec g^t(\z) \vec f(x)}{2i\pi (x-\z)}
=\\
&\&=
 \int_{I} d \z \frac{
\vec g^t(y) \Gamma^{-1}(y;\lambda) \Gamma(\z;\lambda)\vec f(\z)\vec g^t(\z) \vec f(x)} {(2i\pi)^2 \l (x-y)} \le[\frac { 1}{\z-y} - \frac 1{\z-x}\ri]
\cr
&\&  = \frac{
\vec g^t(y) \Gamma^{-1}(y;\lambda)\le( \Gamma(x;\lambda) - \Gamma(y;\l)\ri) \vec f(x)} {2i\pi (x-y)} = 
\l R(x,y;\l)  - K(x,y).
\nonumber 
\eea
The conditions \eqref{endpcond-out}-\eqref{endpcond-inn} guarantee 
that the integrals above are well defined.
\end{proof}

\bth\label{theo-Gam-K}
 $\l\in\C\setminus \{0\}$ is an eigenvalue of $\wh K$ if and only if 
the RHP \ref{RHPGamma} for $\G(z;\l)$ has no solution. 
\et
\begin{proof}
Suppose $\l$ is an eigenvalue and the RHP \ref{RHPGamma} has a solution. Then, according
to  Lemma \ref{lem-kernR}, $\Id - \frac 1 \l \wh K$ is invertible and this is a contradiction. 
Let us now assume that $\l$ is not an eigenvalue.
Using \eqref{svd-def2}, the system of integral equations  \eqref{intsysGam} for columns of $\G$ can be equivalently written
as 
\be
\vec B(z) = \le[\begin{array}{c}  \frac{-i}{\lambda\sqrt{w}}H^{-1}_e(1/\sqrt{w}) \\ 1 \end{array} 
\ri] +\frac 1{\la^2\sqrt{w}}H^{-1}_e H_i (\sqrt{w} {\vec B}),\ z\in I_i
\label{int-eq2}
\ee 
or
\be
\le(\Id - \frac 1  {\l^2} \frac 1{\sqrt{w}}\circ \wh L \circ \sqrt{w}\ri) \vec B = \le[\begin{array}{c}  \frac{-i}{\lambda\sqrt{w}}H^{-1}_e(1/\sqrt{w}) \\ 1 \end{array} 
\ri], 
\label{int-eq3}
\ee
where  $\frac 1{\sqrt{w}}\circ  \wh L \circ \sqrt{w}$ is a conjugation  of $\wh L$ by the  multiplication operator $\sqrt w$ 
and the kernel of $\wh L$ is given by \eqref{Lxy}. Note that  this multiplication operator and its inverse are   bounded  (on $L^2(I_i)$).
 According to Corollary \ref{cor-spec-L}, equation \eqref{int-eq3} has a solution
$\vec B(z)\in L^2(I_i)$, since $\l$ is not an eigenvalue of $\wh K$. Then $A(z)$, given  by \eqref{intsysGam}, is analytic when  $z\not\in I_i$,
and the second equation in \eqref{intsysGam} holds. Thus, $\G(z)= [A(z),B(z)]$ is analytic when $z\not\in I$ and satisfies the jump conditions \eqref{rhpGam} and the normalization \eqref{assGam}. 
Direct calculations show that  $\G(z)= [A(z),B(z)]$ also satisfies endpoint behavior \eqref{endpcond-out}-\eqref{endpcond-inn}.
Thus, $\G(z)=\G(z;\l)$ satisfies RHP \ref{RHPGamma}.
\end{proof}
\subsection{Eigenfunctions and further properties of \texorpdfstring{$\hat K$}{ehK}}\label{sect-eigenf}
According to the spectral theorem,  the resolvent $ \hat R(\l)$ of $\hat K$ is analytic in $ \C \setminus {\rm spec}(\wh K)$ and
  \be\label{reso-spect}
\Id + \hat R(\l) =  \le (\Id - \frac 1 \l \wh K\ri )^{-1} = \sum_{n} \frac 1{1- \frac {\l_n}{\l}} P_n,
  \ee
where the summation runs over all the eigenvalues of $\hat K$, and $P_n$ denotes the (orthogonal) projector on the $n$-th eigenspace. 
 Note that, according to Corollary \ref{simple-K}, $ \hat R(\l)$ has simple poles at  the eigenvalues $\l_n$.
We restrict  the top equation in \eqref{intsysGam} to $I_e$ and divide it by $\sqrt {w(z)}$,  while we restrict  the bottom one to $I_i$ and multiply by $i\sqrt{w(z)}$ so that, after a rearrangement of terms, we obtain
\bea \label{13213}
&\& \frac {\chi_e}{\sqrt{w(z)}} \vec A(z) + i\sqrt{w(z)} \chi_i \vec B(z) - \\
\nonumber -\frac {\chi_e}{\sqrt{w(z)}} &\& \ \ \frac {1} {2i\pi \la} \int_{I_i}\frac{i\vec B(\z) w(\z) d\z}{ \z-z}  - 
\frac {\chi_i\sqrt{w(z)}} {2i \pi\la} \int_{I_e}\frac{\vec A(\z) d \z}{ w(\z)(\z-z)}
 =  \le[
\frac {\chi_e}{\sqrt{w}} \atop 
i \sqrt{w} \chi_i
\ri]. 
\eea
We can read  \eqref{13213} component-wise for the two functions $\varphi_{1,2}\in L^2(I)$ defined hereafter (restoring the  notation $\Gamma_{ji}$ for the entries of $\Gamma=(\vec A,\vec B)$)
\be\label{phi_j}
\varphi_j(z;\l) =  \frac{\G_{j1}(z;\l)}{\sqrt{w(z)}} \chi_{e}(z) + {i}{\G_{j2}(z;\l)}{\sqrt{w(z)}} \chi_i(z)\in L^2(I_i\sqcup I_e),
~~j=1,2.
\ee
Then, according to Theorem \ref{hatKkernel}, equation \eqref{13213} can be written as
\be\label{inteq}
\le(\Id - \frac {1} \l \wh K\ri) \varphi_1(z;\l) = \frac{\chi_{e}(z)}{\sqrt{w(z)}} \ ,\ \ 
\le(\Id - \frac  {1}  \l \wh K\ri) \varphi_2(z;\l) =i{\chi_{i}(z)}{\sqrt{w(z)}}.
\ee
Clearly,  both functions exist for $\l\not \in {\rm spec}(\wh K)$  and are given by
\be\label{varphi_12}
\varphi_1(z;\l)=(\Id +\wh R) \frac{\chi_{e}(z)}{\sqrt{w(z)}},~~~ \varphi_2(z;\l)=(\Id +\wh R) i{\chi_{i}(z)}{\sqrt{w(z)}}.
\ee
Moreover, they are analytic  functions in $\l \in \C \setminus {\rm spec}(\wh K)$ (with values in $L^2(I)$) and have no more than simple poles at the eigenvalues $\l_n$ because $\wh R(\l)$ is analytic in $\l$ and has simple poles by \eqref{reso-spect}.

Vice versa, if $\varphi_{1,2}$ are solutions of \eqref{inteq} for some fixed $\l$, 
they define $\Gamma_{j1}$ on $I_e$ and $\Gamma_{j2}$ on $I_i$
through \eqref{phi_j}. Inserting these values into the right sides of the corresponding two equations in \eqref{intsysGam},
we can reconstruct  solution $\Gamma(z;\l)$ of the RHP \ref{RHPGamma} for all $z\in\bar \C\setminus I$.
Thus we have proved the following corollary. 
\bc
\label{corequiv}
Let us fix some $\l\in\C\setminus\{0\}$. The solution $\Gamma(z;\l)$ of the RHP \ref{RHPGamma} exists if and only if both equations \eqref{inteq} simultaneously have a solution.
\ec

\bp
\label{exacteigenfunctions}
If $\l_n$ is an eigenvalue of $\wh K$, then
the functions $\phi_{n,j}(z)=\res {\l=\l_n}\frac{\varphi_j(z;\l)}{\l}, j=1,2,$ are proportional to each other and
satisfy $\wh K \phi_{n,j} =\l_n \phi_{n,j}$.
Moreover,
\be\label{round-phi_n}
\phi_{n,j}(z)=\frac{\chi_e(z)}{\sqrt{w(z)}} \res{\l=\l_n} \Gamma_{j1}(z;\l)  \frac {1}{\l} +
{i} \sqrt{w(z)}\chi_i(z)\res{\l=\l_n} \Gamma_{j2}(z;\l) \frac {1}{\l},
\ee
where at least one  of $\phi_{n,j}$ is not identical zero on $I$.
\ep
\begin{proof}
Consider $j=1$, with $j=2$ being treated similarly.
According to \eqref{inteq} and  \eqref{reso-spect}, we have 
\be
\varphi_1(z;\l) = 
\sum_{n} \frac \l{\l-  {\l_n}} P_n(  \frac {\chi_e}{\sqrt{w}} ), ~~~\text{so that}~~~\phi_{n,1}(z)= P_n(  \frac {\chi_e}{\sqrt{w}} ).
\ee
According to Corollary \ref{simple-K}, the eigenspace of $\wh K$ corresponding to $\l_n$ is one-dimensional, so that
$\phi_{n,j}(z)$, $j=1,2$, is either a corresponding eigenfunction or  identical zero. Therefore, $\phi_{n,j}(z)$
are proportional to each other.
 Equation \eqref{round-phi_n} follows from \eqref{phi_j}. 
 To prove the last statement recall that
the poles of $\varphi_j(z;\l)$ at $\l=\l_n$ can be at most simple. If their residue is zero then both
$\varphi_j(z;\l)$ must be analytic at $\l_n$. But then, according to Corollary \ref{corequiv}, there exists solution
$\Gamma(z;\l_n)$ to the RHP \ref{RHPGamma} with $\l=\l_n$. The obtained contradiction with 
Theorem \ref{theo-Gam-K} completes the proof.
\end{proof}
We shall use Proposition \ref{exacteigenfunctions} to extract the approximation of the eigenfunctions from the approximation of $\Gamma$,
obtained in Section \ref{secmodelRHPM}. 

Let us introduce 
\be
\label{t2}
t_2(\l) = 
\prod_{n} \le(1 - \frac {\l_n}{\l}\ri) {\rm e}^{\l_n/\l}
\ee
where the product is taken over all the eigenvalues of $\wh K$. Since  $\wh K$ is a Hilbert--Schmidt operator,
it is straightforward to show that the product is absolutely convergent for any $\l\in\C$  and that  
$t_2(\l)$  vanishes if and only if $\l=\l_n$, i.e., $\Id - \l^{-1}\wh K$  is {\em not} invertible.
In fact, $t_2(\l)$  is known as  (Carleman) regularized  determinant  (see \cite{SimonIdeals}, Ch. 3)
of the  Hilbert--Schmidt operator $\wh K$ that is denoted $t_2(\l) = 
\det_2 \le(\Id - \frac 1\lambda \wh K\ri)$.

Our aim now is to calculate the logarithmic derivative   $\pa_\l\ln t_2(\l)$ in terms of $\G$, which then allows to 
find $n_\mathcal C$ --
the number of eigenvalues  lying within a 
closed contour $\mathcal C$ in the $\l$-plane  -- by $n_{\mathcal C}  = 
\frac 1{2i\pi} \oint_{\mathcal C} \pa_\lambda \ln {t_2}(\lambda)  d \l.$ This expression will allow us 
to localize the eigenvalues of $\wh K$ using the approximation of $\G$, obtained in Section \ref{secmodelRHPM}.
We see by elementary calculus that 
\be
\pa\ln t_2(\l) = \frac 1{\l^2} \sum_{n}\le( \frac {\l_n}{1-\frac {\l_n}\l} - {\l_n} \ri) =\frac 1{\l^2} \mathrm {Tr} (\wh K \wh R), \label{dlndet}
\ee
where ${\rm Tr}\,\wh F= \int_I F(x,x)dx$  denotes the trace of a trace-class integral operator $\wh F: L^2(I)\ra  L^2(I)$ with the kernel $F(x,y)$ (see e.g. \cite{SimonIdeals}).
Indeed, observe that: i) the series in \eqref{dlndet} is absolutely convergent; ii)  
the operator $\wh R\wh K$ is of the trace class since both $\wh K, \wh R$ are Hilbert--Schmidt operators, and;
iii) the series in \eqref{dlndet} is
the sum of all the eigenvalues of the trace class operator
$\l\wh R -\wh K=\wh R\wh K$, see \eqref{whR1}. 
The following proposition expresses $\mathrm {Tr} (\wh K \wh R)$ through the matrix $\G$.
\bp
\label{logdet}
Let  $\wh K$ be the integral operator defined in Theorem \ref{hatKkernel}, $t_2(\lambda)$  be its (regularized) determinant, see \eqref{t2}, and $\G(z;\l)$ satisfies the RHP \ref{RHPGamma}.    
Then
\be\label{dlnt}
\pa_\lambda \ln t_2(\lambda) =- \frac 1 { \l i\pi} \oint_{\hat I_i}\le(\Gamma_{21}\Gamma_{12}' - \Gamma_{11}\Gamma_{22}'\ri) d z,
  \ee
  where $\hat I_i$ is a clockwise loop surrounding $I_i$, $\G_{jk}$'s are the entries of $\G$, and prime denotes the derivative with respect to $z$. 
  \ep
\begin{proof}
Expressing the kernel of $\wh R \wh K$ through $R,K$, see \eqref{resolvent}, \eqref{IIKs}, we calculate the right hand side of 
\eqref{dlndet} as 
\begin{equation}
\begin{split}
\label{trRK}
\mathrm{Tr} \le( \wh K \wh R \ri) = &
\int_I d x \int_{I} d \z \frac{
\vec g^t(x) \Gamma^{-1}(x;\l) \Gamma(\z;\l)\vec f(\z)} {2i\pi \l (\z-x)} \frac { \vec g^t(\z) \vec f(x)}{2i\pi (x-\z)}\\
=&\int_{I_e} \frac{-d x}{2\pi w(x)} \int_{I} d \z 
 \frac{
 (\Gamma^{-1}(x;\l) \Gamma(\z;\l)\vec f(\z)\vec g^t(\z))_{21}} {2i\pi \l (\z-x)^2}\\
  &\qquad +
\int_{I_i} \frac{w(x)d x}{2\pi } \int_{I} d \z 
 \frac{
 \big(\Gamma^{-1}(x;\l) \Gamma(\z;\l)\vec f(\z)\vec g^t(\z)\big)_{12}} {2i\pi \l (\z-x)^2}.
\end{split}
\end{equation}
Note that $\vec f(z)\vec g^t(z)=\le[\begin{array}{cc} 0 &     \frac{i\chi_e(z)}{w(z)} \\ 
 {-i}\chi_i(z){w(z)} & 0
 \end{array} 
\ri]
$
and, therefore, the first integral over $I$ in  \eqref{trRK} can be replaced by the same integral over $I_i$, whereas
 the second integral over $I$  can be replaced by the same integral over $I_e$. Thus, both integrals in \eqref{trRK} are
nonsingular.

Considering the $j,k$ entry of differentiated equation \eqref{integ1}, we obtain 
\be\label{G'}
\G'_{jk}(z;\l)=- \int_{I} d \z \frac{
 \big( \Gamma(\z;\l)\vec f(\z)\vec g^t(\z)\big)_{jk}} {2i\pi \l (\z-z)^2}
\ee
provided that the integral is well defined.
But this is exactly the case in  \eqref{trRK} where the integrals are nonsingular. Thus, \eqref{trRK} and \eqref{G'} imply
\begin{equation}
\begin{split}
   \pa_\l\ln t_2(\l) &=\int_{I_e} \frac{d x}{2\pi\l^2 w(x)}
 \le(
 \Gamma^{-1}(x;\lambda) \Gamma'(x;\l)\ri)_{21} 
 -
\int_{I_i} \frac{w(x)d x}{2\pi\l^2  } 
 \le(
 \Gamma^{-1}(x;\lambda) \Gamma'(x;\lambda)\ri)_{12}
\\
=
  -\frac 1 {\l} \frac 1 {2i\pi}& \int_{I_i} \frac{iw(z) }{\l}\le(\Gamma_{22}\Gamma_{12}' - \Gamma_{12}\Gamma_{22}'\ri) d z -\frac 1 {\l} \frac 1 {2i\pi} \int_{I_e} \frac 1{\l iw(z)} \le(\Gamma_{11}\Gamma_{21}' - \Gamma_{21}\Gamma_{11}'\ri)d z.
\end{split}
\label{step1}
\end{equation} 
Note that we did not specify the boundary value ($+$ or $-$) in the elements of $\Gamma$ above,  because the second  
column is analytic across $I_i$, and the first column is analytic across $I_e$. These two facts follow from the jumps 
\eqref{endpcond-out}-\eqref{endpcond-inn}. 
On the other hand, according to \eqref{rhpGam},  $\frac{iw(z)}\l \Gamma_{j2}(z;\l) = \Gamma_{j1+}(z;\l) - \Gamma_{j1-} (z;\l)$ on $I_i$, 
whereas 
 $\frac1{iw(z)\l} \Gamma_{j1}(z;\l) = \Gamma_{j2+}(z;\l) - \Gamma_{j2-} (z)$ on $I_e$ ($j=1,2$). Thus,  we have
\begin{equation}
\begin{split}
 &\pa_\l\ln t_2(\l) =\\
 &= -\frac 1 {\l} \frac 1 {2i\pi} \int_{I_i}\le(J(\Gamma_{21})\Gamma_{12}' -J( \Gamma_{11})\Gamma_{22}'\ri)d z  -\frac 1 {\l} \frac 1 {2i\pi} \int_{I_e} \le(J(\Gamma_{12})\Gamma_{21}' -J( \Gamma_{22})\Gamma_{11}'\ri)d z \\
&=  \frac i {2\pi \l } \oint_{\hat I_i}\le(\Gamma_{21}\Gamma_{12}' - \Gamma_{11}\Gamma_{22}'\ri)d z  +
 \frac i {2\pi \l} \oint_{\hat I_e} \le(\Gamma_{12}\Gamma_{21}' - \Gamma_{22}\Gamma_{11}'\ri)d z
 \\
&=  
 \frac i {\pi \l } \oint_{\hat I_i}\le(\Gamma_{21}\Gamma_{12}' - \Gamma_{11}\Gamma_{22}'\ri) d z,
\end{split}
\end{equation} 
where $\hat I_e$ is a clockwise loop surrounding $I_e$ and $J(h)$ denotes the jump of $h$ across the contour of integration.
In the last step, we have used integration by parts on the 
first term followed by contour deformation (notice that the 
integrand is $\mathcal O(z^{-2})$ at infinity).
\end{proof}
\br
The use of Proposition \ref{logdet} will be that of detecting the presence of an eigenvalue in much the same way as Evans' functions are used for detecting eigenvalues of differential operators. 
\er

\section{Asymptotic solution for \texorpdfstring{$\G(z;\la)$}{Gza} when \texorpdfstring{$\la$}{zla} is small }\label{sec-solGam}

We now start our analysis of the behavior of the solution $\Gamma(z;\l)$ of the RHP \ref{RHPGamma} as $\l\ra 0$ using the steepest descent method (see \cite{DKMVZ}, \cite{Deift}).

{\bf Notation:} The segments $[a_{2j-1}, a_{2j}],\ j=1,\dots g+1$ shall be called the {\bf main arcs} (or branchcuts) and denoted by $\g_j$; 
the segment $[a_{2g}, a_{2g+1}]$ shall be denoted $c_0$, the segments $[a_{2j}, a_{2j+1}],\ j=1,\dots, g-1$ shall be denoted by $c_j$: these
segments shall be called {\bf complementary arcs} 
{(or   gaps)}. We also denote 
$c_g=(-\infty,a_1]\cup [a_{2g+2}, \infty)$.   All the main and  complementary arcs are left to right oriented (see Figure \ref{homology}), 
{whereas $c_g$ is oriented right to left.}

We will use the hyperelliptic Riemann surface $\Rscr$ defined by 
\be
\label{hypersurf}
R^2 = \prod_{j=1}^{2g+2}(z-a_j),
\ee
and by $R(z)$ we will understand the unique analytic function on $\C \setminus \bigcup_{j=1}^{g+1} [a_{2j-1},a_{2j}]$ that satisfies \eqref{hypersurf} and behaves like $z^{g+1}$ near $z=\infty$. 
Points on the Riemann-surface $\mathcal R$ will be denoted by $p=(z,R)$, and they will be understood as pairs of values (of course, once $z$ is fixed, $R$ has at most two distinct values).

We define the $A$ and $B$ cycles according to the general prescriptions \cite{FarkasKra} that are adapted to our problem (cf. Figure \ref{homology}):
\begin{itemize}
 \item for $k=1,\dots, g-1$ the cycle $A_k$ is $[a_{2k},a_{2k+1}]$ on both sheets and  the corresponding $B_k$ is a clockwise loop around the branchcut $[a_{2k+1},a_{2g}]$;
\item the cycle $A_g$ is the union of $(-\infty,a_1]\cup [a_{2g+2},\infty)$ on both sheets. 
The orientation is shown in Figure \ref{homology}, where the two sheets are two copies of the $\C$ plane with cuts omitted and such that on the top sheet $p= (z,R(z))$ and on the bottom sheet $p=(z,-R(z))$.
 The corresponding cycle $B_g$ is a clockwise loop embracing the branchcut $[a_{2g+1},a_{2g+2}]$;
 \item for convenience we also define $A_0$ to be the cycle $[a_{2g},a_{2g+1}]$.
\end{itemize} 
\begin{figure}
\begin{center}
\resizebox{0.9\textwidth}{!}{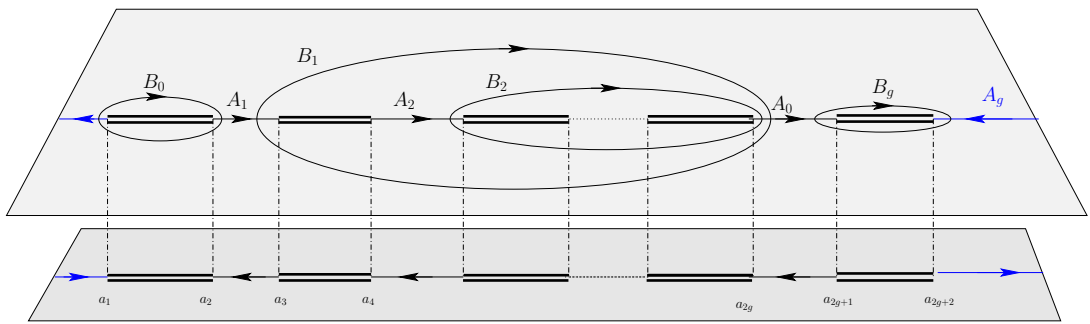}
\end{center}
\vspace{-10pt}
\caption{The choice of contours $A, B$.}
\vspace{-12pt}
\label{homology}
\end{figure}

The normalized differentials of the first kind $\o_jdz$ are defined as the (unique) holomorphic differentials satisfying
\be
 \qquad \oint_{A_j} \omega_k(z)  d z  = \delta_{jk},~~~~j,k=1,\cdots,g.         \label{normalization}
\ee
It is well known that $\omega_j(z) = \frac{ P_j(z) }{R(z)}$, where all
$P_j(z)$ are polynomials of degree $\leq g-1$. They are explicitly computed as follows: 
\be
\label{1stkind}
\vec \omega^t(z)= \le[
\begin{matrix}
\omega_1(z),\dots,
\omega_g(z)
\end{matrix}
\ri] = \frac { \le[\begin{array}{c}
1,\dots, z^{g-1}
\end{array}
\ri]}{R(z)}\mathbb A^{-1}, \ \ \ 
[\mathbb A]_{ji}:= 
\oint_{A_j} \frac {\z^{i-1} d \z}{R(\z)}. 
\ee
It is well known \cite{FarkasKra} that the matrix $\mathbb A$ is nondegenerate.

\bl
\label{lemmazeroes}
Let $\omega_j(z)dz = \frac {P_j(z)}{R(z)}dz$, $j=1,\cdots,g$,  be
 the normalized first-kind differentials. Then  each  polynomial $P_j(z)$ is  {\em real} 
and has exactly one zero in the interior of each $c_k$, where $k=1,\cdots, g$, and $k\neq j$. 
In particular $P_1(z)$ has one zero in each interval $[a_{2m}, a_{2m+1}],\ m=2,\dots, g-1$, and 
one in $[-\infty, a_1]\cup[a_{2g+2},\infty]$, where we understand that if this last zero is at infinity, 
the polynomial $P_1$ is of degree $g-2$. 
\el
\begin{proof}
The reality of $P_j$'s follows because the matrix $\mathbb A$ is real. 
The radical $R(z)$ is real and has constant sign in each finite $c_k$ 
(this sign alternates from one $c_k$ to the next). Thus to have a zero integral in a given $c_k$, the polynomial 
$P_j(z)$ must have at least one root in it. The normalization \eqref{normalization} requires $g-1$ 
vanishing conditions. This forces the roots of $P_j(z)$ to lie in the corresponding complementary arcs. 
The condition $\oint_{A_k}\omega_k dz=1$ fixes the proportionality constant. The statement about $P_1$ follows at once.
\end{proof}

Lemma \ref{lemmazeroes} implies that $P_1(z)$ 
has a zero at some $z_0\in [-\infty,a_1) \cup (a_{2g+1} ,\infty]$ and one zero in each inner gap $[a_{2j}, a_{2j+1}], \ j=2,\dots, g-1$. 
We now define the $\gg$ function as 
\be
\label{defg}
\gg(z) = \frac 12 - 2 \int_{a_1}^z\!\!\!  \omega_1d z,
\ee
{where the path of integration is in $\C \setminus [a_1,a_{2g+2}]$ (note that $\gg(z)$ is  single--valued on this domain since $\omega_1$ has no residue at infinity).} 
\begin{figure}[t]
\vspace{-20pt}
\includegraphics[width=0.5\textwidth]{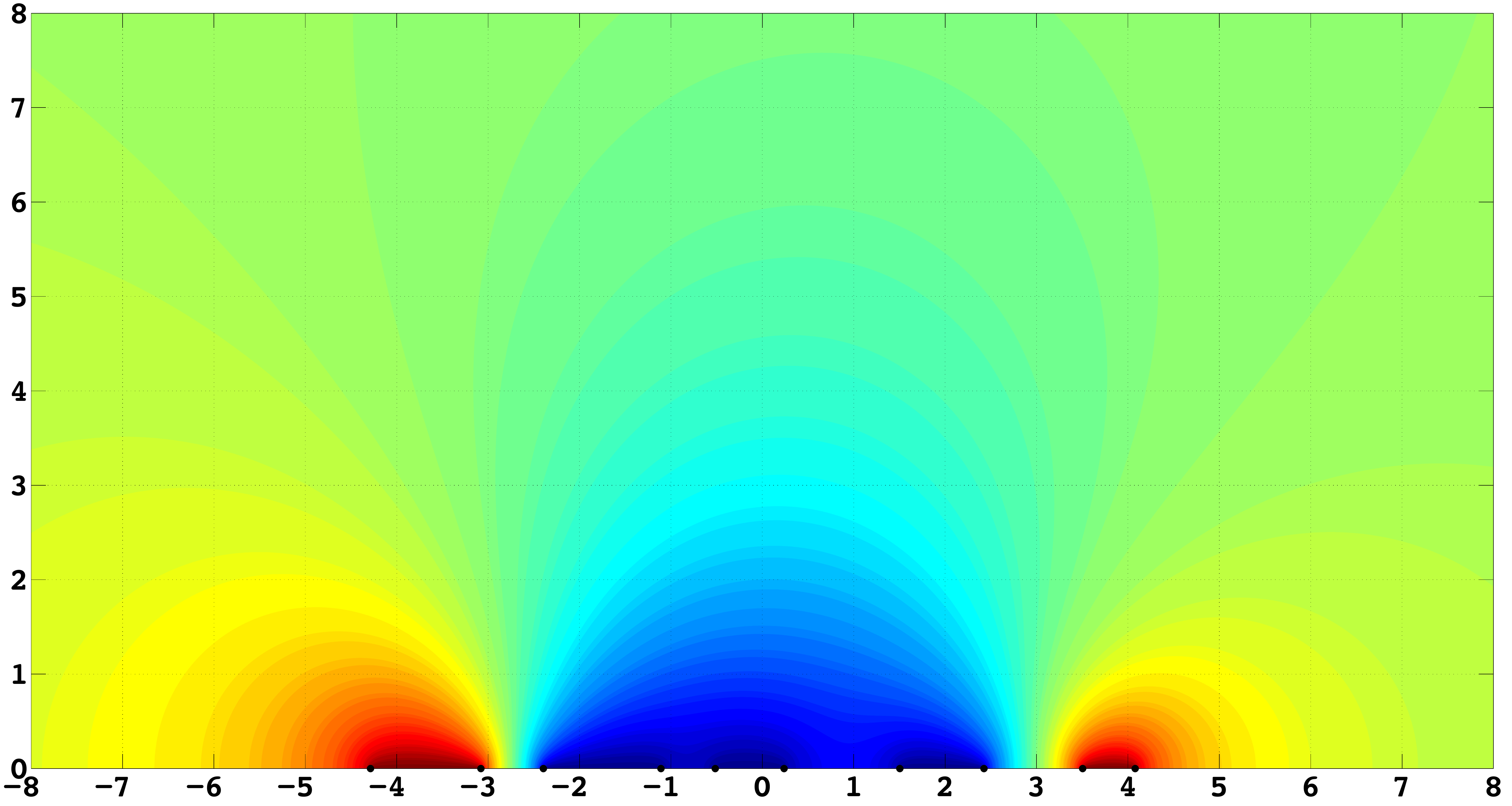}
\includegraphics[width=0.4\textwidth]{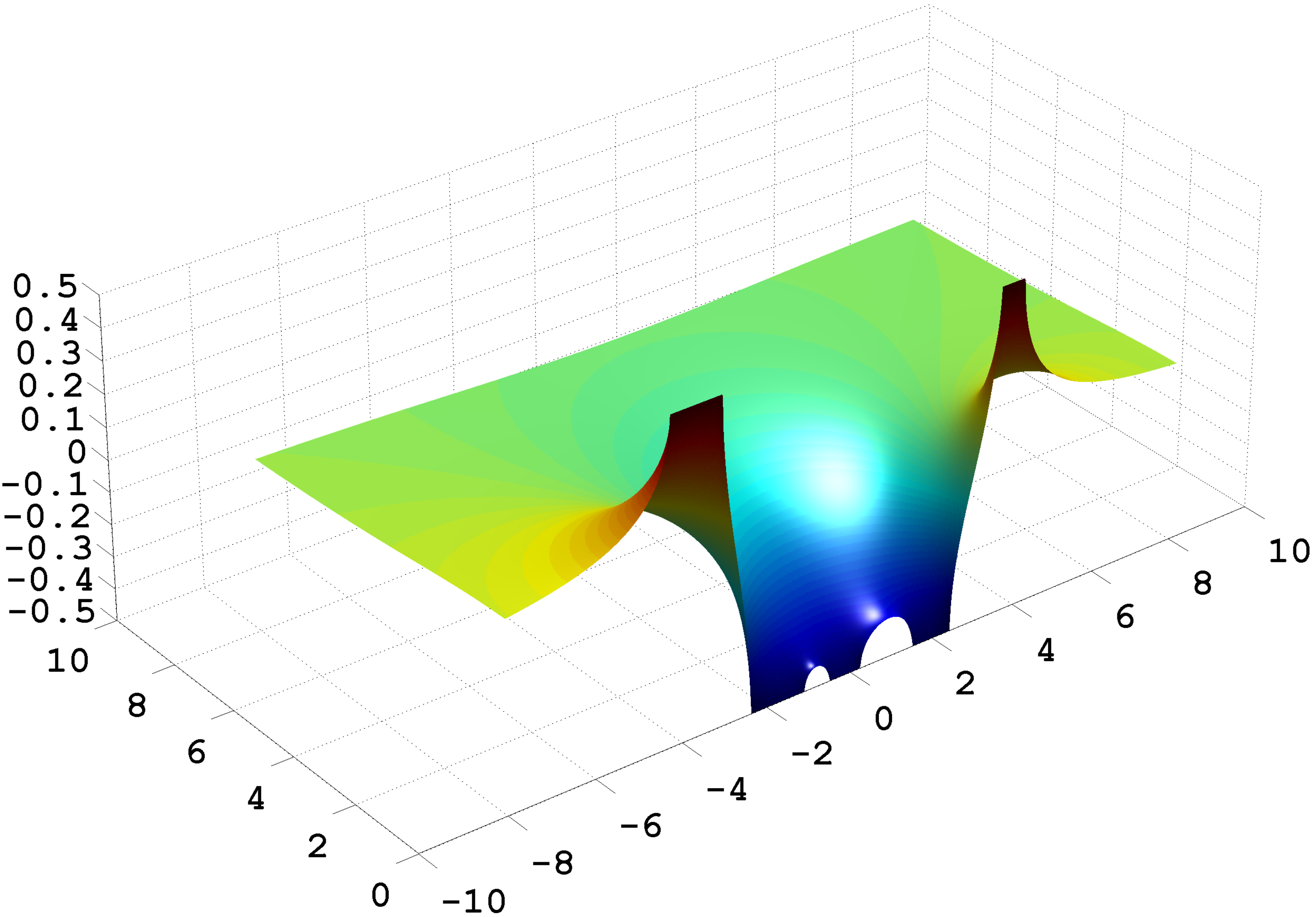}
\caption{The graph of $\Re \gg(z)$, 2D (left) and 3D (right) views for an example of genus $g=4$.
}
\vspace{-10pt}
\label{Regz}
\end{figure}

\bp 
\label{propositiongg}
{\bf (1)} The function \eqref{defg} satisfies  the jump conditions
\be\label{geqm}
\gg_+(z)+\gg_-(z)=-1~~~~{\rm on}~I_i,~~~~~\gg_+(z)+\gg_-(z)=1~~~~{\rm on}~I_e
\ee
\be\label{geqc}
~{\rm and}~~~~~\gg_+(z)-\gg_-(z)=i\O_j~~~~{\rm on}~c_j,~~j=0,1,\cdots,g-1,
\ee 
where 
$\O_0= \frac 4 i \sum_{k=1}^g\int_{a_{2k-1}}^{a_{2k}}\!\!\!\!\! \omega_1dz\in \R$ and
 $\Omega_j = \frac 4 i \sum_{k=1}^j\int_{a_{2k-1}}^{a_{2k}} \!\!\!\!\! \omega_1dz\in \R$.
 \\
{\bf (2)}
The imaginary part of $\gg_+(z)$ satisfies 
\be
\label{signs}
\Im \gg_+' (z)  = -2\Im \omega_1(z)  \le\{
\begin{array}{ll}
>0 & \ \ z\in I_e, \\
<0 &\ \  z\in I_i. 
\end{array}
\ri.
\ee
{\bf (3)} Let $N_i$ denote a small neighborhood of $I_i$, so that $N_i\cap I_e=\varnothing$. Similarly, let
 $N_e$ denote a small neighborhood of $I_e$, so that $N_e\cap I_i=\varnothing$ ($N_e$ consists of two disjoint
regions around $\g_1$ and $\g_{g+1}$, respectively). Then
\be
\label{signdistro}
\Re (2\gg(z) + 1) > 0, \ \ \ z\in N_i\setminus I_i,\qquad 
\Re (2\gg(z) - 1) <  0, \ \ \ z\in N_e\setminus I_e.
\ee
{\bf (4)} The function $\gg(z)$ is analytic in $\bar\C\setminus[a_1,a_{2g+2}]$.
\ep
\begin{proof}
{\bf (1)} This follows from the fact that $R_+ = -R_-$ on the corresponding sides of each main arc $\g_j$, $j=1,\dots, g+1$, and the 
normalization \eqref{normalization}. In particular, if $z\in I_e$, then $\int_{a_1}^z\o_1=0$, so the second jump condition
\eqref{geqm} holds. If  $z\in I_i$, then $\int_{a_1}^z\o_1=1$ and the first jump condition
\eqref{geqm} holds. Finally,  \eqref{geqc} holds  if $z\in c_j{=[a_{2j}, a_{2j+1}]}$, 
$j=1,\dots,g-1$, or $z\in c_0{=[a_{2g}, a_{2g+1}]}$,
see Figure \ref{homology}.

{\bf (2)} Notice that $\arg R(z)=0$ when $z>a_{2g+2} $ and $\arg R_+(z)$ increases by $\frac {i\pi}2$ each time we pass 
a point $a_j$, $j=1,\cdots,2g+2$
from right to left. 
Then $\arg R_+(z)=\frac {i\pi}2$ when $z\in \g_{g+1}$ and $\arg R_+(z)=-\frac {i\pi}2$ when $z\in I_i$.
According to Lemma \ref{lemmazeroes}, the polynomial $P_1(z)$, where
$\o_1=\frac{P_1(z)dz}{R(z)}$, has one zero in each $c_j$, $j=2,\cdots, g-1$. Thus the sign of 
$\Im \o_1$ is the same on  all the main arcs $\g_k \subset I_i$. The sign of $\Im \o_1$ on the main arcs $\g_1$ and $\g_{g+1}$ (they form  $I_e$) is also the same, but opposite
to the sign on $I_i$. Since $\o_1$ is positive on $c_1$, we conclude that  $\Im \o_1>0$ on $I_i$ and $\Im \o_1<0$ on $I_e$. This implies  \eqref{signs}. 

{\bf (3)}  According to (1),
$\Re (2\gg(z)+1) \equiv 0$ on $I_i$ and $\Re (2\gg(z)-1)\equiv 0$ on $I_e$. The imaginary part of $\gg_+$ is {\em increasing} on $I_e$. 
Hence, by Cauchy-Riemann equations, the real part is {\em decreasing} in the normal direction and thus the claim. 
A similar argument works  for $I_i$.
 
{\bf (4)} This follows from \eqref{defg} and the integrability of $\omega_1$ as $z\ra\infty$. 
\end{proof} 

 We will also need an additional auxiliary function $d(z)$. To this end we first introduce the  matrix
\bea\label{T-n}
 T= \mathbb A L=2\le[ 
\begin{array}  {c c c c} 
\int_{c_1}\frac{d\z}{R(\z)} &\cdots & \int_{c_{g-1}}\frac{ d\z}{R(\z)}& \int_{c_{0}}\frac{ d\z}{R(\z)} \\
\vdots &\vdots &\vdots & \vdots \\
\int_{c_1}\frac{\z^{g-1} d\z}{R(\z)} & \cdots&\int_{c_{g-1}}\frac{\z^{g-1} d\z}{R(\z)}&\int_{c_{0}}\frac{\z^{g-1} d\z}{R(\z)} 
\end{array}
\ri],\\
  L=
\le[
\begin{array}{ccccc}
1 & 0 & \dots&0 & -1\\
0 & 1& 0\dots &0& -1\\
& &\ddots &\\
&&\dots &1&-1\\
0&0&\dots &0&1
\end{array}
\ri]\nonumber
\eea
and $\mathbb A$ is defined in \eqref{1stkind}.
Then we define 
\be\label{dz}
d(z)=\frac{R(z)}{2\pi i}\le(-\sum_{j=1}^{g+1}\int_{\g_j}\frac{\ln w(\z)d\z}{(\z-z)R_+(\z)}+
\sum_{j=0}^{g-1}\int_{c_j}\frac{
{i\d_j}
d\z}{(\z-z)R(\z)}\ri),
\ee
where the vector $\vec\d=[\d_1,\dots,\d_{g-1},\d_0]^t$  is given by 
\be\label{delta}
\vec{\d}=2\pi T^{-1}\left\{2\int_{a_1}^\infty - \int_{a_1}^{a_{2g+2}}\right\}\le[ 
\begin{array}  {c c  } 
\frac{d\z}{R_+(\z)}, \frac{\z d\z}{R_+(\z)} ,
\dots,
\frac{\z^{g-1} d\z}{R_+(\z)} 
\end{array}
\ri]^t.
\ee

\bp
The function $d(z)$given by \eqref{dz}
 is analytic on $\bar \C \setminus [a_1,a_{2g+2}]$ (in particular, analytic  at infinity) and satisfies the jump conditions 
\be
d_+  +d_- =- \ln w~~{\rm on~} \g_j,~ j=1,\dots g+1,~~~\qquad
 d_+  -d_- =i\d_j~~{\rm on~} c_j,~ j=0,\dots, g-1.
\label{propertyDelta}
\ee
\ep
\begin{proof}
Given the jump conditions  \eqref{propertyDelta}, expression \eqref{dz} follows from the Sokhotski--Plemelj formula. However,
for an arbitrary $\vec\d$, $d(z)=O(z^{g})$ as $z\ra\infty$. Expanding $\frac{1}{\z-z}$ for large $z$ (and bounded $\z$),
we obtain the system
\bea\label{del_j}
\hspace{5mm}i\sum_{j=0}^{g-1} \d_j\int_{c_j}\frac{\z^md\z}{R(\z)}=\sum_{j=1}^{g+1}\int_{\g_j}\frac{\z^m\ln w(\z)d\z}{R_+(\z)}=
i\pi\left\{\int_{a_{2g+2}}^\infty+\int_{a_1}^\infty\right\}
\frac{\z^md\z}{R_+(\z)},\\
\nonumber {\rm where}~~m=0,1,\dots,g-1,
\eea
of $g$ linear equations for $g$ unknowns $\d_j$, $j=0,1,\dots,g-1$, which is solved by  \eqref{delta}.  \end{proof}

Using similar considerations, we calculate

\be\label{jump-constants-n}
\vec{\O}=-4iT^{-1}\sum_{l=2}^{g}\int_{\g_l}\le[ 
\begin{array}  {c c  } 
\frac{d\z}{R_+(\z)},
 \frac{\z d\z}{R_+(\z)} ,
\dots,
\frac{\z^{g-1} d\z}{R_+(\z)} 
\end{array}
\ri]^t,~~~~~
\ee
where $\vec{\O}=(\O_1,\dots,\O_{g-1},\O_0)^t$ and $\Omega_j$'s are the constants in \eqref{geqc}.

\bp\label{prop-Del}
The function ${\rm e}^{d(z)}$ has the behavior ${\rm e}^{d(z)}  =C (z-a_j)^{-\frac  1 4} (1 + o(1))$  as $z\to a_j$,  where $C\neq 0$,
near the branchpoints $a_j$, $j=1,2g+2$, and 
is bounded at the remaining branchpoints $a_j$, $j=2,\cdots,2g+1$.
 As a consequence, the functions $\rho(z):= i w(z) {\rm e}^{2d(z)}$ and $\frac 1 {\rho(z)}$ are bounded near $a_1, a_{2g+2}$ 
(as well as  near all the other branchpoints). Here $w(z)$ (cf. \eqref{hilb-inv}) is understood as analytic on $\C \setminus (-\infty,a_1]\cup [a_{2g+2},\infty)$, and thus $\rho$ is analytic on $\C \setminus \R$.
 \ep
\begin{proof} According to \eqref{dz}, $d(a_j)=0$ for $j=2,\cdots,2g+1$. So, we consider the behavior of $d(z)$ at $z=a_1$,
where $\ln w(\z)$ has a logarithmic singularity. According to \cite{Gakhov}, Section 8.6, the behavior of $d(z)$ in a small
neighborhood of $a_1$ and away from the cut $\g_1$ is given by
\be\label{Gakh-log}
d(z)=-(z-a_1)^\hf\frac{\ln(z-a_1)-i\pi}{4(z-a_1)^\hf} (1+o(1))=\left(-\frac{\ln(z-a_1)}{4}+\frac{i\pi}{4}\right)(1+o(1)),
\ee
which implies the required behavior at $z=a_1$. Similar arguments work at $z=a_{2g+2}$. 
\end{proof}

\subsection{Transformation of the RHP}
Let $\k=-\ln \la$.  Then $\k>0$ when $\la\in(0,1)$ and $\k\ra\infty$ as $\l\ra 0$.
In this subsection we very briefly describe the nonlinear steepest descent method 
of Deift and Zhou, which allows to reduce the original RHP \ref{RHPGamma} to its leading 
order term, known as the  model RHP (see RHP \ref{modelRHP}), in the limit $\k\ra +\infty$.
The $g$--function $\gg(z)$ and, to a lesser extent, $d(z)$ are important parts of this reduction.  
A detailed description of the nonlinear steepest descent method can be found, for example,  
in \cite{Deift}, \cite{DeiftZhou}.
The first substitution
\begin{equation}\label{y-def}
Y(z;\k)=e^{-(\varkappa \gg_\infty+d_\infty)\sigma_3}\Gamma(z;e^{-\k}) e^{(\varkappa \gg(z)+d(z))\sigma_3},
\end{equation}
where $ d(\infty) = d_\infty\in \C$ and $\gg(\infty)= \gg_\infty\in\C$, reduces the RHP \ref{RHPGamma} to the 
following RHP.
\begin{problem}\label{prob-Y}
Find a $2\times 2$ matrix-function $Y(z;\k)$ with the following properties:
\begin{enumerate}
\item[{\bf (a)}] $Y(z;\k)$ is analytic in $\C\setminus[a_1,a_{2g+2}]$; 
\item[{\bf (b)}] $Y(z;\k)$ satisfies the jump conditions
\bea
\label{jumpY}
\begin{split}Y_+& =Y_- \left[
\begin{matrix}
 e^{(\k \gg +d)_+ -(\k\gg +d)_- } & 0 \\ 
{iw} e^{\k(\mathcal \gg_+ +\gg_- +1) + d_+ +d_-} & e^{-(\k \gg +d)_+ +(\k\gg +d)_- } 
\end{matrix}\right],  \ \ z\in  I_i,
 \\
Y_+&=Y_-\left[\begin{matrix} e^{(\k \gg +d)_+ -(\k\gg +d)_- } & 
 -\frac iw e^{-\k(\mathcal \gg_+ + \gg_- -1) - d_+-d_-  } \\ 0 & e^{-(\k \gg +d)_+ +(\k\gg +d)_- }\end{matrix}\right],\ \ z\in  I_e 
 \\
 & Y_+=Y_- e^{[(\k\gg+d)_+ -(\k\gg+d)_-] \sigma_3},~~ z\in[a_{2j},a_{2j+1}],~ j =1,\dots, g ;
 \end{split}
\eea
\item[{\bf (c)}]
$Y=\1+O(z^{-1})~~~~{\rm as} ~~z\ra\infty,$
 and;
\item[{\bf (d)}] Near the branchpoints (we indicate the behavior for the columns if these have different behaviors)
\be\label{endpcondY}
\begin{split}
Y(z;\k)&=[ \mathcal O(1), \mathcal O(z-a_{j})^{-\frac 1 2}],\ j=1,2g+2;\\
 Y(z;\k)&=O(\ln(z-a_j)), ~~~j=2,\dots, 2g+1.
 \end{split}
\ee
\end{enumerate}
\end{problem}

The jump matrices in \eqref{jumpY} can be calculated directly from \eqref{y-def} and \eqref{rhpGam}.
The branchpoint behavior \eqref{endpcondY} follows from Proposition \ref{prop-Del} and \eqref{endpcond-out}-\eqref{endpcond-inn},
where we take into the account that the first column $\vec A(z)$ of $\G(z;e^{-\k})$ is analytic on $I_e$ (see the proof of Proposition \ref{prop-RHPG}).

We can now rewrite  the jump conditions \eqref{jumpY}  as 
\be
\label{jumpY2}
\begin{split}
Y_+&=Y_- 
\le[\begin{matrix}
1 & \frac { {\rm e}^{-\k (2\gg_{_-} +1)-2d_- }}{iw}\\
0 & 1
\end{matrix}\ri]
\le[\begin{matrix}
0& i\\ 
 i & 0
\end{matrix}\ri]
\le[\begin{matrix}
1 & \frac { {\rm e}^{-\k (2\gg_{_+} +1)-2d_+}}{iw}\\
0 & 1
\end{matrix}\ri]
\text{ on } I_i,\\
Y_+&=Y_-\le[\begin{matrix}
1 & 0\\
 {iw}{ {\rm e}^{ \k (2\gg_{_-} -1)+2d_-}} & 1
\end{matrix}\ri]
\le[\begin{matrix}
0& -i \\ 
-i& 0
\end{matrix}\ri]
\le[\begin{matrix}
1 & 0\\
{iw} { {\rm e}^{ \k (2\gg_{_+} - 1)+2d_+ }} & 1     
\end{matrix}\ri]
\text{ on } I_e,\\
Y_+ &= Y_- {\rm e}^{i(\k \Omega_j + \d_j)\s_3} \text { on } c_j, \ j=0,\dots g-1,
\end{split}
\ee
which can be checked by direct matrix multiplication, together with the fact that $-\ln w - d_+ - d_- \equiv 0$ on the main arcs and $\gg_+ +\gg_- \pm 1\equiv 0$ on $I_i$ and $I_e$ respectively,
see \eqref{geqm} and \eqref{propertyDelta}. 
In both factorizations of the jump matrices in \eqref{jumpY2} (on $I_i$ and $I_e$), 
the left and right (triangular) matrices admit analytic extension on the left/right vicinities of the corresponding
main arcs because they are boundary values of analytic matrices in those vicinities. This suggests 
opening of  the lenses $\pa \mathcal L_{e}^{(\pm)},~ \pa \mathcal L_{i}^{(\pm)}$ around the corresponding main arcs,
see Figure \ref{Lenses} top  panel,
and introduction of the new unknown matrix
\bea
Z(z;\k) =
\le\{\begin{array}{ll}
 Y(z;\k) & \text{ outside the lenses,}  \\
\ds  Y(z;\k) \le[\begin{matrix}
1 & 0\\
\mp {iw}{ {\rm e}^{ \k (2\gg -1)+2d}} & 1
\end{matrix}\ri] & z\in {\mathcal L_e^{(\pm)}},\\
\ds  Y(z;\k) \le[\begin{matrix}
1 & \frac{\mp 1}  {iw}{ {\rm e}^{ -\k (2\gg +1)-2d}} \\
0& 1
\end{matrix}\ri] & z\in  \mathcal L_i^{(\pm)}.
 \end{array}
 \ri.
 \label{422}
\eea
\begin{figure}
\begin{center}
\resizebox{0.7\textwidth}{!}{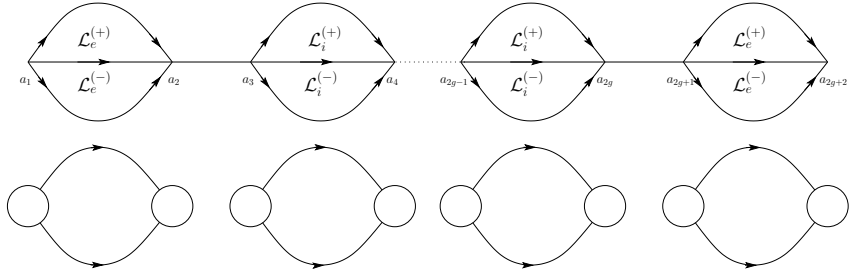}
\end{center}
\vspace{-10pt}
\caption{The regions of the lenses $\mathcal L^{(\pm)}_{i,e}$ (above) and the jumps of the error matrix $\mathcal E$ (below).}
\label{Lenses}
\end{figure}
Consequently, the new matrix $Z$ satisfies the following RHP.
\begin{problem}
\label{ZetaRHP}
Find the matrix $Z$, analytic on the complement of the arcs of Figure \ref{Lenses}, top panel, satisfying  the jump conditions (note also the orientations marked in Figure \ref{Lenses})
\bea
\label{RHPZ}
\\
Z_+(z;\k) = Z_-(z;\k) \le\{
\begin{array}{ll}
{\rm e}^{i(\k \Omega_j + \d_j)\s_3} & z\in  c_{j}, \ j=0,\dots g-1,\\
 \le[\begin{matrix}
1 & 0\\
{iw}{ {\rm e}^{ \k (2\gg -1)+2d}} & 1
\end{matrix}\ri]  & z\in  \pa \mathcal L_{e}^{(\pm)} \setminus \R,\\
\le[\begin{matrix}
1 & \frac{1}  {iw}{ {\rm e}^{ -\k (2\gg +1)-2d}} \\
0& 1
\end{matrix}\ri]  & z\in \pa \mathcal L_{i}^{(\pm)} \setminus \R,\\
i\s_1 & z\in I_i,\\
-i\s_1 & z\in I_e,
\end{array}
\ri.
\nonumber
\eea
normalized by 
\be
Z(z;\k) \to \1\ ,\ \ z\to \infty, \label{ass-Z}
\ee
and with the same endpoint behavior   as $Y$   near the endpoints $a_j$'s, see \eqref{endpcondY}. 
\end{problem}
We shall provide a uniform approximation to the RHP \ref{ZetaRHP} (which is entirely equivalent to the RHP 
\ref{RHPGamma} for $\Gamma$).

Since the real part of $\gg$ satisfies conditions \eqref{signdistro}, the off-diagonal entries in the jump matrices on the boundaries of the lenses of \eqref{RHPZ}  tend exponentially to zero in any $L^p$ norm ($1\leq p<\infty$) on the lenses
while away from the branchpoints $a_j$, $j=1,\cdots,2g+2$  as long as $\Re \k\to +\infty$ and $\Im \k$ remains bounded. 
Thus, neglecting the jumps on the lenses away from the branchpoints leads to exponentially small (in the limit $\k\ra +\infty$) errors. 
The error caused by the jumps on the parts of the lenses that are close to branchpoints require introduction
of local parametrices. This will be considered in Subsection \ref{subsec-localpar}. 
If we completely neglect the jumps on the lenses, we arrive at the following ``model problem'' which captures 
the essence of the approximation.
\begin{problem}[Model problem]
\label{modelRHP}
Find a matrix $\Psi=\Psi(z;\k)$, analytic on $\C\setminus[a_1,a_{2g+2}]$ and satisfying the following conditions:
 \begin{equation}\label{jumpPsi}
\begin{split}
& \Psi_+=\Psi_-  (-1)^{s(j)}(i\s_1)\ ,\ s(j) = \delta_{j,1}  + \delta_{j,g+1},
{\rm ~on~ } \g_j, j=1,\dots g+1;\\
&\Psi_+=\Psi_-e^{i(\k\O_j+\d_j)\sigma_3}{\rm ~on~ } c_j, j=0,\dots, g-1;\\
&\Psi(z) = \mathcal O( |z-a_j|^{-\frac 1 4}) \ , \ \ z\to a_j, \ j=1,\dots, 2g+2;\\
&\Psi(z) = \1 + \mathcal O(z^{-1})\ ,\ \ z\to\infty,~~{\rm and}~~\Psi_\pm(z)\in L^2([a_1,a_{2g+2}]).
\end{split}
\end{equation}
Here $\d_{j,k}$ denotes the Kronecker delta 
{(not to be confused with the vector $\vec \d$ and its components 
$\d_j$ introduced in \eqref{dz},  \eqref{delta}). } 
\end{problem}
It is well known that solution to the RHP \ref{modelRHP}, if it exists, is unique. The proof of uniqueness proceeds analogously to the proof of Proposition \ref{prop-RHPG}. For our purposes, the most important information to extract from the Problem \ref{modelRHP} is {\em for what values of $\k$ it is {\bf not} solvable}. This will be accomplished in section \ref{secmodelRHPM}.
 
\bp[Symmetry]
\label{symmetry}
If $\Psi(z;\k)$ satisfies the  RHP \ref{modelRHP} then $\det \Psi\equiv 1$ and $\wt \Psi\equiv \Psi$,
where $\wt \Psi(z;\k)=\overline{\Psi(\bar z; \ov \k)}$.
In particular, for $\k\in \R$,  $\Psi_{j1+}(z;\k)= \ov \Psi_{j1-}(z;\k)$ for any $z\in I = I_i\cup I_e$.
\ep
The proof of these statements follows the same arguments used in the proof of Proposition \ref{prop-RHPG} and Remarks \ref{SchwartzGamma}, \ref{symmetrylambda}.
\subsection{Local Riemann--Hilbert problems}\label{subsec-localpar}
The uniform approximation approach to $Z(z;\k)$ requires that we analyze neighborhoods of the branchpoints 
(where the jump matrices on the lenses do not approach the identity in $L^\infty$ norm as $\Re (\k) \to +\infty$.)
We define the local coordinates $\xi_j(z)$ near each of the $a_j$'s to be 
\bea
\label{defloc}
\sqrt{\xi_j} = \sqrt{\xi_j(z)}  := \k \int_{a_j}^z \omega_1d\z = C_j \sqrt{z-a_j} (1 + \mathcal O(z-a_j))\ ,\\
\ \ j=1,2, 2g+1,2g+2,\cr
\sqrt{\xi_j} = \sqrt{\xi_j(z)}  := -\k \int_{a_j}^z \omega_1 d\z= C_j \sqrt{z-a_j} (1 + \mathcal O(z-a_j))\ ,\\
\ \ j=3,\dots, 2g.\nonumber
\eea
Inspecting the integrand, we see that the constants $C_j$ are given by $\pm \frac{2 \k P_1(a_j)}{\sqrt{\prod_{k\neq j} (a_k-a_j)}}$ and thus they do not vanish by Lemma \ref{lemmazeroes}. 
We choose the determination of $\xi_j$ always in such a way that the main arc originating at $z=a_j$ is mapped to 
the negative real axis by $\xi_j(z)$. Note that 
each $\xi_j$ is analytic in a full neighborhood of $a_j$ and each is a local conformal mapping.
These local coordinates are clearly related to the function $\gg$ in view of \eqref{defg}, and, thus, 
also to the exponents appearing in the jump conditions of  \eqref{jumpY2} near the $a_j$'s (from the $\pm $ sides of the real axis). Specifically, we have 
\bea
\label{glocal}
&& 
\gg(z)  = \frac 12   -2 \int_{a_j}^z \omega_1d\z \ 
\Rightarrow \ \ \k(2\gg(z)-1)  =
   -4\sqrt{\xi_j},
  \ \ \ \  j=1, 2g+2;
   \cr
&& \gg(z)  = \frac {1}2 \pm\frac i{{2}} \Omega_1    - 2 \int_{a_2}^z \omega_1 d\z
\Rightarrow \ \ \k(2\gg(z)-1)  =
 \pm i\k \Omega_1 -4\sqrt{\xi_2},\cr
&& \gg(z)  = {-}\frac {1}2 \pm \frac i {{2}} \Omega_{\lfloor \frac \ell 2\rfloor}   - 2 \int_{a_{\ell}}^z \omega_1d\z
 \Rightarrow \ \  - \k(2\gg(z){+}1)  =
 \mp i\k\Omega_{\lfloor \frac \ell 2\rfloor}  - 4\sqrt{\xi_{\ell}},\qquad  \ \ \ell =3,4,\dots, 2g, \cr
&& \gg(z)  =\frac {1}2 \pm \frac i {{2}} \Omega_0  - 2 \int_{a_{2g+1}}^z \omega_1d \z
 \Rightarrow \ \ \k(2\gg(z)-1)  =
  \pm i\k \Omega_0 -4 \sqrt{\xi_{2g+1}}.\nonumber
\eea
In terms of these newly defined local coordinates $\xi_j$, the jump conditions in \eqref{RHPZ} become as indicated in Table \ref{tablejumps}.
We shall need to construct local {\em exact} solutions of the jump conditions  of the problem for $Z$ (see \eqref{RHPZ})  in the neighborhood of each branchpoint. The prototypical RHP near those branchpoints is summarized here:

\begin{small}
\begin{table}[h]
\begin{center}
\begin{tabular}{c|c|c|c|c|c}
Endpoint & Lenses & Main & Complementary  & Orientation\\
\hline&&&&\\
$a_1$
&$\ds \le[
\begin{matrix}
1 & 0 \\
{\rho}{\rm e}^{-4\sqrt{\xi_1}} & 1
\end{matrix}
\ri]
$
& 
$-i\s_1$ 
&
$\ds 
\1 $
& Out &\\[4ex]
$a_2$ &
$\ds \le[
\begin{matrix}
1 & 0 \\
{\rho}{\rm e}^{-4\sqrt{\xi_2} \pm i\k \Omega_1} & 1
\end{matrix}
\ri]
$
& 
$\ds 
-i\s_1 $ 
&
$\ds 
{\rm e}^{i(\k \Omega_1 + \d_1)\s_3}
$
& In \\[4ex]
$a_\ell $ &
$\ds \le[
\begin{matrix}
1 & \frac 1{\rho}{\rm e}^{-4\sqrt{\xi_\ell } \mp  i\k \Omega_{\lfloor \frac \ell 2\rfloor}} \\
0& 1
\end{matrix}
\ri]
$
& 
$\ds 
i\s_1 $ 
&
$\ds 
{\rm e}^{i(\k \Omega_{\lfloor \frac \ell 2\rfloor} + \d_{\lfloor \frac \ell 2\rfloor})\s_3}
$
& odd $\ell$ Out; even $\ell$ In \\[4ex]
$a_{2g+1}$ &
$\ds \le[
\begin{matrix}
1 & 0 \\
{\rho}\,{\rm e}^{-4\sqrt{\xi_{2g+1}} \pm i\k \Omega_0} & 1
\end{matrix}
\ri]
$
& 
$\ds 
-i\s_1$ 
&
$\ds 
{\rm e}^{i(\k \Omega_0 + \d_0)\s_3}
$
& Out \\[4ex]
$a_{2g+2}$&
$\ds \le[
\begin{matrix}
1 & 0 \\
\rho {\rm e}^{-4\sqrt{\xi_{2g+2}}} & 1
\end{matrix}
\ri]
$
& 
$\ds 
-i\s_1$ 
&
$\ds 
\1 $
& In \\
\end{tabular}\\[14pt]
\resizebox{.71\textwidth}{!}
{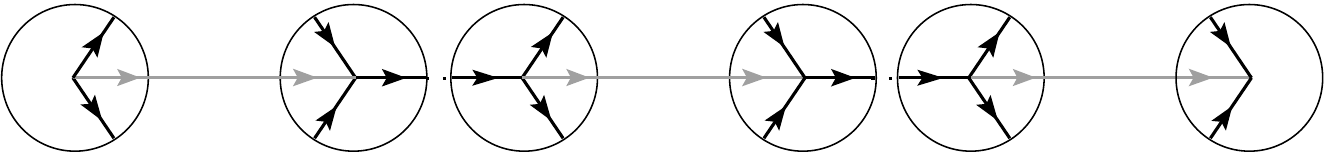}
\end{center}\vspace{10pt} 
\caption{
The jump matrices of $Z$, see \eqref{RHPZ}, near each of the endpoints are given in terms of the local coordinates $\xi_j$, $j=1,\dots,2g+2$. 
Here $\rho(z):= i w(z) {\rm e}^{2 d(z)}$.
The little diagrams show the arcs near the endpoints $a_j$ and indicate the jump on the main arcs (lighter shade).
The orientation ``In/Out" refers to the lens and main arcs, and it means that they are oriented towards or away from $a_j$ (note that the complementary arc has the opposite orientation). The upper/lower choice of signs refers to the jumps on the lenses belonging to the upper or lower half-plane respectively. }
\label{tablejumps}
\end{table}
\end{small}

\begin{problem}[Local Bessel RHP]
\label{BesselRHP}
Let $\vartheta\in (0,\pi)$ be any fixed number. Find a matrix $\mathcal B_\nu (\zeta)$ {($ |\nu|<1$)}   that  is analytic off the rays 
$\R_-,  {\rm e}^{\pm i\vartheta } \R_+$ and satisfies the following conditions (the contours oriented from the origin to infinity):
\bea
\label{BRHP}
&& \mathcal B_{\nu+} (\zeta)=\mathcal B_{\nu-} (\zeta)
\le[
\begin{matrix}
1 &  0 \\
 {\rm e}^{- 4\sqrt{\zeta}  \pm i\pi \nu } & 1
\end{matrix}
\ri],\ \ \z \in {\rm e}^{ \pm i\vartheta }\R_+;\\
&& \mathcal B_{\nu+} (\zeta)=\mathcal B_{\nu-} (\zeta)
\le[\begin{matrix}
0 &1\\
-1 &0
\end{matrix}\ri],\ \  \zeta\in \R_-;
\cr
&&  \mathcal B_\nu(\zeta) = \mathcal O(\zeta^{-\frac {|\nu|}2})  \text{  for  } \nu \neq 0 \text { or }  \mathcal O(\ln \zeta) \text { for  } 
\nu=0 \text { as } \zeta\to 0;
\cr
&& \mathcal B_\nu(\zeta) =  
F(\zeta)
  \le(\1 + \mathcal O(\zeta^{-\frac 1 2 })\ri),\ \ \zeta \to \infty,\cr
 &&   F(\zeta):= (2\pi)^{-\sigma_3/2}
        \zeta^{-\frac{\sigma_3}{4}}\frac{1}{\sqrt 2}
       \le[ \begin{matrix}
            1 & -i \\
            -i & 1
        \end{matrix}\ri].\nonumber 
\eea
\end{problem}
The solution to the RHP \eqref{BesselRHP} was obtained in \cite{VanlessenStrong}.  It is given  
(in a form convenient for our purposes) below.
\begin{equation}
    \label{BRHPsol}
    \mathcal B_\nu(\zeta)=
    \begin{cases}
     \le[   \begin{matrix}
            I_\nu (2\zeta^{\frac{1}{2}}) &
                -\frac{i}{\pi}K_\nu (2\zeta^{\frac{1}{2}}) \\[1ex]
            -2\pi i\zeta^{\frac{1}{2}}I_\nu '(2\zeta^{\frac{1}{2}}) &
                -2\zeta^{\frac{1}{2}}K_\nu '(2\zeta^{\frac{1}{2}})
        \end{matrix} \ri]{\rm e}^{-2\sqrt{\zeta} \s_3}
        , & \arg(\zeta)\in (-\vartheta,\vartheta), \\[5ex]
       \le[ \begin{matrix}
            \frac{1}{2}H_\nu ^{(1)}(2(-\zeta)^{\frac{1}{2}}) &
                -\frac{1}{2}H_\nu ^{(2)}(2(-\zeta)^{\frac{1}{2}})
                \\[1ex]
            -\pi\zeta^{\frac{1}{2}}(H_\nu ^{(1)})'(2(-\zeta)^{\frac{1}{2}}) &
                \pi\zeta^{\frac{1}{2}}(H_\nu ^{(2)})'(2(-\zeta)^{\frac{1}{2}})
        \end{matrix}\ri]{\rm e}^{-2\sqrt{\zeta} \s_3} e^{\frac{1}{2}\nu \pi
        i\sigma_3}, &\arg(\zeta)\in (\vartheta,\pi), \\[5ex]
        \le[\begin{matrix}
            \frac{1}{2}H_\nu ^{(2)}(2(-\zeta)^{\frac{1}{2}}) &
                \frac{1}{2}H_\nu ^{(1)}(2(-\zeta)^{\frac{1}{2}}) \\[1ex]
            \pi\zeta^{\frac{1}{2}}(H_\nu ^{(2)})'(2(-\zeta)^{\frac{1}{2}}) &
                \pi\zeta^{\frac{1}{2}}(H_\nu ^{(1)})'(2(-\zeta)^{\frac{1}{2}})
        \end{matrix}\ri]{\rm e}^{-2\sqrt{\zeta} \s_3}e^{-\frac{1}{2}\nu \pi
        i\sigma_3}, & \arg \zeta\in ( -\pi, -\vartheta),
    \end{cases}
\end{equation}
where $H_\nu^{(1)}$, $H_\nu^{(2)}$ denote  Hankel's functions
(Bessel functions of the third kind) and $I_\nu, K_\nu$ denote the modified Bessel functions (see, for example, \cite{gr}).
\br
The arbitrariness of $\vartheta$ is simply indicative of the fact that the rays can be freely moved within the indicated sector.
For our purposes we can fix any $\vartheta\in (0,\pi)$ by requiring that the boundaries of the lenses are the preimages of straight lines in the respective $\xi_j(z)$ local coordinates, within the disks $\D_j$. For example $\theta = 3\pi/4$ so as to match the shape indicated in Fig. \ref{Lenses}.
\er
\br
Although the shape of the contours supporting the jumps in the inner endpoints in Figure \ref{tablejumps} resemble those of the Airy parametrix \cite{DKMVZ}, the jump matrices are different.
This is also  an expected consequence of the fact that the behavior of the $\gg$-function near the branchpoints is like a square root rather than the power $\frac 32$.
 \er

\br
In our labeling, the constants  $\Omega_0,\d_0$ are the jumps of $\gg(z)$ and $d(z)$, respectively, on the complementary arc $c_0 = [a_{2g}, a_{2g+1}]$. This can be seen from \eqref{geqc}, \eqref{jump-constants-n} for $\gg$, and from \eqref{propertyDelta}, \eqref{delta} -- for $d(z)$. 
\er
\subsection{Final approximation}
\label{finalapprox}
We shall denote the final and uniform approximation to the matrix $Z(z;\k)$ by $\wt Z(z;\k)$.

The accuracy of this approximation to $Z$ (and, thus, ultimately to $\Gamma$) is discussed in 
Section \ref{erroranalysis} after the solvability of the model Problem \ref{modelRHP} has been analyzed in Section \ref{secmodelRHPM}.

Here and henceforth we denote by $\mathbb D_j$  the disks around the branchpoints $a_j$ of the same radius $r$, 
which should be sufficiently small so that the disks do not intersect each other, see Figure \ref{Lenses}, lower
panel. We also use the notation
$\mathbb D_j^\pm = \mathbb D_j \cap \{\pm \Im z>0\}$. Then $\wt Z(z;\k)$ is defined by
\begin{small}
\be\label{approx}
\wt Z(z;\k) := \le\{
\begin{array}{l l} 
\ds \Psi(z)& \!\!\!\! \!\!\!\! \!\!\!\! z \in  \C\setminus\bigcup_{j=1}^{2g+2} \mathbb D_j,\\[2ex]
\ds \Psi(z) \rho^{-\frac {\s_3}2}   {\rm e}^{\pm \frac{i\pi}4\s_3} [F(\xi_1)]^{-1}  
\mathcal B_{\frac 1 2}(\xi_1)  \rho^{\frac{\s_3}2}  {\rm e}^{\mp \frac{i\pi}4 \s_3} & z\in  \mathbb D_1^{\pm},
\\[2ex]
\ds \Psi(z)  \rho^{-\frac{\s_3}2}  {\rm e}^{\mp \frac {i\k} 2 \Omega_1 \s_3} {\s_3}[F(\xi_2) ]^{-1} 
\mathcal B_0(\xi_2) {\s_3} {\rm e}^{\pm \frac {i\k}2  \Omega_1\s_3} \rho^{\frac{\s_3}2} & \text { in } \mathbb D_2^{\pm},\\
\\[2ex]
\ds \Psi(z) \rho^{-\frac{\s_3}2} {\rm e}^{\mp \frac {i\k} 2 \Omega_{k-1} \s_3} {\s_1}[F(\xi_{2k-1})]^{-1} 
\mathcal B_0(\xi_{2k-1}) {\s_1} {\rm e}^{\pm\frac {i\k}2 \Omega_{k-1}\s_3} \rho^{\frac{\s_3}2}  & z\in  \mathbb D_{2k-1}^{\pm},\\
\\[2ex]
\ds \Psi(z)   \rho^{-\frac{\s_3}2} {\rm e}^{\mp \frac {i\k} 2 \Omega_{k} \s_3} {\s_3} {\s_1} [F(\xi_{2k}) ]^{-1}
\mathcal B_0(\xi_{2k}){\s_1} {\s_3} {\rm e}^{\pm \frac {i\k}2 \Omega_k\s_3}\rho^{\frac{\s_3}2} & z\in  \mathbb D_{2k}^{\pm},\\
\\[2ex]
\ds \Psi(z) \rho^{-\frac{\s_3}2}   {\rm e}^{\mp \frac {i\k} 2 \Omega_{0} \s_3}[F(\xi_{2g+1}) ] ^{-1}
\mathcal B_0(\xi_{2g+1})  {\rm e}^{\frac {i\k}2 \Omega_0\s_3} \rho^{\frac{\s_3}2}  & z\in  \mathbb D_{2g+1}^{\pm},\\
\\[2ex]
\ds \Psi(z)   {\rm e}^{\pm \frac{i\pi}4 \s_3}\rho^{-\frac {\s_3}2}  {\s_3} [F(\xi_{2g+2}) ] ^{-1}
\mathcal B_{\frac 1 2}(\xi_{2g+2}) {\s_3} \rho^{\frac{\s_3}2}  {\rm e}^{\mp \frac{i\pi}4 \s_3} & z\in  \mathbb D_{2g+2}^{\pm},\\
\end{array}
\ri.
\ee
\end{small}
where $F(\xi)$ was defined in \eqref{BRHP}.
\br
The function $\rho$, which is analytic in $\C \setminus \R$, was introduced in Proposition \ref{prop-Del}, and $w(z) = \sqrt{(z-a_1)(a_{2g+2}-z)}$ is understood here as analytic on $\C \setminus (-\infty,a_1]\cup [a_{2g+2},\infty)$ and positive on $[a_1,a_{2g+2}]$.
Note that  $\rho^{\frac {\s_3}2} = {\rm e}^{\frac {i\pi}4\s_3} w^{\frac {\s_3}2} {\rm e}^{d \s_3}$ and  $\rho_+^{\frac {\s_3}2}\rho_-^{\frac {\s_3}2} = {\rm e}^{\frac{i\pi} 2\s_3}$ on the main arcs, while $\rho_+^{\frac {\s_3}2}  = \rho_-^{\frac {\s_3}2} {\rm e}^{i\d_j \s_3}$ on the complementary arcs $c_{j}$. 
All of these properties follow from \eqref{propertyDelta}.
\er
\br
In checking the properties at $a_1,a_{2g+2}$ it should be reminded that the function $w(z) = \sqrt{(z-a_1)(a_{2g+2}-z)}$ has the cuts extending on $(-\infty, a_1]\cup [a_{2g+2},\infty)$. 
\er
\subsubsection{How the parametrices are constructed.}
Here we  explain the rationale behind formula \eqref{approx}. The expressions in each $\mathbb D_j$ of formula \eqref{approx} are usually called {\em (local) parametrices} and thus we will conform to the accepted convention.

The main driving logic is that the proposed $\wt Z$ must fulfill:
\begin{itemize}
\item inside each disk $\mathbb D_k$ the jump conditions of $\wt Z(z;\k)$ are exactly the same as those satisfied by $Z(z;\k)$ 
(see Table \ref{tablejumps}) ;
\item The jump conditions satisfied by $\wt Z$ across the boundary of each disk are of the form 
$\wt Z_+(z;\k) = \wt Z_-(z;\k) \le(\1 + \mathcal G(z)\ri)$,  $z\in \pa \mathbb D_j$ where $\mathcal G(z): \pa \mathbb D_j \to SL_2(\C)$ is a matrix  that tends to zero uniformly as $\k\to \infty$ at a certain rate (which will turn out to be $\k^{-1}$). 
\end{itemize}
For the benefit of the reader we show how the jump conditions of $\wt Z$ match those of $Z$ in a ``constructive'' way, rather than 
simply checking them one by one {\em post-facto}.

Consider the case of the endpoint $a_1$.
We start from the exact jumps of $Z$ as in Table \ref{tablejumps} and by multiplication on the right by 
appropriate matrices from the 
second line
of \eqref{approx}, we see how to reduce the jump matrices to the form that matches those  of the Problem \ref{BesselRHP}.  
{\bf(i)} Multiply $Z\mapsto Z\, \rho^{-\frac{ \s_3} 2}$ so that  on the lenses they become $
\le[
\begin{matrix}
 1 & 0 \\
 {\rm e}^{-4 \sqrt{\xi_1}} & 1
 \end{matrix}
\ri]$, and on the main arc it becomes $i\s_2$. However this introduces an additional jump on $(-\infty,a_1]\cap \mathbb D_1$ due 
to the jump of $w = \sqrt{(z-a_1) (a_{2g+2}-z)}$ of the form ${\rm e}^{\frac {i\pi}2 \s_3}$. To remove the latter (undesired) jump we 
{\bf (ii)}  multiply  by ${\rm e}^{\pm \frac {i\pi} 4\s_3}$ in the regions $\mathbb D_1^\pm$, respectively. 
This removes the additional jump on $\xi_1>0$ ($z<a_1$),  but  transforms the jump matrices on the lenses  to $
\le[
\begin{matrix}
 1 & 0 \\
 {\rm e}^{-4 \sqrt{\xi_1} \pm i \pi \frac 1 2} & 1
 \end{matrix}
\ri]$, which now matches precisely those of Problem \ref{BesselRHP} with 
$\nu =\frac 1 2$. 
{The jump matrix $i\s_2$ on the main arc does not undergo any change.}
 Reversing the transformations, we can 
state that 
\be
\mathcal B_{\frac 1   2}  (\xi_1) \,\rho^{\frac {\s_3}2} {\rm e}^{\mp \frac {i\pi}4\s_3}
\label{B12}
\ee
has exactly the same jump conditions as $Z(z,\k)$ (and as $\wt Z(z,\k)$)  in the neighborhood $\mathbb D_1$. 
At the same time we may multiply \eqref{B12} on the left by an arbitrary  invertible matrix-function. We use this to our advantage in such a way that on the boundary $\pa \mathbb D_1$ this analytic prefactor matches  the behavior of \eqref{B12}. 
To this end, consider $\Psi(z) \rho^{-\frac {\s_3}2}   {\rm e}^{\pm \frac{i\pi}4\s_3} [F(\xi_1)]^{-1}  $:
direct calculations show that  
it has no jumps near $a_1$ (note that $F(\xi_1) {\rm e}^{\mp \frac{i\pi}4\s_3}\rho^{\frac {\s_3}2} $ is a local solution of the "model" RHP near $a_1$).  Thus, it has at worst a pole at $z=a_1$. However,
simple power counting shows that it may have at most a square root singularity. Thus,
it  is  analytic at $z=a_1$. 
We have established that 
\be
\wt Z(z;\k)=\Psi(z) \rho^{-\frac {\s_3}2}   {\rm e}^{\pm \frac{i\pi}4\s_3} \framebox{ $[F(\xi_1)]^{-1}\mathcal B_{\frac 1   2}  (\xi_1)$}
\rho^{\frac {\s_3}2} {\rm e}^{\mp \frac {i\pi}4\s_3}
\label{B310}
\ee
in $\mathbb D_1$ has exactly the same jump conditions as $Z(z;\k)$ inside $\mathbb D_1$. Now we examine this formula on the boundary of $\mathbb D_1$; 
here $\sqrt{\xi_1} = \k \int_{a_1}^z \omega_1 d\z$, so that  $|\xi_1|= \mathcal O(\k^2)$.
Thus the $[F(\xi_1)]^{-1}\mathcal B_{\frac 1   2}  (\xi_1)$ term in \eqref{B310} (the framed term), due to \eqref{BRHP}, behaves like 
\be
[F(\xi_1)]^{-1}\mathcal B_{\frac 1   2}  (\xi_1) = \1 + \mathcal O\le(\frac 1{\k}\ri)\ ,
\ee
where the $\mathcal O$ term is uniform for $z\in \pa \mathbb D_1$.
The details of the error analysis are deferred to  section \ref{erroranalysis}.

These steps have to be repeated for each of the branchpoints $a_j$.
The overall result is summarized in the following proposition.
\bp
\label{propglue}
{\bf (1)} The matrix  $\wt Z(z;\k)$ defined in \eqref{approx}  has exactly the same jumps as $Z(z;\k)$ within $\mathbb D_j$ and
on all main arcs $\g_j$.\\
{\bf (2)}  There is  a constant $\mathcal N$, independent of $\k$, such that 
\be
\wt Z_+(z;\k) = \wt Z_{-}(z;\k) (\1 + \mathcal O(\k^{-1}))\ ,\ \ \ z\in \pa \mathbb D_j,~~~j=1,\cdots,2g+2.\label{gluing}
\ee
 with $|\mathcal O(\k^{-1})|\leq \mathcal N/|\k|$ for large $|\k|$.\\
{\bf (3)} The estimate above is valid {\bf uniformly} as ${\Re \k \to +\infty}$  and $\Im \k$ remains bounded.
\ep
\begin{proof}
Only the points (2),(3) have to be proven. The jump of $\wt Z$ is expressed in \eqref{B310}, since $\wt Z_- = \Psi$ and 
thus  ($\nu_1 =\nu_{2g+2} =  \frac 1 2, \nu_j =0$ otherwise)
\be
\wt Z_+(z;\k) (\wt Z_{-}(z;\k))^{-1} = \rho^{-\frac {\s_3}2}   {\rm e}^{\pm \frac{i\pi}4\s_3} [ F(\xi_j)]^{-1}\mathcal B_{\nu_j}  (\xi_j)
\rho^{\frac {\s_3}2} {\rm e}^{\mp \frac {i\pi}4\s_3}.\label{57}
\ee
In the right hand side the only dependence on $\k $ is in $\xi_j$, and on the boundary of the disks $\mathbb D_j$ we have $|\xi_j(z)|> C  |\k|^2 $. 
We then use property \eqref{BRHP} of the Bessel parametrix: $ [F(\xi_j)]^{-1}\mathcal B_{\nu_j}  (\xi_j) = \1 + \mathcal O(\xi_j^{-{\hf}})$. So, on the boundary of each disk $\mathbb D_j$, we have  $ [F(\xi_j)]^{-1}\mathcal B_{\nu_j}  (\xi_j) = \1 + \mathcal O(\k^{-1})$.
The last point {\bf (3)} follows from the fact that the only dependence of \eqref{57} on $\k$ is the  factor $\k$ in the definition of $\xi_j$ (cf. \eqref{defloc}).
\end{proof}

\section{(Non)solvability of the model problem}
\label{secmodelRHPM} 
The solution of the  model RHP \ref{modelRHP} was discussed in the 
literature, see, for example, \cite{DKMVZ, Deift:1997p11, Korotkin05}, where problems of this nature are solved in greater generality. 
Nonetheless, in this section we try to give a relatively  brief but to a large degree self-contained exposition of solution of the RHP \ref{modelRHP}, 
using  only some standard  facts from the geometry of compact Riemann surfaces
(\cite{FarkasKra, Faybook}). Some of this information can be found in Appendix \ref{thetaapp}.  We would like to remind the reader that our interest is not just in {\em solving} the RHP, but rather in knowing when it 
is {\em not solvable}. Reference  \cite{Korotkin05} turns out to be especially useful in this respect.

We recall the definition of the cycles $A,B$ (refer to Figure \ref{homology}) and of the normalized first-kind 
differentials $\omega_j$ \eqref{1stkind}. 
The {\bf normalized matrix of $B$-periods} is then 
\be
\tau = [\tau_{ij}]= \le[
\oint_{B_i}\!\!\! \omega_j d\z
\ri]_{i,j=1,g}.
\label{taumatrix}
\ee
\bth[Riemann \cite{FarkasKra}]
\label{Riemann1}
The matrix $\tau$ is {\bf symmetric} and its imaginary part is strictly positive definite.
\et
In our case it is promptly seen from the definitions \eqref{1stkind}, \eqref{taumatrix} that $\tau$ is purely imaginary.
The Abel map (of the first sheet of the Riemann surface) is defined to be 
\be
\label{Abelmap}
\mathfrak u(z) = \int_{a_1}^z \vec \omega(\z)d\z  ,\ \ \ \ z\in \C \setminus [a_1,\infty).
\ee
\br[Abel map on the Riemann surface]
\label{disclaimer}
Here we have opted for the definition (\ref{Abelmap}) which coincides with the one in the literature only on the first sheet.
On occasions we will need the Abel map extended to a dense simply connected domain of the whole Riemann surface (the canonical dissection). When thinking of a point on the two sheets of the canonical dissection (i.e. a pair of values $p = (z,R(z))$) we shall use the symbol $\mathfrak u(p)$. The effect of the exchange of sheets is the change of sign of $\mathfrak u$. 
Also by the symbol $\u(\infty)$ (without any subscript) we  always denote the Abel map of the point at infinity on the first sheet, and if necessity arises we will use $\u(\infty_{1,2})$ to distinguish between the two points at infinity on different sheets.
\er
The aim of this section is to prove the following theorem:
\bth
\label{hiteig}
The model RHP \ref{modelRHP} is {\bf not solvable} if and only if
\be
 \Theta\le ( W - W_0\ri  )=0,
\ee
where $ \Theta$ is the Riemann Theta-function, see Appendix \ref{thetaapp}, and the vectors $W, W_0$ are given by 
 \be
 \label{wukappa} W = W(\k)=\frac {\k}{i\pi} \tau_1 + 2\mathfrak u(\infty)  + \frac {{\bf e}_1}{2},\ \ 
 W_0  = \frac {\tau_1}2 - \frac { {\bf e}_1 + {\bf e}_g}2
  \ee
with $\tau_1$ denoting the first column of the matrix $\tau$,
and  ${\bf e}_j$, $j=1,\dots,g$, being the  vectors of the standard basis in $\C^g$.
\et
Although this theorem can be derived from the results  of \cite{Korotkin05}, we decided to  give an independent  
proof below for the benefit of the  reader  and also because some notation will be needed in the sequel.
\subsection{Proof of Theorem \ref{hiteig}}
The solution shall be written explicitly in terms of Theta-functions. 
Then, {\em by inspection}, we shall see that under special choices of $W$ there is a row-vector solution 
 that tends to zero at infinity.

It is useful to keep in mind that our definition \eqref{Abelmap} of $\mathfrak u(z)$  leads to the following proposition,
which was obtained by direct calculations.
\bp[Properties of the Abel map $\mathfrak u$]\label{prop-jumpu}
The vector-valued function $\mathfrak u: \C\setminus [a_1,\infty)$ defined in \eqref{Abelmap} satisfies 
\bea
\label{jumpu}
&& \mathfrak u(z) _+ =- \mathfrak u(z)_-  + \le\{
\begin{array}{cl}
0  & z\in [a_1,a_2],\\
\sum_{\ell=1}^k{\bf e}_\ell  & z\in [a_{2k+1},a_{2k{+2}}]\ , { 1\leq} k\leq g-1,\\
{\bf e}_g & z\in [a_{2g+1},a_{2g+2}],
\end{array}\ri.\\
&& \mathfrak u(z) _+ =\mathfrak u(z)_-  + \le\{
\begin{array}{cl}
-\tau_k - \tau_g & z\in [a_{2k},a_{2k+1}]\ ,\ \ {1\leq} k\leq g-1,\\
- \tau_g & z\in [a_{2g},a_{2g+1}],\\
0  & z\in (-\infty, a_1]\cup[a_{2g+2}, \infty),
\end{array}
\ri.
\eea
where $\tau_j$ denotes the $j$-th 
column of the matrix $\tau$ in \eqref{taumatrix}.
In particular we have (all integrals taken on the $+$ side):
\bea
\begin{array}{ccc}
\ds \mathfrak u(a_1)=0, &\ds\mathfrak u(a_{2k+1})  = \sum_{\ell=1}^k \frac {{\bf e}_\ell}2  - \frac {\tau_k  + \tau_g}2, 
& \ds \mathfrak u(a_{2g+1}) =\frac {{\bf e}_g- \tau_g}2,\\
\ds\mathfrak u(a_{2k}) =\sum_{\ell=1}^{k-1} \frac {{\bf e}_\ell}2  - \frac {\tau_k + \tau_g}2, 
&\ds\mathfrak u(a_{2g})  =\sum_{\ell=1}^{g-1}\frac{  {\bf e}_\ell}2 - \frac{ \tau_g}{2},  &
\ds  \mathfrak u(a_{2g+2})  =\frac{ {\bf e}_g}{2}. 
\end{array}
\label{abela_j}
\eea
\ep

Recall that $\mathcal R$ is hyperelliptic and, according to  Proposition \ref{prop-K} (see \cite{FarkasKra}, p. 325, formula (1.2.1)), the vector $\mathcal K$ of Riemann constants is given by
$\mathcal K=\sum_{j=1}^{g} \mathfrak u(a_{2j+1})$ 
{(modulo periods)}.
 Thus, according to \eqref{abela_j}, 
\be\label{2K}
2\mathcal K   =  -   g  \tau_g  - \sum_{k=1}^{g-1} \tau_{k}  + {\bf e}_g + \sum_{\ell=1}^{g-1} (g-\ell) {\bf e}_{\ell}.
  \ee

Let us denote by $\Lambda_\tau = \Z^g + \tau \Z^g\subset \C^g$ the {\em lattice of periods}. 
The {\bf Jacobian} is the quotient $\mathbb J_\tau = \C^g\mod \Lambda_\tau$ and it is a compact torus of real 
dimension $2g$ on account of Theorem \ref{Riemann1}.

Let now  $J = \{ 1,5, 7,9,11, \dots, 2g-1\}$ and $J' = \{1,2, 3, \dots 2g+2\}\setminus J$ so that $|J| = g-1$ and $|J|' =g+3$. 
The image of the degree $g-1$ divisor $ \mathcal D_J = a_1 + a_5+ a_7+ \dots + a_{2g-1}$ in the quotient $\mathbb J_\tau$ is
\be
\label{remdivg-1}
 \mathfrak u (\mathcal D_J):=\sum_{j\in J} \mathfrak u(a_j)   =\mathcal K -\mathfrak u(a_3) - 
\mathfrak u(a_{2g+1})= \mathcal K + \frac {\tau_1}2 - \frac {{\bf e}_1 + {\bf e}_g}2.
\ee

\bl
\label{lemma1}
Define the  functions  
\be
F_1^{(\pm)}(z):= \Theta\le(\mathfrak u(z) \mp \mathfrak u(\infty) - W_0\ri),\qquad
F_2^{(\pm)}(z):= \Theta\le(-\mathfrak u(z) \mp  \mathfrak u(\infty) - W_0\ri)
\ee
on $\C \setminus [a_1, a_{2g+2}]$, where $W_0$ is given in \eqref{wukappa}. 
Then
\begin{itemize}
 \item  the vector $W_0$ in $\mathbb J_\tau$ equals
\be\label{W0long}
W_0= \sum_{j\in J} \mathfrak u(a_j) + \mathcal K.
\ee
\item The functions $F_1^{{(+)}}(z), F_2^{{(-)}}(z)$ vanish at $z=\infty$;
\item The functions $F_1^{{(-)}}(z), F_2^{{(+)}}(z)$  do not vanish at $z=\infty$;
\item For each $j\in J$  they both vanish at $z=a_j$ like $\sqrt{z-a_j}$.
\end{itemize}
\el
\begin{proof}
{Formula \eqref{W0long} follows by noticing that $\u(a_3) + \u(a_{2g+1}) = W_0$ (in $\mathbb J_\tau$) and using \eqref{remdivg-1}.}
We consider only the case $F_k^{(+)}$, $k=1,2$ (the case of $F_k^{(-)}$ is completely similar). In the following discussion we  omit the 
super-index $^{(+)}$ for brevity.
First, note that $F_2$ is the analytic continuation of $F_1$ across the cuts   because of the jumps \eqref{jumpu} and
the periodicity properties of the Theta functions \eqref{thetaperiods}.
Next we use the general Theorem \ref{generalTheta} which asserts that there are exactly $g$ zeroes of this extension on the 
whole Riemann surface $\Rscr$.
According to the  definition of $W_0$, we have  
\be
\mathfrak u(\infty)+ W_0 = \mathfrak u(\infty) +\sum_{j\in J} \mathfrak u(a_j) + \mathcal K. \label{613}
\ee
Note that the divisor consisting of the points  $\infty_{1,2}, a_j, j\in J$ ($\infty_{1,2}$ refer to the point at infinity on one or the other sheet, respectively)  used in \eqref{613} on the hyperelliptic Riemann surface  is non-special as recalled immediately after Definition \ref{defspecial}, and hence the $\Theta$ functions in the expressions for $F$ are not identically zero.
Then, by Theorem \ref{generalThetadiv}, the $g$ points $a_j,~j\in J$ and $\infty_1$ 
(the infinity on the main sheet of $\Rscr$)
are the only zeroes of the extension of $F_1$. 
Whence the proof of the second and the fourth points. The third point is also proved because all the zeroes that the extension of $F_1$ on the second sheet can possibly have, have already been accounted for.
Finally, the reason why the vanishing at $z=a_j, \ j\in J$, is square-root like is due to the fact that the local coordinate in the Riemann surface near the branch point is $\sqrt{z-a_j}$. 
\end{proof}
The functions $F_{1,2}^{(\pm)}(z)$ satisfy the jump conditions
\begin{small}
\bea
\label{Fjumps}
&& F_1^{(\pm)}(z)_+ = F_2^{(\pm)}(z)_-,\quad  z\in [a_{2k-1}, a_{2k}], \ k=1,\dots g+1;
\cr
&& F_1^{(\pm)}(z)_+  = {\rm e}^{2i\pi
V_1^{(\pm)} \cdot({\bf e}_{k} + {\bf e}_g) - i\pi({\bf e}_{k} + {\bf e}_g)\cdot \tau \cdot ({\bf e}_{k} + {\bf e}_g)}  F_1^{(\pm)}(z)_-,\quad
 z\in [a_{2k}, a_{2k+1}]\ ,\ \ k =1,\dots g-1;
\cr
&& F_1^{(\pm)}(z)_+  = {\rm e}^{2i\pi V_1^{(\pm)}\cdot {\bf e}_g   -i\pi  {\bf e}_g\tau {\bf e}_g} F_1^{(\pm)}(z)_-,\quad z\in [a_{2g}, a_{2g+1}];
\cr
&& F_2^{(\pm)}(z)_+ = F_2^{(\pm)}(z)_-,\quad z\in  \bigcup[a_{2k-1},a_{2k}];
\cr
&& F_2^{(\pm)}(z)_+  = {\rm e}^{-2i\pi V_2^{(\pm)} \cdot({\bf e}_{k} + {\bf e}_g) - i\pi({\bf e}_{k} + {\bf e}_g)\cdot \tau \cdot ({\bf e}_{k} + {\bf e}_g)}  F_2^{(\pm)}(z)_-,\quad
 z\in [a_{2k}, a_{2k+1}]\ ,\ \ k =1,\dots g-1,
\cr
&& F_2^{(\pm)}(z)_+  = {\rm e}^{-2i\pi V_2^{(\pm)}\cdot {\bf e}_g   -i\pi  {\bf e}_g\tau {\bf e}_g} F_2^{(\pm)}(z)_-,\quad  
z\in [a_{2g}, a_{2g+1}],\nonumber
\eea
\end{small}
where 
\be
V_1^{(\pm)}(z):= \mathfrak u(z) \mp  \mathfrak u(\infty) -  W_0\ ,\ \ \ 
V_2^{(\pm)} (z):= -\mathfrak u(z) \mp \mathfrak u(\infty) - W_0.
\ee
The jumps  \eqref{Fjumps} follow from Propositions 
\ref{prop-jumpu} and \ref{thetaproperties}.

Consider the function
\bea
&h(z):=  \sqrt[4]{\frac {\prod_{j\in J} (z-a_j)}{\prod_{\ell\in J'} (z-a_\ell)}}, \ \ \ z\in \C \setminus [a_1,a_{2g+2}],
\label{spinorh}
\eea
defined so that it is analytic in $\C\setminus [a_1,a_{2g+2}]$ and at infinity behaves like $\frac 1 z$.
Note that the points of the divisor $\mathcal D_J$ have been chosen to coincide with the finite zeroes of $h(z)$. 
Direct calculations show that:
\bea
&\&h_+ =i  h_-\ ,  \ z\in I_i =\bigcup_{k=1}^{g-1}  [a_{2k+1},a_{2k+2}],\cr
&\& h_+ =-h_- \ , \ z\in [a_{2k}, a_{2k+1}]\ , k=2\dots, g,\cr
&\&h_+ =-i  h_-\ ,\  \  z\in I_e = [a_1,a_2]\cup [a_{2g+1},a_{2g+2}],\ \cr
&\& h_+ = h_-  \ ,\ \ \ z\in (\infty, a_1]\cup [a_2,a_3]\cup[a_{2g+2},\infty).\label{jumph}
\eea
The main solution of the problem is provided by the following theorem.
\bth
\label{theoremM}
The  matrix 
\bea
\mathcal M(z):= 
\le[\begin{array}{cc}
\ds  \frac {\Theta( \mathfrak u(z)- \mathfrak u(\infty)  - W_0 + W) h(z)}{\Theta(\mathfrak u(z) - \mathfrak u(\infty) -W_0)} & 
 \ds \frac {\Theta( -\mathfrak u(z)- \mathfrak u(\infty)  -W_0+ W) h(z)}{\Theta(-\mathfrak u(z) - \mathfrak u(\infty)  - W_0)}\\[15pt]
\ds  \frac {\Theta( \mathfrak u(z)+ \mathfrak u(\infty)  -W_0 + W) h(z)}{\Theta(\mathfrak u(z) + \mathfrak u(\infty)  - W_0)} & 
 \ds \frac {\Theta( -\mathfrak u(z) + \mathfrak u(\infty)  - W_0+ W) h(z)}{\Theta(-\mathfrak u(z) +\mathfrak u(\infty) -   W_0)}
 \end{array}\ri],
\eea
where $W\in\C^g$ is an arbitrary vector and
$ z\in \C \setminus [a_1,a_{2g+2}]$, has the following properties:
\begin{itemize}
 \item [{\bf (1)}]  $\mathcal M(z)$ satisfies the jump conditions
\bea
&&\mathcal M(z) _+ = \mathcal M(z)_- \le[
\begin{array}{cc}
0 & i\\
i & 0
\end{array}\ri] \ ,\ \ z \in I_i;\ \    \mathcal M(z) _+ = \mathcal M(z)_- \le[
\begin{array}{cc}
0 & -i\\
-i & 0
\end{array}\ri] \ ,\ \ z \in I_e; \cr
&& \mathcal M(z) _+ = \mathcal M(z)_- 
\exp \le[\le(2i\pi W \cdot({\bf e}_\ell   + {\bf e}_g) + i\pi \sum_{k=2}^{g-1} \delta_{\ell k} \ri)\sigma_3 \ri],
 \cr 
 && \qquad\qquad \ \ z \in [a_{2\ell},a_{2\ell +1}]\ ,\ \ \ell \leq g-1;\cr
&& \mathcal M(z) _+ = \mathcal M(z)_- 
\exp \le[\le(2i\pi W \cdot {\bf e}_g + i\pi\ri)\sigma_3 \ri];\cr
&& \qquad\qquad
 \ \ \ z \in [a_{2g},a_{2g+1}];\label{jumpsM}\nonumber 
\eea
\item [{\bf (2)}] near any branchpoint $z=a_j$ each entry of $\mathcal M(z)$ is bounded by $|z-a_j|^{-\frac 1 4}$;
\item [{\bf (3)}] {$\mathcal M(z)$ has the behavior}
\be
\mathcal M(z) = \Theta\le( W -W_0\ri)
{C_0^{-1}} \s_3
 +  \mathcal O(z^{-1}) ~~~~{\rm as} ~~~z\ra\infty,
\label{diagonal}
\ee
where $C_0\neq 0 $ and 
\be
\label{defC_0}
C_0   = 
[\mathbb A^{-1} \nabla \Theta(W_0)]_g
\ee
with $[N]_g$ denoting the $g$-th component of the vector $N$.
\end{itemize}
\et
\begin{proof}
{\bf (1)} The proof of \eqref{jumpsM} follows from straightforward application of the periodicity properties of $\Theta$, see
Proposition \ref{thetaproperties} and the jump conditions \eqref{jumpu} and \eqref{jumph}.\\
{\bf (2)}
The function $h(z)$ from \eqref{spinorh} behaves like $(z-a_k)^{-\frac 1 4}$ at all branchpoints that are not in $J$. 
 On the other hand, in each entry the denominators vanish at $z=a_j$, $j\in J$, like $\sqrt{z-a_j}$ by Lemma \ref{lemma1}, 
while the function $h(z)$  vanishes like $(z-a_j)^\frac 14$, $j\in J$.
 Thus each entry has the behavior  $(z-a_j)^{-\frac 1 4}$ at {\em all} the branchpoints.\\
 {\bf (3)} The boundedness follows because the denominators of the $(1,1)$ and $(2,2)$ entries vanish like $\frac 1 z$, but so does $h(z)$ at infinity. The off-diagonal entries $(1,2), (2,1)$ instead tend to zero because the denominators do not vanish (Lemma \ref{lemma1}), while $h(z)$ still does. 
 This proves that the leading order term of $\mathcal M$ at infinity is a diagonal matrix.
To find this matrix, we calculate
 \be
 \lim_{z\to\infty} \frac{\ds \Theta\le(\int_\infty^z\hspace{-5pt} \vec \omega\, d\z  -W_0\ri)} {h(z)}=  
\lim_{z\to\infty} z \Theta\le(\int_\infty^z \!\!\!\vec \omega \, d\z  -W_0\ri).
 \ee
 The computation of this last limit is done using l'Hopital's rule and taking into account \eqref{1stkind}
and the parity of $\Theta$:
 \begin{equation}
 \label{4123}
\begin{split}
\lim_{z\to\infty} - z^2 \frac {d}{dz }  \Theta\le(\int_\infty^z\!\!\!\! \vec \omega \,d\z -W_0\ri) &= \lim_{z\to\infty} -z^2\vec \omega^t(z) \nabla \Theta\le(\int_\infty^z \!\!\!\! \vec \omega \,d\z -W_0\ri)\\
&=  -[\mathbb A^{-1} \nabla \Theta(-W_0)]_g=
[\mathbb A^{-1} \nabla \Theta(W_0)]_g.
\end{split}
\end{equation}
Repeating the computation \eqref{4123} for the entry $(2,2)$ of \eqref{diagonal} gives $-[\mathbb A^{-1} \nabla \Theta(W_0)]_g$.
The fact that $C_0\neq 0$ is a consequence of Theorem \ref{generalTheta}. 
Indeed  $\Theta\le(\mathfrak u(z)  -W_0\ri)$ has a {\em simple} zero at $z=\infty$ (i.e. vanishes linearly in $1/z$) because the other 
$g-1$ zeroes are at  $z=a_j,~j\in J$,  as stated in Lemma \ref{lemma1}. 
\end{proof}
We can restate  Theorem \ref{theoremM} as follows.
\bth
\label{theoremPsi}
Let 
\bea 
\Psi(z;W):= 
C_0
\le[\begin{array}{cc}
\ds  \frac { \Theta( \mathfrak u(z)- \mathfrak u(\infty)  - W_0 + W) h(z)}
{\Theta(W-W_0) \Theta(\mathfrak u(z) - \mathfrak u(\infty) - 
W_0)} & 
 \ds \frac { \Theta( -\mathfrak u(z)- \mathfrak u(\infty)  - W_0+ W) h(z)}
 {\Theta(W-W_0)\Theta(-\mathfrak u(z) - \mathfrak u(\infty)  -
W_0)}\cr
\ds  \frac {-\Theta( \mathfrak u(z)+ \mathfrak u(\infty)  -W_0 + W) h(z)}
{\Theta(W-W_0)\Theta(\mathfrak u(z) + \mathfrak u(\infty)  -
W_0)} & 
 \ds \frac {-\Theta( -\mathfrak u(z) + \mathfrak u(\infty)  -W_0+ W) h(z)}
 {\Theta(W-W_0)\Theta(-\mathfrak u(z) +\mathfrak u(\infty) -  W_0)}
 \end{array}\ri],
 \eea
be a matrix function 
in $ \C \setminus [a_1,a_{2g+2}]$, where $W_0$ is given by  \eqref{wukappa}  and
$W\in \C^g$ is an arbitrary vector such that  $\Theta(W-W_0)\neq 0$. Then: \\
{\bf (1)} this matrix has  the same jumps as in  \eqref{jumpsM} and $\det \Psi\equiv 1$.\\
{\bf (2)} $\Psi(z)=O((z-a_j)^{-\frac 1 4})$ near each $a_j$, $j=1,\cdots,2g+2$;\\
{\bf (3)} at infinity the matrix tends to $\1$; \\
{\bf (4)} The constant $C_0\neq 0$  and is independent of $W$;\\
{\bf (5)} The matrix $\Psi(z;W)$ is {\em invariant} under integer shifts of the vector $W$,  
 $ W\mapsto W + \Z^g$ (i.e. it is periodic in each component of $W$ along the real direction.).
\et
Only the last item needs additional verification, but this follows from the fact that all theta functions have said periodicity.

For the rest of the paper, we will use notation $\Psi(z;\k)=\Psi(z;W(\k))$, where
$\Psi(z;W)$ is defined in Theorem \ref{theoremPsi}.

\bc
 \label{corHW}
{The RHP}
with jumps as in \eqref{jumpsM}, the  branchpoint behavior 
$O((z-a_j)^{-\frac 14})$, $j=1,\cdots,2g+2,$ and bounded behavior at infinity
admits a  nontrivial solution vanishing at infinity if and only if  the vector $W\in \C^g$ is such that
 \be\label{no-model-sol}
 \Theta\le( W -W_0 \ri)=0.
 \ee 
 \ec
\begin{proof} The sufficiency is obvious from \eqref{diagonal}. Suppose now $\Theta\le( W -W_0 \ri)\neq 0$. Then,
according to \eqref{diagonal},  $\lim_{z\ra\infty}  \Psi (z)=\1$.
But the solution of this 
normalized RHP is unique by the same arguments that were used in the proof of Proposition \ref{prop-RHPG}. The existence of a solution that vanishes at infinity 
would violate this uniqueness. \end{proof}
\begin{proof}[{Proof of Theorem \ref{hiteig}.}]
 To complete the proof of the theorem, we must find the vector $W$ so that the jump matrices in \eqref{jumpsM}  are the same as those of $\Psi$ in \eqref{jumpPsi}. 
Comparing them, we see that the vector $W$ must satisfy 
\be
2\pi L
\cdot W+\pi \sum_{\ell=2}^{g}{\bf e}_\ell  = \k\vec \Omega + \vec \delta, 
\ee
where $L$, $\vec \delta, \vec \O$ are given by \eqref{T-n}, \eqref{delta} and \eqref{jump-constants-n}, respectively.
Then,  according to \eqref{T-n}, \eqref{jump-constants-n} and \eqref{delta},
\be\label{Om-del}
\vec \Omega = -2i  L^{-1} \tau_1\ ,\ \ \ 
\vec \d = 2\pi L^{-1}  \le(2\mathfrak u(\infty) - \mathfrak u(a_{2g+2}) \ri).
\ee

Thus, in view of \eqref{abela_j},
\be\label{W(kap)}
W = \frac {\k}{i\pi} \tau_1 + 2\mathfrak u(\infty) - \mathfrak u(a_{2g+2})+ \frac 1 2 {\bf e}_1  - \frac 1 2 {\bf e}_g= 
\frac {\k}{i\pi} \tau_1 + 2\mathfrak u(\infty)  + \frac {{\bf e}_1}{2} - {\bf e}_g. 
\ee
The proof of Theorem \ref{hiteig} now follows from Theorem \ref{theoremM} and  Corollary \ref{corHW}.  Note that since $\Theta$ is periodic in $\Z^g$, the last term ${\bf e}_g$ is irrelevant.
\end{proof}
\bd\label{def-eigen}
The values $\k^{exact}_n=-\ln \l_n$, where $n\in N$ and $\l_n>0$,  for which the RHP \eqref{RHPGamma} for $\G(z;\l_n)$ does not have a  solution,
we will call (with a mild abuse of terminology) {\it exact eigenvalues} of the RHP  \eqref{RHPGamma} or simply  exact eigenvalues. The 
values $\k_n$, $n\in N$,  for which the model RHP \eqref{modelRHP}  does not have a  solution,
will be called approximate eigenvalues.
\ed
According to Theorem \ref{theo-Gam-K},   $\l_n ={\rm e}^{-\k_n^{exact}}$   are positive eigenvalues of the compact integral
operator $\wh K$, so that $\k=+\infty$ is the 
 only possible point of accumulation of the  exact eigenvalues $\k^{exact}_n$. 
Since the RHP \ref{modelRHP} ``approximates'' the RHP \eqref{RHPGamma}, one can expect that  the
approximate
eigenvalues $\k_n$ will approximate the exact eigenvalues $\k^{exact}_n$ as $n\ra\infty$.
This question will be explored in Section \ref{erroranalysis} below.
\br
\label{enumeration}
So far we have not mentioned any particular way of enumerating the 
approximate eigenvalues $\k_n$. Because of the expected approximation of exact eigenvalues,
we shall assume that this enumeration is chosen in such a way that $\k_n$ is ``close" to 
$\k^{exact}_n$ for  sufficiently large $n\in\N$.
\er
\section{Error estimates}
\label{erroranalysis}
In order to estimate the  difference between the exact solution  $Z(z;\k)$ of the RHP \ref{ZetaRHP}
and its approximation $\wt Z(z;\k)$, we introduce   the error matrix 
\be
\mathcal E(z;\k):= Z(z;\k) \wt Z^{-1}(z;\k)\ .
\ee
We note that the jumps of $\wt Z$ are the same as those of $Z(z;\k)$, see \eqref{RHPZ}, on the main arcs  and,
within each disk $\mathbb D_j$,
on the  complementary
arcs. Thus the only jump discontinuities of $\mathcal E$ are on the boundaries $\pa \mathbb D_j$ and 
across the boundaries of the lenses (on the lenses for briefness) outside of these disks, see Figure \ref{Lenses}, bottom panel. According to \eqref{RHPZ} and \eqref{approx}, the jump matrices on the lenses are of the form 
\be
\mathcal E^{-1}_- (z;\k)\mathcal E_+ (z;\k) = \Psi(z;\k) \le (\1 + 
\s_\pm {\rm e}^{\mp\k(2g\pm 1)} (\rho)^{\mp 1}
\ri) \Psi^{-1}(z;\k),\label{lens}
\ee
where the two signs refer to the lenses  $\pa\mathcal L^{(\pm)}_i$ and  $\pa\mathcal L^{(\pm)}_e$ around the intervals $I_i$ and $I_e$, 
respectively (see \eqref{RHPZ}).
We can always assume that  $\pa\mathcal L^{(\pm)}_i\subset N_i$ and   $\pa\mathcal L^{(\pm)}_e\subset N_e$,
where the regions $N_i,N_e$ were defined in Proposition \ref{propositiongg}, part (3). Thus, according to 
the sign conditions \eqref{signdistro}, the factor ${\rm e}^{\mp\k(2g\pm 1)}$ in \eqref{lens} is  
exponentially small 
 as $\Re \k\to +\infty$ and $\Im \k$ is bounded (say in $[-\pi/2,\pi/2]$). This exponential decay 
is uniform on the lenses (outside the corresponding disks $\mathbb D_j$).
Since the remaining factors in \eqref{lens} are bounded  on the lenses (outside the disks), 
we see that the jumps of $\mathcal E$ outside the disks tend to $\1$  exponentially fast and uniformly in $z$
as $\Re\k\ra +\infty$.
The jumps of $\mathcal E$ on the boundary of the disks {$\pa \mathbb D_\ell$} , according to \eqref{gluing}, are 
\bea
\mathcal E_+ (z;\k) && = \mathcal E_-(z;\k) \wt Z_-(z;\k)\wt Z^{-1}_+(z;\k)\cr
&\&=
 \mathcal E_-(z;\k)\Psi(z;\k) \rho^{-\frac {\s_3}2}   {\rm e}^{\pm \frac{i\pi}4\s_3}  F^{-1}(\xi_j)\mathcal B_{\nu_j}  (\xi_j)
\rho ^{\frac {\s_3}2} {\rm e}^{\mp \frac {i\pi}4\s_3}\Psi^{-1}(z;\k)  \cr
 &\&  = \mathcal E_-(z;\k)\Psi \rho ^{-\frac {\s_3}2}   (\1 + \mathcal O(\k^{-1}))
\rho ^{\frac {\s_3}2} \Psi^{-1}.
  \label{disk}
\eea
Consider now the expression for $\Psi(z;W)$ in Theorem \ref{theoremPsi}: the reader can verify that if $\|\Im W\|$ is bounded, 
then the supremum of each entry as $z$ ranges on the boundaries $\pa \mathbb D_\ell$ is bounded as follows:
\be
\max_{\ell\leq 2g+2}\sup_{z\in \mathbb D_\ell} |\Psi_{ij}| \leq \frac{N}{|\Theta(W-W_0)|}\ .
\ee
with $N>0$ some constant. Then the jumps on $\pa \mathbb D_\ell$ of $\mathcal E$ in \eqref{disk} are uniformly close to 
the identity jump to within $\mathcal O(\k^{-1})/|\Theta(W(\k)-W_0)|^2$, that is 
\be
\label{65}
\mathcal E_+ (z;\k) =\mathcal E_- (z;\k)  \le(\1 +\frac{ \mathcal O( \k^{-1})}{|\Theta(W(\k)-W_0)|^2}\ri),
\ee
where $W=W(\k)$ is defined by \eqref{W(kap)}.
We have already established that if $\Re \k$  is sufficiently large (and $\Im \k$ is bounded by some constant) then the jumps 
on the lenses outside of the disks are $\mathcal O({\rm e}^{-c |\k|})$ small, for some $c>0$.

Let us consider a matrix RHP for some $G(z)$ with a jump matrix $V(z)$ on a contour $\Sigma$ and normalized by $\1$ at $z=\infty$.
Moreover, let the limiting matrices $G_\pm(z)\in L^2(\Sigma)$.
It is well known (under the general title of ``small norm theorem'') that if the $L^2$ and $L^\infty$ norms of $V(z)-\1$
are sufficiently small then the RHP is solvable in terms of a convergent Neumann series and the solution  $G(z)$ is also ``close'' to $\1$ ,  see for example Ch. 7 in \cite{Deift}.
In the  case of the error matrix $\mathcal E$, ~ $ V(z)-\1$ becomes small as $\Re \k\to +\infty$ in any $L^p$ norm.
The largest contribution  to  $ V(z)-\1$ comes from the boundaries of $\mathbb D_\ell$, and its estimate can be seen from \eqref{65}. 
Thus, $\mathcal E(z;\k)$ exists and is close to $\1$ if $\k$ is sufficiently large.
In turn, the solvability of the RHP for $\mathcal E$ implies the solvability of the RHP \ref{RHPGamma} by reverting the chain of exact transformations 
\be
\Gamma \mathop{\longrightarrow}^{\eqref{y-def}} Y \mathop{\longrightarrow}^{\eqref{422}} Z \mathop{=}^{\eqref{approx}} \mathcal E \wt Z.
\ee
Since the solvability of RHP \eqref{RHPGamma} implies the absence of (exact) eigenvalues,  
there are no exact eigenvalues as long as the error term in \eqref{65} is smaller than a suitable constant, namely, in the region
\be
\Re \k > R_0; \ \  \le|\Theta(W(\k)-W_0)\ri| >\frac C{\sqrt{|\k|}}
\label{exclusion}
\ee
for suitable constants $R_0, C>0$.

Recall that the positive approximate eigenvalues $\k_n$ are defined by the equation 
$ \Theta(W(\k)-W_0)=0$. We shall show in Lemma \ref{lemmatransversal} that  $ \Theta(W(\k)-W_0)$   has only simple zeroes $\k=\k_n$, and $ \frac {d}{d \k} \Theta(W(\k)-W_0)\big|_{\k=\k_n}$ is  bounded away from zero uniformly in $n$. Thus it follows from \eqref{exclusion} that the exact eigenvalues $\k_n^{exact}$
can only be found in small $\mathcal O(\k^{-\frac 1 2})$-size neighborhoods of the approximate eigenvalues. 
In the following Subsection \ref{sect-multi}
we shall see  that (asymptotically) near each approximate eigenvalue $\k_n$ there is precisely one exact eigenvalue of
multiplicity one.
\subsection{The location of the eigenvalues}\label{sect-multi}
The core of this section is the proof of the following theorem.
\bth
\label{finalth} 
If $n\in\N$ is sufficiently large, then there is exactly one exact eigenvalue $\k_n^{exact}$  within a distance
\be
 |\k_n^{exact}-\k_n| =\mathcal O ({\k_n}^{-\frac 1 2})
\ee
of each approximate eigenvalue 
$\k_n$.
\et 
\begin{proof}
The bound on the distance has already been argued after \eqref{exclusion}, but the actual existence of an exact eigenvalue within a neighborhood of $\k_n$ has not yet been established.
We shall show that on a small disk around each $\k_n$ (for $n$ sufficiently large) there is exactly one exact eigenvalue $\k_n^{exact}$. We start our analysis with Proposition \ref{logdet}. 
Since the integral in Proposition \ref{logdet} is on the cycle $B_1$ (denoted by $\hat I_i$ in \eqref{dlnt}) that 
avoids the branchpoints, we 
have 
\be\label{Gam-on-B1}
\Gamma(z;{\rm e}^{-\k}) =  {\rm e}^{(d_{\infty}+ \k \gg_{\infty})\s_3}\mathcal E(z;\k) \Psi(z;\k) {\rm e}^{-(\k \gg(z)+ d(z)) \s_3},
\ee
where $z\in B_1$.
Note that  $\lambda = {\rm e}^{-\k}$ and, thus, $-\lambda \pa_\lambda = \pa_\k$.
Inserting \eqref{Gam-on-B1} into Proposition \ref{logdet} and taking into account \eqref{defg}, we obtain
\bea
\pa_\k \ln \det \le(\Id - \frac 1 \lambda K\ri)&&=
-\l \pa_\l \ln \det \le(\Id - \frac 1 \lambda K\ri)  =    \oint_{B_1}\le(\Gamma_{21}\Gamma_{12}' - \Gamma_{11}\Gamma_{22}'\ri)\frac{ d z}{i\pi} \cr
&\&\hspace{-2cm}=  \oint_{B_1}\le(\Psi_{21}\Psi_{12}' - \Psi_{11}\Psi_{22}'\ri)\frac{ d z}{i\pi} - \oint_{B_1} (\k \gg'(z) +d'(z))\frac{ d z}{i\pi} + \mathcal O(\k^{-1})\cr
 &\& \hspace{-2cm}=  \oint_{B_1}\le(\Psi_{21}\Psi_{12}' - \Psi_{11}\Psi_{22}'\ri)\frac{ d z}{i\pi} + \frac{2 \k}{i\pi}\tau_{11}  - \oint_{B_1} d'(z)\frac{ d z}{i\pi} + \mathcal O(\k^{-1}).
 \label{logdetapp}
\eea
Here we have used the fact that $\mathcal E(z;\k)$ is analytic in $z$ in a neighborhood of $B_1$ and uniformly close to the 
identity matrix $\1$, and thus, by Cauchy's theorem, also its derivative $\mathcal E_z(z;\k)$ is uniformly $\mathcal O(\k^{-1})$ on $B_1$.
This approximation is uniform as long as $\k-\k_n$ remains bounded away from $0$  {\em even if $\k$ is allowed to take complex values} and $\Re \k\to \infty$. 

To compute the number of eigenvalues lying within the disk  $|\k-\k_n|<\epsilon$,  we evaluate  the integral (in $d\k$) 
of both sides of equation \eqref{logdetapp} along the circle    $|\k-\k_n|=\epsilon$. 
The only term in \eqref{logdetapp} which may have a pole at $\k=\k_n$ is the first integral,  which we now set out to compute.

We start with observing that 
\be\label{Q(z)}
Q(z):=\Psi_{21}(z)\Psi_{12}'(z) - \Psi_{11}(z)\Psi_{22}'(z)   = \lim_{x\to z} \frac 1{x-z} \le(1-\det \le[
\begin{matrix}
\Psi_{11}(z) & \Psi_{12}(x)\\
\Psi_{21}(z)  & \Psi_{22}(x)
\end{matrix}
\ri]  \ri)
\ee
 and thus we want to compute the determinant appearing above.
\bl
\label{Fayid}
We have
\be\label{fay}
S(x,z):= \frac 1{x-z} \det \le[
\begin{matrix}
\Psi_{11}(z) & \Psi_{12}(x)\\
\Psi_{21}(z)  & \Psi_{22}(x)
\end{matrix}
\ri] = \frac {C_0 h(z) h(x)  \Theta(\mathfrak u(z) -\mathfrak u(x) + W-W_0)}{\Theta(\mathfrak u(z)-\mathfrak u(x) -W_0)\Theta(W-W_0)},
\ee
where $C_0$ is defined by \eqref{defC_0}.
\el
\begin{proof}
Refer to the matrix $\Psi$ in Theorem \ref{theoremPsi}. Factoring out $h(z)h(x)$, which appears on both sides,  identity 
\eqref{fay} amounts to  
\bea
\nonumber
\det \le[\begin{array}{cc}
\ds  \frac {C_0 \Theta( \mathfrak u(z)- \mathfrak u(\infty)  + W - W_0) }
{\Theta(W-W_0) \Theta(\mathfrak u(z) - \mathfrak u(\infty) - 
W_0)} & 
 \ds \frac {C_0 \Theta( -\mathfrak u(x)- \mathfrak u(\infty)  + W- W_0) }
 {\Theta(W-W_0)\Theta(-\mathfrak u(x) - \mathfrak u(\infty)  -
W_0)}\cr
\ds  \frac {-C_0\Theta( \mathfrak u(z)+ \mathfrak u(\infty)  +W -  W_0)}
{\Theta(W-W_0)\Theta(\mathfrak u(z) + \mathfrak u(\infty)  -
W_0)} & 
 \ds \frac {-C_0\Theta( -\mathfrak u(x) + \mathfrak u(\infty)  + W- W_0)}
 {\Theta(W-W_0)\Theta(-\mathfrak u(x) +\mathfrak u(\infty) -  W_0)}
 \end{array}\ri]\\=
 \label{fay1}
 \frac {C_0(x-z) \Theta(\mathfrak u(z) -\mathfrak u(x) + W-W_0)}{\Theta(\mathfrak u(z)-\mathfrak u(x) -W_0)\Theta(W-W_0)}.
\eea
This is an instance of the famous Fay identities (\cite{Faybook}, page 33). 
The idea is to compare the two sides of \eqref{fay1} as functions of $z,x$ and verify that they have the same poles and the same periodicity 
around the $A,B$ cycles. Then a simple argument using the non-specialty of the divisor of degree $g$ whose image is 
$W-W_0 -\mathcal K$, proves that they must be proportional to each other.
Evaluation of the proportionality constant is achieved by noticing that the left side of \eqref{fay1}  tends to $[h(z)]^{-2}$ when $z=x$ 
(because it gives exactly $[ h(z)]^{-2}\det \Psi(z)$, where $\det \Psi\equiv 1$, see Theorem \ref{theoremPsi}).
Observe now the right hand side: expression \eqref{B2} from  Lemma \ref{lemmafay} implies that 
$\lim_{x\to z}  \frac {x-z}{\Theta (\mathfrak u(z) -\mathfrak u(x) + W-W_0) } =  \frac 1{C_0 h^2(z)}$. 
The proof is completed.
\end{proof}
We now  need to examine the behavior of the denominator of the right side of \eqref{fay} along the diagonal $z\sim x$. Using \eqref{B2} and substituting  $x = z+\delta$ and $\delta\to 0$
\bea
\Theta(\mathfrak u(z) -&&\!\!\! \mathfrak u (z+\delta) - W_0) =\cr
&&=
\le(-\delta{ \vec \omega}^t -\frac {\delta^2} 2{\vec  \omega}^{t'}\ri)\nabla \Theta(-W_0) + 
\frac {\delta^2}{2} \vec \omega^t\nabla^2 \Theta(-W_0) \vec \omega + \mathcal O(\delta^3)
\nonumber \\
&\&=\le(\delta \vec \omega^t +\frac {\delta^2} 2\vec  \omega^{t'}\ri)\nabla \Theta(W_0) + \frac {\delta^2}{2} \vec \omega^t\nabla^2 \Theta(W_0) \vec \omega + \mathcal O(\delta^3)\ds \nonumber \\
&\&=\delta C_0 h^2(z) +\delta^2 C_0 h h'  + \frac {\delta^2}2 \vec \omega^t\nabla^2 \Theta(W_0) \vec \omega + \mathcal O(\delta^3).
\eea
Notice now that the vector 
 $W_0 = \frac 1 2 \le(-{\bf e}_1 - {\bf e}_g +\tau {\bf e}_1\ri)$ satisfies the conditions in Lemma \ref{oddchar} 
with $\vec n = {\bf e}_1$.  
Then, writing the statement \eqref{djdkth} of  Lemma \ref{oddchar} in matrix form as
\be
\nabla^2\Theta(W_0) = -i\pi \vec n \nabla \Theta(W_0)^t  -i\pi  \nabla \Theta(W_0)\vec n^t,\ \ \vec n = {\bf e}_1,
\ee
and using \eqref{B2}, we obtain
$\vec \omega^t\nabla^2 \Theta(W_0) \vec \omega  = -2i\pi C_0h^2(z) \omega_1(z)$. 
Thus, 
\be
\Theta\le(-\int_{z}^{z+\delta} \hspace{-15pt}\vec  \omega \,{d\z} - W_0\ri) 
 =\delta C_0 h^2(z) + \delta^2 C_0 h h'  -i\pi \delta^2C_0 h^2(z) \omega_1(z). 
\ee
Hence,  using \eqref{B2} again and expanding in Taylor series as $\delta \to 0$
\begin{equation}
\begin{split}
S(z+\delta,z) & = \frac 
{C_0h(z) h(z+\delta) \Theta\le(-\int_{z}^{z+\delta} \hspace{-1pt}\vec  \omega \,{d\z} + W - W_0\ri)}
{\Theta(W-W_0)\Theta\le(-\int_{z}^{z+\delta} \hspace{-1pt}\vec  \omega \, {d\z} - W_0\ri)}\\
 & = \frac {C_0h(z) \le(h(z) + \delta h'(z) \ri)  \le(\Theta\le(W - W_0\ri) - \delta \sum_{j} \omega_j(z) \nabla_j \Theta(W-W_0)\ri)}
 {\Theta(W-W_0) \le(\delta C_0 h^2(z) + \delta^2 C_0 h h'  -i\pi \delta^2 C_0 h^2(z) \omega_1(z)\ri) } \\
& =\frac 1 \delta - \vec \omega(z) \cdot\frac{  \nabla \Theta(W-W_0) }{\Theta(W-W_0)} +  i\pi \omega_1(z) + \mathcal O(\delta).
\end{split}
\end{equation}
Thus we have obtained that the function $Q(z)$ (cf. \eqref{Q(z)}) is given by 
\be
Q(z) =\vec \omega(z) \frac {\nabla \Theta(W - W_0)}{\Theta(W-W_0)} - i\pi \omega_1(z)
\ee
and, so,
\bea
\oint_{B_1} {Q}(z)\frac{d z}{i\pi} =  \sum_{j=1}^g \frac{\tau_{1j}}{i\pi} \frac{\nabla_j 
\Theta(W - W_0)}{\Theta(W-W_0)} - \oint_{B_1}\!\!\!\! \omega_1 d z  
= \pa_\k \ln \Theta(W(\k)-W_0) - \tau_{11},
\eea
where we have used the fact that $\pa_\k W(\k) = \frac 1{i\pi} \vec \tau_1$.  
Finally, we obtain 
\bea\label{1q2w}
-\l \pa_\l \ln \det \le(\Id - \frac 1 \lambda K\ri )
= \pa_\k \ln \Theta(W(\k)-W_0) - \tau_{11}+ \frac{2 \k}{i\pi}\tau_{11}  - 
\oint_{B_1} \Delta'(z)\frac{ d z}{i\pi} + \mathcal O(\k^{-1}).  
\eea
Now the integral of the left-hand side of \eqref{1q2w} about the small circle gives precisely $2\pi i$ times the number of exact eigenvalues contained within the circle $|\k-\k_n|=\epsilon$.  We have just shown that the corresponding integral of the right-hand side gives $2\pi i$ times the number of approximate eigenvalues within the same circle. This number is equal to $1$ if $\epsilon$ is small enough because the zeroes of $\Theta(W(\k)-W_0)$ for $\k\in \R$ are all simple. This latter fact will be independently proven in  Lemma \ref{lemmatransversal} below. This shows that   
 there is only one exact eigenvalue $\k_n^{exact}$ in the $\epsilon$ circle around $\k_n$, so  
the proof of Theorem \ref{finalth} is complete.
\end{proof}

Recall that the eigenvalues $ \l_n = {\rm e}^{-\k_n^{exact}}$ 
of $\wh K$ are also singular values of $H^{-1}_e$. That is why ${\rm e}^{-\k_n}$, where $\k_n$ are approximate eigenvalues, see
Definition \ref{def-eigen}, will be called approximate singular values. The following corollary is an immediate consequence of
Theorem \ref{finalth}.

\bc
\label{cor-finalth} 
If $n\in\N$ is sufficiently large, then there is exactly one singular value $\l_n = {\rm e}^{-\k_n^{exact}}$ within a distance
\be
 |\k_n^{exact}-\k_n| =\mathcal O ({\k_n}^{-\frac 1 2})
\ee
of each approximate singular value 
${\rm e}^{-\k_n}$.
\ec 
\section{Asymptotics of the  singular values and singular functions}
\label{secdivisor}
In view of  Theorems \ref{hiteig} and \ref{finalth}, 
we need to see when, how often and with what tangency  the straight line $W(\k), \ \k\in\R,$  given by \eqref{wukappa},
intersects the theta divisor in the Jacobian $\mathbb J_\tau$ (see definition in Appendix \ref{thetaapp}).
Note that in Theorem \ref{hiteig} the vector $W(\k)$ belongs to $\R^g$ for real values of $\k$; therefore we are interested in 
studying the implicit equation $\Theta(W-W_0)=0$ for $W\in \R^g$. For this purpose we prove the next proposition.
\bp
\label{ThetaDiveig}
If $W\in \R^g$ and $W_0$ is given as in \eqref{wukappa} then
\be
\Theta\le ( W - W_0\ri)=0~~\Longleftrightarrow~~~  W = \sum_{\ell=1}^{g-1}\le( \mathfrak u (p_\ell) -\mathfrak u(a_{j_\ell}) \ri)\mod\ \Z^g ,
\ee 
where  $p_\ell= (z_\ell, R_\ell), \ell =1,\dots, g-1,$ are arbitrary points with $z_\ell\in [a_{2\ell}, a_{2\ell+1}],$ $\ell =1,\dots, g-2,$ and $z_{g-1}\in \R\setminus [a_1,a_{2g+2}]$ (i.e. belonging to the cycles $A_{1+\ell}, \ell =1,\dots, g-1$), and $j_\ell \in J = \{1, 5,7, 9,11, \dots,2g-1\}$. 
\ep
\br
The  additional information in Proposition \ref{ThetaDiveig} relative to Corollary \ref{generalThetadiv} is that we can localize the points of the degree $g-1$ divisor within the specified segments.
\er
\begin{proof}
The following proof is essentially a rephrasing of the one contained in  \cite{Faybook}, Chapter VI.
By  Theorem \ref{generalThetadiv} the Theta function will vanish if and only if 
\bea
W- W_0=  \sum_{\ell =1}^{g-1}\mathfrak u (p_\ell ) + \mathcal K\ ,
\label{condW}
\eea
for some choice of $g-1$ points $p_\ell  = (z_\ell, R_\ell)$.  
According to \eqref{W0long}  we can rewrite condition \eqref{condW} in the Jacobian $\mathbb J_\tau$ (recalling that the Abel maps of all branchpoints are half-periods) as 
\bea
W = \sum_{\ell=1}^{g-1} \le(
\mathfrak u(p_\ell) - \mathfrak u (a_{j_\ell})
\ri).
\eea
However, we need $W$ to be purely real and we want to conclude that the only possibility is that there is exactly one point $z_j$ in each of the complementary arcs containing $a_{4+2j}$ (on one or the other sheet), which will conclude the proof. 

Denote $\mathcal D_P$ the divisor consisting of the points $(z_1, R_1) , \dots,( z_{g-1}, R_{g-1})$. We shall show that these points must all be real (which means both components $(z_j,R_j)$ are real, thus in particular $z_j$ belongs to the ``gaps'').
Suppose (by contradiction) $\mathcal D_P$ is not real, but $W\in \R^g$ nonetheless. 
The functions $\omega_j$ are all real (i.e. $\ov {\omega_j(z)} = \omega_j(\ov z)$) and thus 
\be
\ov W  
= \sum_{\ell=1}^{g-1} \le(
\mathfrak u(\ov{p_\ell}) -\ov { \mathfrak u (a_{j_\ell})} 
\ri)
=  \sum_{\ell=1}^{g-1} \le(
\mathfrak u(p_\ell) - \mathfrak u (a_{j_\ell})
\ri)
=W.\label{74}
\ee
Since the imaginary parts of $\mathfrak u(a_j)$ are all half-periods, see \eqref{abela_j}, then equation \eqref{74} can be rewritten as follows
\be
\mathfrak u(\mathcal D_P)  = \mathfrak u(\mathcal D_{\ov P}) -\overbrace{ 2i \Im \le(\sum_{\ell=1}^{g-1} \mathfrak u (a_{j_\ell})\ri)}^{\in \Lambda_\tau} =   \mathfrak u(\mathcal D_{\ov P}) \mod \Lambda_\tau
\ee
and hence the divisor consisting of the conjugate points  $\mathcal D_{\ov P}$ (which is not the same) is equivalent in the Jacobian 
$\mathbb J_\tau$. 
A standard theorem (Abel's theorem, \cite{FarkasKra}, Theorem III.6.3) guarantees that there is a function $F(z)$ on the 
Riemann surface with poles at $\mathcal D_z$ and zeroes at $\mathcal D_{\ov z}$. This means that the divisor $\mathcal D_z$ 
is {\em special} (Definition \ref{defspecial}) \cite{FarkasKra}.
For a hyperelliptic surface like ours, this can only be if there is at least a pair of points on the two sheets of the 
form $p_1 = (z_1,R_1)$, $p_2 = (z_2,R_2)$ with $z_1=z_2, R_1=-R_2$. Denote by  ~~$\wh{ }$~~ the exchange of sheet. 
Then the divisor $\mathcal D= \mathcal D_P$ can be written as $ 
\mathcal D = \mathcal D_1 + \wh{ \mathcal D_1} + \mathcal D_0
$, with $\mathcal D_1$ not empty. Then re-define 
\be
\mathcal D'  = \mathcal D_1 + \wh{\ov{ \mathcal D}_1} + \mathcal D_0
\ee
which is then non-special.
We still have $\mathfrak u(\mathcal D') = \mathfrak u(\ov {\mathcal D'})$ and -again by Abel's theorem- there would exist a non-constant function $\wt F$ with poles at $\mathcal D'$ and zeroes at $\ov {\mathcal D'}$: but  now the divisor $\mathcal D'$ is -by construction- non special and of degree $g-1$ and thus we reach a contradiction.

We thus have established that all points $p_j$ must belong to gaps on the real axis. 
It is then easily seen using the linear independence of the columns of $\tau$, that they must belong each to 
the appropriate $A$-cycle,  as stated.\end{proof}

The Theta divisor  $(\Theta)$ denotes the whole zero locus of $\Theta$ in the Jacobian (i.e. $\C^g$), 
see Definition \ref{def-thetadiv}.
In view of Proposition \ref{ThetaDiveig}, we shall denote by $(\Theta)_\R$ the 
locus of $W \in \R^g$ such that
 $\Theta(W-W_0)=0$ 
(i.e. a particular real section of $(\Theta)$). 
Then, Proposition \ref{ThetaDiveig} and \eqref{wukappa} imply the following corollary.
\bc\label{cor-eigenv-g}
The approximate eigenvalues $\k_n$ are given by 
the intersections of the straight line $W(\k)$ (see \eqref{wukappa}) and the surface $(\Theta)_\R$ 
defined parametrically 
inside the $g$-dimensional \em real torus $\mathbb T_g = \R^g \mod \Z^g$ by
\be\label{Theta-div-g}
\vec X(p_1,\dots, p_{g-1}) = \sum_{\ell =1}^{g-1} \int_{a_{j_\ell}} ^{p_\ell}  \vec \omega(\z)d\z, 
\ee
where 
the points $p_\ell = (z_\ell, R_\ell)$ belong to the cycles $A_{\ell+1}$, $\ell = 1,\dots, g-1$\footnote{
Recall that the expression $\int_{a_{j_\ell}}^{p_\ell} \vec \omega (\z)d \z$ means $\pm \int_{a_{j_{\ell}}} ^{z_\ell} \vec \omega(\z) d \z$ with the sign depending on whether $p_\ell$ belongs to the first ($+$) or second ($-$) sheet. }
and $ \vec X (p_1,\dots, p_{g-1} ) =[ X_1(p_1,\dots, p_{g-1}), \dots,   X_g(p_1,\dots,p_{g-1})]^t$.
\ec
\br
\label{rem81}
Formula \eqref{Theta-div-g} represents a map from the $g-1$-dimensional {\em real} torus $A_2\times \dots \times A_g$ into the $g$-dimensional {\em real} torus $\mathbb T_g$. If we think of it on the respective universal covering spaces, then one can see that as one of the points $p_\ell$ makes a turn on its corresponding cycle $A_{\ell+1}$, the $(\ell+1)$st component of $\vec X$ is incremented by one due to the normalization of the vector $\vec \omega$, see \eqref{normalization}.
\er
We represent the torus by choosing a fundamental domain $[-\frac 12, \frac 1 2]^g$ with the opposite sides identified. 
The  pictures of the parametric surfaces  $(\Theta)_\R$  in the cases $g=2,3$ are shown in Figure \ref{ThetaDivisor}.

\begin{figure}[h]
\includegraphics[width=0.48\textwidth]{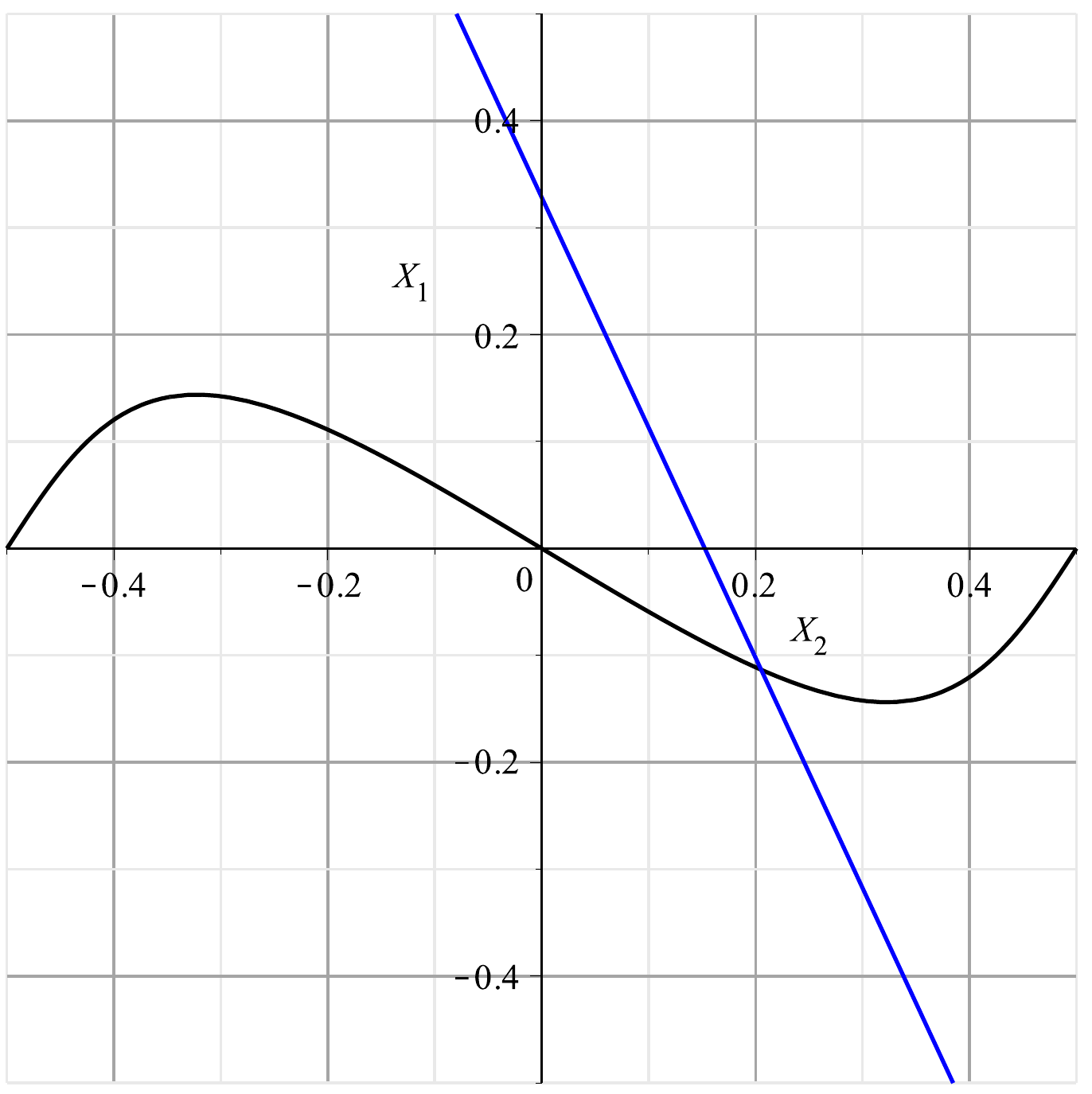}
\includegraphics[width=0.48\textwidth]{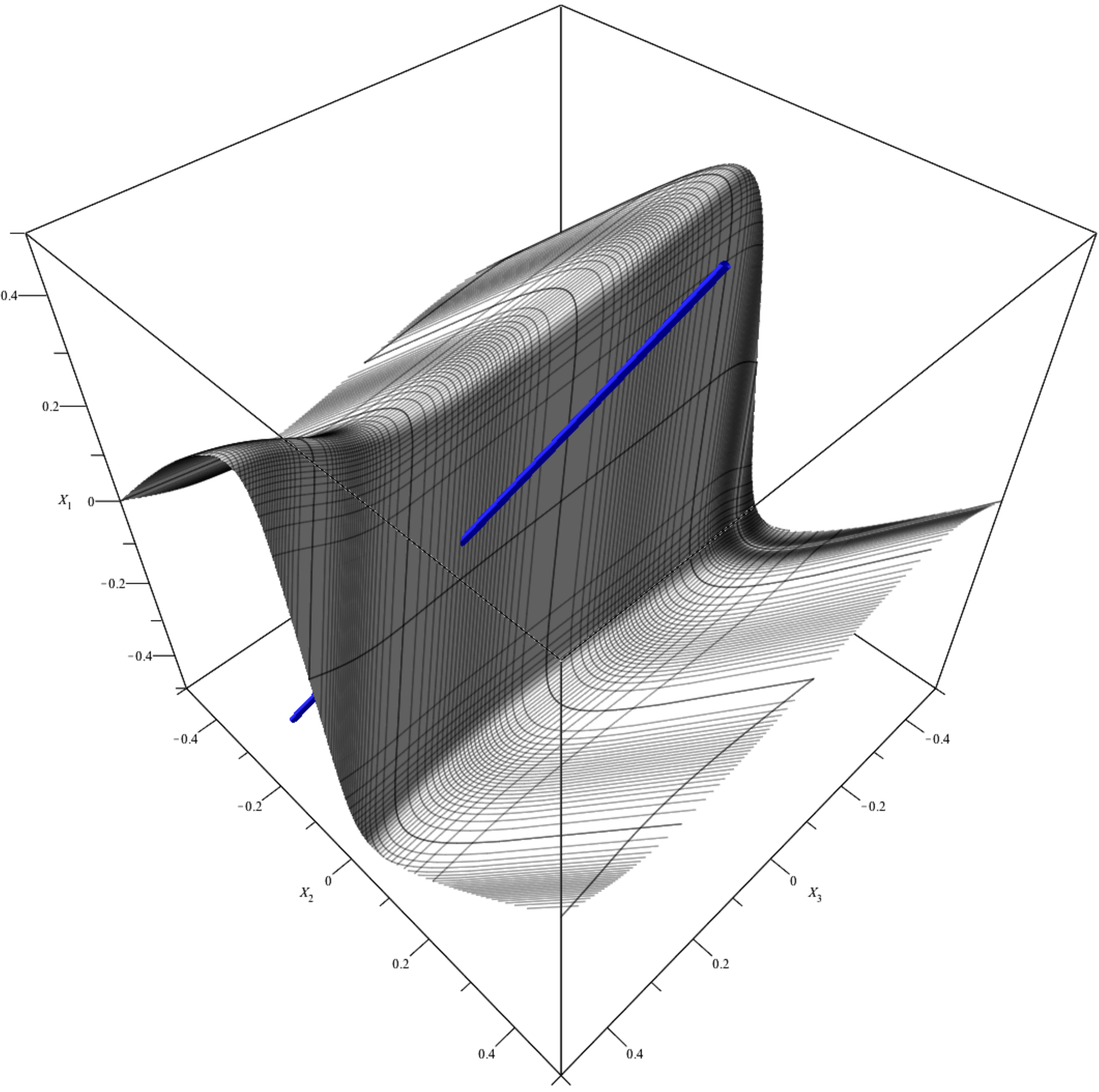}
\caption{Case $g=2$ (left panel): the image $\mathfrak u(z)$, $z\in \R \setminus [a_1,a_6]$,  of the 
surface $(\Theta)_\R$  (on the square $[-1/2,1/2]^2$ with periodic boundary identifications) as per Proposition \ref{ThetaDiveig}. 
The model problem \eqref{modelRHP} is unsolvable for some $\k_n$ if and only if  the straight line  $W(\k)\in \R^2 \!\!\mod \Z^2$ 
(plotted as an example)  intersects the plotted curve at $\k=\k_n$. 
Case $g=3$ (right panel): the image of $\pm  \int_{a_1}^{z_1} \vec \omega d\z \pm  \int_{a_4}^{z_2} \vec \omega d\z$, $z_1\in \R \setminus [a_1,a_8]$ and $z_2\in [a_4,a_5]$. In blue we plotted an example of the line $W(\k)\in \R^3 \!\! \mod \Z^3$: the values of $\k$ corresponding to the intersections are the approximate eigenvalues $\k_n \sim -\ln\l_n$. In both cases we have $X_1 = F(X_2)$ and $X_1 = F(X_2,X_3)$, respectively, as explained in Lemma \ref{Thetasmooth}. The line $W(\k)$ is parallel to the (imaginary part of the) first column of the normalized matrix of $B$-periods and  intersects the surface by always forming an acute angle with the upward normal (Lemma \ref{lemmatransversal}).
}
 \label{ThetaDivisor}
\end{figure} 
\bl
\label{Thetasmooth}
Each connected component of the surface $(\Theta)_\R$ on the universal covering of $\mathbb T_g\sim \R^g$  
is smooth.
Moreover, it 
 can be expressed as the graph of a function $X_1 = F(X_2,\dots X_g)$, where the function  $F: \R^{g-1}\to \R$ is odd and periodic 
of period $1$ in each argument. 
\el
\begin{proof}
By Proposition \ref{ThetaDiveig} the points of $(\Theta)_\R$ are parametrized by a divisor of $g-1$ points $p_j= (z_j, R_j)$ chosen arbitrarily in the  $A_{j+1}$ cycle ($j=1,\dots, g-1$). Therefore, with notation \eqref{Theta-div-g},
\be
\Theta( \vec X(\vec p)  - W_0)\equiv 0.\label{impl}
\ee
We want to study the gradient of $\Theta(W-W_0)$ with respect to $W\in \R^g$ at the points where $\Theta(W-W_0)=0$; 
we notice that since $\Theta$ is periodic on the lattice  $\Z^g$ \eqref{thetaperiods}, the surface $\Theta(W-W_0)=0,\ \ W\in \R^g$ is certainly periodic.
To this end we introduce
\be
v(z) = \vec \omega(z) \cdot \nabla \Theta(W-W_0).
\label{vdiff}
\ee
By Lemma \ref{lemmafay}, $v(z)$ is not identically zero and vanishes at the points $p_j$ (see  Corollary \ref{cor-eigenv-g}).
Therefore it is of the form 
\be
v(z) = C \frac{\prod_{j=1}^{g-1}(z-z_j)}{R(z)}\ , \ \ C\neq 0.
\ee
Equations \eqref{vdiff} and  \eqref{normalization} (normalization  of  $\omega_j(z)$) imply that
\be\label{partTheta}
\frac\pa {\pa W_s} \Theta(W-W_0)  = \oint_{A_s} v(z) d z, \  \ \  s=1,\dots, g,
\ee
where $W_s$ denotes the $s$-th component of the vector $W\in\R^g$.
In particular, the first component of the gradient of $\Theta$ in \eqref{partTheta} is given by $C \oint_{A_1}\frac {\prod_{j=1}^{g-1} (\z-z_j) d \z }{R(\z)}$. This expression never vanishes on the surface, because we have established that none of the points $z_j$ belongs to $A_1$ and thus the integrand has a definite sign. Thus, by the implicit function theorem, it follows that we can express $X_1$ as a  smooth function $F(X_2,\dots, X_g)$ of the remaining components of $\vec X$, locally, around any point of the surface and:
\be
\frac {\pa X_1 }{\pa X_s} =- \frac {\oint_{
A_s} \frac {\prod_{j=1}^{g-1} (\z-z_j) d \z }{R(\z)}}{\oint_{
A_1} \frac {\prod_{j=1}^{g-1} (\z-z_j) d \z }{R(\z)}}\in\R,~~~~~s=2,\dots,g.
\label{normal}
\ee
%
%
%
We now turn to global injectivity.  If there were two $\vec p^{(1)}, \vec p^{(2)}\in A_2\times \dots \times A_g$ with the same image in the Jacobian under the Abel mapping, then there would exist (by Abel's 
theorem) a non-constant and nonzero meromorphic function with at most $g-1$ poles at the points $p^{(1)}_j$. But this divisor 
is non-special (Definition \ref{defspecial} and following remark) because there is at most one point in each cycle $A_j$. 
This is a contradiction, which proves the required statement.

Thus we have proved that  the map $\vec X:A_2\times \dots \times A_g\to \mathbb T_g$  defines  a smooth local function $X_1 = F(X_2,\dots X_g)$ in the neighborhood of any point of the surface. Denote by  $\pi_{\wh 1}:\mathbb T_g\to \mathbb T_{g-1}$ the projection onto the coordinates $X_2,\dots,X_g$, $\mod \Z$: then the assertions proven above imply that $\pi_{\wh 1}\circ\vec X: A_2\times \dots \times A_g\to\mathbb T_{g-1}$ is a smooth map, in particular it is an open map. Now, $A_2\times \dots \times A_g$ is a topological space without boundary (meaning that each point is an interior point) and $\pi_{\wh 1}\circ \vec X$ is smooth and open between two compact spaces. Thus the range of the map cannot have a boundary (by a simple compactness argument) and hence it is surjective. In terms of the universal covering spaces this means that the domain of $F$ is the whole $\R^{g-1}$, and therefore, by a compactness argument,  $F$ is a  {\em global} smooth function.
Hence for each $N \in \Z$ the corresponding connected component of $(\Theta)_\R$ is given by  
 $X_1 = F(X_2,\dots X_g)+N$.

Finally we prove the symmetry: since $W_0$ is a half-period of the Jacobian, then $\Theta(W-W_0)=0$ if and only if $\Theta(W+W_0)=0= \Theta(-W-W_0)$ (using that $\Theta$ is an even function).  Therefore the surface $(\Theta)_\R$ is symmetric about the origin, which means that $F(x_2,\dots x_g)$ is an odd function.
\end{proof}
\br
On the universal covering of $\mathbb T_g$ (i.e. $\R^g$) the surface $(\Theta)_\R$ is then represented as 
a countable union of the graphs $X_1 = F(X_2,\dots, X_g) + n$ with $n\in \Z, X_2,\dots, X_g\in \R$.
\er
\bl
\label{lemmatransversal}
{\bf (1)} The surface $(\Theta)_\R$ in $\mathbb T_g$ is orientable. In particular, there is a continuous choice of normal 
direction which forms an acute angle with the fixed vector $i\t_1$ (see Fig. \ref{ThetaDivisor});
{\bf (2)} The intersections of the line $W(\k): \R \to \mathbb T_g$ with $(\Theta)_\R$ are transversal and, as a consequence,
  the zeroes of $\Theta(W(\k)-W_0)$ as a function of $\k$ are all simple.
\el
\begin{proof}
{\bf (1)} According to \eqref{normal} a choice of normal direction to the surface {$(\Theta)_\R$ at the point parametrized by $\vec p\in A_2\times \dots\times  A_{g}$ (see Proposition \ref{ThetaDiveig})}  is the vector 
\be
\vec n = \vec n (\vec p) = \le[
\oint_{
A_1} \frac {\prod_{j=1}^{g-1} (\z-z_j) d \z }{R(\z)}, \dots, \oint_{
A_g} \frac {\prod_{j=1}^{g-1} (\z-z_j) d \z }{R(\z)}
\ri]^t.
\ee
Then
\be
i\tau_1\cdot \vec n =i \sum_{j=1}^g \oint_{B_j} \omega_1d \z\oint_{
A_j} \frac {\prod_{j=1}^{g-1} (\z-z_j) d \z }{R(\z)}.
\label{dotp}
\ee
By the Riemann Bilinear Identity (\cite{FarkasKra}, equation (3.1.1), p. 64),  we have
\begin{equation}
i\tau_1\cdot \vec n  =i \sum_{j=1}^g \overbrace{\oint_{A_j} \omega_1d \z}^{=\delta_{1j}} \oint_{
B_j} \frac {\prod_{j=1}^{g-1} (\z-z_j) d \z }{R(\z)} 
=2 \sum_{j=2}^{g} \int_{a_{2j-1}}^{a_{2j}} \frac {i\prod_{j=1}^{g-1} (\z-z_j) d \z }{R{_+}(\z)}.
\label{dotp2}
\end{equation}

Each integral in the latter sum is of the same sign. Indeed, on any main arc (branchcut)
$R(\z) = \sqrt{\prod (\z-a_j)}\in i\R$ and the sign of $\Im (R_+(\z))$ on the main arcs alternates between the neighboring main arcs.
The product in  the numerator  of \eqref{dotp2} has exactly one zero in the gap in between.
Thus, the  sign of the (real) integrand in the last term of \eqref{dotp2} does not change from one
main arc $[{a_{2j-1}},a_{2j}]$ to another. Therefore, $i\tau_1\cdot \vec n\neq 0$ for any point of 
the surface $(\Theta)_\R$ and ${\rm sign}(i\tau_1\cdot \vec n) $ is constant on $(\Theta)_\R$.

{\bf (2)} The surface  $(\Theta)_\R$ is the zero level surface of $\Theta(z)$ and, therefore, 
the vector  $\nabla\Theta$ is parallel to $\vec n$.
By part (1), $\nabla\Theta$
cannot be orthogonal to 
$\frac d{d\k} W(\k)  = i\tau_1$. Now $\Theta$ is an analytic function and the derivative 
$ \frac d{d\k}  \Theta(W(\k)-W_0)  = \frac d{d\k}  W(\k)\cdot  \nabla \Theta$ does not vanish on $(\Theta)_\R$.  
Thus the zeroes of $\Theta(W(\k)-W_0)$ are simple.
\end{proof}

\bc\label{cor-1-inters}
The line $W(\k)$ given by \eqref{wukappa} intersects each connected component of $(\Theta)_\R$ in  the universal cover exactly once. 
\ec
\begin{proof}
Consider the function $G(X_1,\dots, X_g)= X_1 - F(X_2,\dots, X_g)$, where $F$ is the function whose graph defines the 
hypersurface $(\Theta)_\R$. Then the composite function  $G(\vec W(\k))$ is strictly monotonic   by the transversality of 
Lemma \ref{lemmatransversal}. Thus it vanishes exactly once. \end{proof}

The straight line $W(\kappa)$ has slopes $\frac {\tau_{1\ell}}{\tau_{11}}$ and it traverses the fundamental domain from top to bottom with period $\frac {i\pi}{\tau_{11}}$. 
It should be clear that this line must intersect {\em at least one} connected component of $(\Theta)_\R$ for each 
increment $\k \mapsto \k + \frac {i\pi}{\tau_{11}}$.

\bth
\label{maxexcursion} Let  $X_1= F(X_2,\dots, X_{g-1})+ n $ be an equation of the $n$-th connected component of $(\Theta)_\R$ 
(see Lemma \ref{Thetasmooth}), where $n\in\Z$.
Then the ``maximum excursion'' of each connected component of $(\Theta)_\R$ in the $X_1$ direction is bounded by $g-1$, namely, 
\be
-\min_{\vec X\in \R^{g-1}} F(\vec X) = \max_{\vec X\in \R^{g-1}} F(\vec X)\leq \frac {g-1}2,
\ee
where $\vec X= (X_2,\dots,X_g)$.
\et
\begin{proof}
Consider each term in 
$X_1 = X_1(\vec z) = \sum_{\ell=1}^{g-1} (\pm) \int_{a_{j_\ell}}^{z_\ell} \omega_1(\z) d\z$
separately and with the sign that yields a positive contribution.  First of all, since our 
points $z_\ell$ are on the real axis, each term is real and, thus, we can consider the real part of $X_1$. 
But then each term is the restriction  of the  harmonic functions  $h_\ell (z):= \Re \int_{a_{j_\ell}}^z \omega_1d\zeta$ to a real segment.
Each restriction is  the same harmonic function up to an additive constant. 
Each of them is a harmonic function on $\bar\C\setminus \bigcup_{j=1}^3 [a_{2j-1}, a_{2j}]$ (harmonic also at $\infty$).
Its boundary value on each main arc $[a_{2j-1}, a_{2j}]$ is constant and equal to $0$ or $\pm \frac 1 2$ (depending on the value of $\ell$). 
Thus the maximum and minimum of each $h_\ell(z)$ on the {\em whole} $\C \setminus I_i\cup I_e$ are $\pm \hf$. So, (since 
the sign on the other sheet is the opposite one) each term can contribute at most a ``maximum excursion'' of $1/2$. 
\end{proof}
\br
After extensive numerical investigation we observed that the stronger statement might be true: for any
 genus $g$, the maximum excursion of each component of the surface $(\Theta)_\R$ is at most $1$.  
\er
The following corollary  provides the bounds on the number $n$ of intersections of 
the line $W(\k)$, see \eqref{wukappa},
with $(\Theta)_\R$ when 
$\k \in \le [\k_0 , \k_0 + \frac{N(g-1) i\pi}{\tau_{11}}\ri)$, where $\k_0>0$ and $N\in\N$.
\bp\label{prop-ass}
For any $N\in\N$ the number $m(N)$ of intersections of 
the line $W(\k)$, see \eqref{wukappa}, where $\k \in \le [\k_0 , \k_0 + \frac{N(g-1) i\pi}{\tau_{11}}\ri)$,
with $(\Theta)_\R$ is bounded by
\be\label{bound_n}
(N-1)(g-1)\leq m(N) \leq (N+1)(g-1).
\ee
\ep
\begin{proof} The number of different connected components of $(\Th)$ in the consecutive $N(g-1)$ periods in the $x_1$ direction cannot exceed 
$(N+1)(g-1)$. Now the upper bound follows from Corollary \ref{cor-1-inters}. On the other hand, 
 any  consecutive $N(g-1)$ periods in the $x_1$ direction contain at least $(N-1)(g-1)$  different connected components of $(\Th)$
 (that is, these components are completely contained within the aforementioned periods). Using Corollary \ref{cor-1-inters} again,
 we obtain the lower bound in \eqref{bound_n}.
\end{proof}
\bc \label{cor-ass}
Let $\k_n$ be the $n$-th approximate eigenvalue, that is, the  $n$-th zero of  
\be
 \Theta\le ( W(\k) - W_0\ri  )=0,\label{zeroes-of-theta}
\ee
where $W(\k)$ and $W_0$ are defined by \eqref{wukappa} and the labelling of solutions of \eqref{zeroes-of-theta} fixed by Remark \ref{enumeration}.
Let  $2\l_n=e^{-\k_n+\ln 2}$ denote the corresponding approximate singular values. Then
\be\label{asskap}
\k_n=n\pi \frac{i }{\t_{11}}+O(1),~~~~~~~~~\qquad \la_n=e^{-n\pi \frac{i}{\t_{11}}+O(1)}.
\ee
\ec
The proof of Corollary \ref{cor-ass} follows from \eqref{wukappa} and Proposition \ref{prop-ass}. 
Combined with Corollary \ref{cor-finalth}, Corollary  \ref{cor-ass} yields the following result.
\bth \label{theo-ass}
Let $\k^{exact}_n$ be the $n$-th  exact eigenvalue of the RHP  \eqref{RHPGamma}, see Definition \ref{def-eigen}, and let 
$2\l^{exact}_n=e^{-\k^{exact}_n+\ln 2}$ denote the corresponding  singular values of $\CH^{-1}_e$. Then
\be\label{asskap-exact}
\k^{exact}_n=n\pi \frac{i }{\t_{11}}+O(1),~~~~~~~~~\qquad 2\la^{exact}_n=e^{-n\pi \frac{i}{\t_{11}}+O(1)}.
\ee
\et
\subsection{Comparison with the two-interval (genus 1) case}\label{subsec-svass}

We would like to compare  the one interior and two exterior interval case (i.e. $g=2$ in \eqref{def-hilb}--\eqref{int-int})  with the corresponding asymptotics in the two interval (four-point) case ($I_e=[a_1,a_2]$, $I_i=[a_3,a_4]$). The latter is obtained from the former in the limit when  the rightmost exterior interval $[a_5,a_6]$ collapses into a point $a$ as $a_5=a-\e,~a_6=a+\e$ with $\e\ra 0^+$. 
The singular values of the FHT (with no  weight) in the  
four point case was studied in \cite{kt12}.  Since the weight $w(z)$  affects only $d(z)$, but not $ \gg(z)$, see
\eqref{y-def}, the leading large $\k$ asymptotics of singular values for both FHTs remains the same. 
According to \eqref{taumatrix}, \eqref{1stkind}, \eqref{T-n} and \eqref{T-n} (for $g=2$)
\begin{equation}
\begin{split}
\label{tau11}
\t_{11}&=
[1,0] \mathbb A^{-1} \oint_{\hat\g_1}\le[
\begin{matrix}
1\\
\z
\end{matrix}
\ri]\frac{d\z}{R(\z)}=
[1,0] \le[\begin{array}{cc}
1 & -1 \\
0 & 1\end{array}\ri]\dfrac{\le[ 
\begin{array}  {c c c } 
\int_{c_2}\frac{\z d\z}{R(\z)} & -\int_{c_2}\frac{ d\z}{R(\z)} \\
-\int_{c_1}\frac{\z d\z}{R(\z)} & \int_{c_1}\frac{ d\z}{R(\z)} 
\end{array}
\ri]
\le[
\begin{matrix}
\int_{\g_1}\frac{d\z}{R(\z)}\\
\int_{\g_1}\frac{\z d\z}{R(\z)}
\end{matrix}\ri]
}{\le | 
\begin{array}  {c c c } 
\int_{c_1}\frac{d\z}{R(\z)} & \int_{c_2}\frac{ d\z}{R(\z)} \\
\int_{c_1}\frac{\z d\z}{R(\z)} & \int_{c_2}\frac{\z d\z}{R(\z)} 
\end{array}
\ri |}
\\
&= \frac{ \int_{c_2\cup c_1}\frac{\z d\z}{R(\z)}\int_{\g_1}\frac{d\z}{R(\z)}- \int_{c_2\cup c_1}\frac{ d\z}{R(\z)}\int_{\g_1}\frac{\z d\z}{R(\z)}}
{\int_{c_2}\frac{\z d\z}{R(\z)}\int_{c_1}\frac{d\z}{R(\z)}- \int_{c_2}\frac{ d\z}{R(\z)}\int_{c_1}\frac{\z d\z}{R(\z)}}=
\dfrac{\le | 
\begin{array}  {c c c } 
\int_{\g_1}\frac{d\z}{R(\z)} & \int_{c_1\cup c_2}\frac{ d\z}{R(\z)} \\
\int_{\g_1}\frac{\z d\z}{R(\z)} & \int_{c_1\cup c_2}\frac{\z d\z}{R(\z)} 
\end{array}
\ri |}{\le | 
\begin{array}  {c c c } 
\int_{c_1}\frac{d\z}{R(\z)} & \int_{c_2}\frac{ d\z}{R(\z)} \\
\int_{c_1}\frac{\z d\z}{R(\z)} & \int_{c_2}\frac{\z d\z}{R(\z)} 
\end{array}
\ri |}.
\end{split}
\end{equation}
Since $\t_{11}$ is the ratio of two determinants, we can replace $\z d\z$  in the numerators of the second rows of each determinant
by $(\z-a)d\z$.
Let us calculate the limits of   $ \int_{c_2}\frac{(\z-a)^k d\z}{R(\z)}$, where $k=0,1$, as    $\e\ra 0^+$. First note that
\be\label{R}
R(\z)=R_0(\z)\sqrt{(\z-a)^2-\e^2},
\ee
where $R_0(\z)=\prod_{j=1}^4 (\z-a_j)^\hf$ is an analytic function near $\z=a$. Because of the shrinking interval $\g_2$,
the large contribution to the integrals over the complementary arc $c_2$ may come only from a neighborhood of $a$. Replacing
this integral with the integral over the corresponding cycle on the Riemann surface $\Rscr$, we obtain
\be\label{approx-int-c2}
\int_{c_2}\frac{(\z-a)^k d\z}{R(\z)}=-\hf\le(\int_a^{a+ir} +  \int_a^{a-ir} \ri)    \frac{(\z-a)^k d\z}{R(\z)} +O(1)
\ee
as $\e\ra 0$, where $r>0$ is a fixed number. Notice that for any $k\in\R^+$ the integrals in the right hand side of 
\eqref{approx-int-c2} are bounded. Indeed, $R_0(\z)$ is bounded and
\be
\le |\frac{(\z-a)^k}{\sqrt{(\z-a)^2-\e^2}}\ri | \leq |\z-a|^{k-1}\qquad {\rm when}~~~~\z\in[a-ir,a+ir].
\ee

In the remaining case $k=0$, we start with the Taylor formula $[R(\z)]^{-1}=[R(a)]^{-1}+(\z-a)(R^{-1})'(\eta)$,
where $\eta\in (\z,a)$. Since $(R^{-1})'(\z)$ is bounded on $[a-ir,a+ir]$, we can rewrite \eqref{approx-int-c2} as
\begin{equation}
\begin{split}
\label{approx1-int-c2}
\int_{c_2}\frac{ d\z}{R(\z)}&=-\frac{1}{2R_0(a)}\le(\int_a^{a+ir} +  \int_a^{a-ir} \ri)    \frac{d\z}{\sqrt{(\z-a)^2-\e^2}} +O(1)\\
&=-\frac{1}{2R_0(a)}\le(\int_0^{\frac{ir}\e} +  \int_0^{-\frac{ir}\e} \ri)  \frac{dt}{\sqrt{t^2-1}}+O(1)=\frac{\ln\e}{R_0(a)}+O(1).
\end{split}
\end{equation}

Now, according to \eqref{tau11} and \eqref{approx1-int-c2},
$\ds \lim_{\e\ra 0} \t_{11}=\lim_{\e\ra 0} \frac{\int_{\g_1}\frac{(\z-a) d\z}{R(\z)}}{\int_{c_1}\frac{(\z-a) d\z}{R(\z)}}= 
\frac{\int_{\g_1}\frac{ d\z}{R_0(\z)}}{\int_{c_1}\frac{ d\z}{R_0(\z)}}$.
So,
\be
2\la_n=e^{-n\pi   \frac{\int_{c_1}\frac{ d\z}{R_0(\z)}}{\int_{\g_1}\frac{ d\z}{iR_0(\z)}}        +O(1)},
\ee
which coincides with the corresponding expression from \cite{kt12}.

\subsection{Asymptotics of Eigenfunctions}
Let $\k_n$ be the $n$-th approximate eigenvalue. Then $\Theta(W(\k_n)-W_0)=0$. If $\mathbf f_n:= W(\k_n)-W_0$ then,
according to Theorem \ref{generalThetadiv},   $\mathbf f_n=  \sum_{j=1}^{g-1}\mathfrak u(p_j) + \mathcal K$,
where the points $p_j =(z_j, R_j)$, discussed in Proposition \ref{ThetaDiveig}, belong to the cycles $A_{j+1}$ of the Riemann surface. For this reason it makes sense to consider ${\bf f}(\vec p) :=   \sum_{j=1}^{g-1}\mathfrak u(p_j) + \mathcal K$, where $\vec p=(p_1,\dots,p_{g-1})$, as a function on the (universal cover) of the torus $A_2\times \dots A_g$, so that ${\bf f}_n = {\bf f} (\vec p_n)$ for some points $\vec p_n\in  A_2\times \dots A_g$.
\bl
\label{lemmarespsi}
{\bf (1)} For $\Psi(z;\k)$ from Theorem \ref{theoremPsi} we have
\bea\label{res-Psi}
\res{\k=\k_n} \Psi(z;\k)  =
C_0 \le[
\begin{matrix}
\ds \frac{i\pi \Theta\le(\mathfrak u(z) -\mathfrak u(\infty) + {\bf f}(\vec p_n) \ri) h(z)}{\vec \tau_1 \!\!\cdot\!\!\nabla \Theta(
{\bf f}(\vec p_n)) \Theta\le(\mathfrak u(z) -\mathfrak u(\infty)+W_0\ri)} 
&
\ds \frac{i\pi \Theta\le(-\mathfrak u(z) -\mathfrak u(\infty) +{\bf f}(\vec p_n) \ri) h(z)}{\vec \tau_1 \!\!\cdot\!\!\nabla \Theta(
{\bf f}(\vec p_n)) \Theta\le(-\mathfrak u(z) -\mathfrak u(\infty)+W_0\ri)} \\[5ex]
\ds \frac{-i\pi \Theta\le(\mathfrak u(z) +\mathfrak u(\infty) +{\bf f}(\vec p_n)\ri) h(z)}{\vec \tau_1 \!\!\cdot\!\!\nabla \Theta(
{\bf f}(\vec p_n)) \Theta\le(\mathfrak u(z) +\mathfrak u(\infty)+W_0\ri)} 
&
\ds \frac{-i\pi \Theta\le(-\mathfrak u(z) +\mathfrak u(\infty) + {\bf f}(\vec p_n) \ri) h(z)}{\vec \tau_1 \!\!\cdot\!\!\nabla \Theta(
{\bf f}(\vec p_n)) \Theta\le(-\mathfrak u(z) +\mathfrak u(\infty)+W_0\ri)} 
\end{matrix}
\ri].
\eea
${\bf (2)}$  For any $\vec p_n \in A_2\times \dots\times A_g$ the matrix is not identically zero. {\bf(3)} The two rows are proportional to each other for any $\vec p_n \in A_2\times \dots\times A_g$.
\el
\begin{proof}
{\bf (1)} Expression \eqref{res-Psi} follows directly from Theorem \ref{theoremPsi}.

{\bf (2)} Let $\vec p$ be an arbitrary point on  $A_2\times \dots A_g$.
 It was  proven in Lemma \ref{lemmatransversal}, {\bf(1)}, that $\vec \tau_1 \cdot \nabla \Theta({\bf f}(\vec p))\neq 0$.
   Similarly the other Theta functions in the denominators cannot be identically zero functions of $z$ because the degree $g$ divisors $\sum_{j=1}^{g-1} p_j + \infty_{1,2}$ are  non-special as well.  The only expressions that may be identically zero are the Theta functions $\Theta(\pm \mathfrak u(z) \pm \mathfrak u(\infty) + {\bf f}(\vec p_n))$. According to the Riemann theorem  (\cite{FarkasKra}, p. 317)  this happens if and only if $ {\bf f}(\vec p) \pm \mathfrak u(\infty) - \mathcal K = \sum_{j=1}^{g-1}\mathfrak u(p_j)\pm \mathfrak u(\infty)$ is a special divisor of degree $g$. According to the description of special divisors (see after Definition \ref{defspecial}) this happens only if there is a pair of points with the same $z$--coordinate. By inspection we see that this happens if and  only if $p_{g-1}\in A_g$ lies above $z=\infty$ in such a way that $\u(p_{g-1}) \pm \u(\infty)=0$ (i.e. on the secondary sheet in the case of the $+$ sign and on the main sheet in the case of the $-$ sign).
For example, if $p_{g-1} = \infty_1$ (the infinity on the main sheet), then the whole first row of \eqref{res-Psi} vanishes; but in that case the second row is not identically vanishing because $\sum_{j=1}^{g-1}\mathfrak u(p_j)+ \mathfrak u(\infty_2 )$ is a non-special divisor of degree $g$ (and thus the entries of the second row vanish only at $g$ points of the Riemann surface). Similarly, if $p_{g-1} = \infty_2$, the {\em second} row vanishes, but the first one does not. 

{\bf (3)}  We now address the linear dependence of the rows  of \eqref{res-Psi} as a function of $\vec p=\vec p_n\in A_2\times \dots \times A_g$. First of all we see that  ${\bf f}(\vec p)$ changes by integer valued vectors as $\vec p$ winds around the torus $A_2\times \dots A_g$. Since  $\Theta$-function  is (strictly) periodic in $\Z^g$, the right hand side of \eqref{res-Psi} is a well--defined function of $\vec p$ on the torus $A_2\times \dots A_g$. Now, of course, we can assume without loss of generality that both rows are non identically zero (in $z$) (if one of them is identically zero  there is nothing else to prove).
Consider the ratio of $(1,1)$ and $(2,1)$ entries of \eqref{res-Psi}
\be\label{1121}
-\frac{ \Theta\le(\mathfrak u(z) -\mathfrak u(\infty) + {\bf f}(\vec p) \ri) }{\Theta\le(\mathfrak u(z) -\mathfrak u(\infty)+W_0\ri)}
 \frac{\Theta\le(\mathfrak u(z) +\mathfrak u(\infty)+W_0\ri)} 
 { \Theta\le(\mathfrak u(z) +\mathfrak u(\infty) + {\bf f}(\vec p) \ri)}\ .
\ee
Using the periodicity of the $\Theta$-functions \eqref{thetaperiods} and the jump relations for $\mathfrak u(z)$ in Proposition \ref{prop-jumpu}, the reader can verify that the above expression {\em does not} have jumps across the real axis, and therefore defines a rational function. In fact the denominator can vanish at most at $g$ points by Theorem \ref{generalTheta}.
The zeroes of the numerator are at $\infty_{1,2}$ (on the two sheets), at the points $a_{j}{,\ j\in J} = \{1,5,7,9...,2g-1\}$ (see Proposition \ref{ThetaDiveig}) and at the points $\wh p_1,\dots, \wh p_{g-1}$ (recall, $\wh{}$ denotes the sheet exchange);  both these statements follow from the explicit form of $W_0$ in \eqref{W0long}, the definition of  ${\bf f}(\vec p) $ and from Theorem \ref{generalTheta}. The denominator has zeroes at the same points and thus the ratio is constant.
In a similar way, one can verify that  
the ratio of the $(1,2)$ and $(2,2)$ entries of \eqref{res-Psi} is also a constant: 
\be
-\frac{ \Theta\le(-\mathfrak u(z) -\mathfrak u(\infty) + {\bf f}(\vec p)\ri)} {\Theta\le(-\mathfrak u(z) -\mathfrak u(\infty)+W_0\ri)}
 \frac{\Theta\le(-\mathfrak u(z) +\mathfrak u(\infty)+W_0\ri)} 
 { \Theta\le(-\mathfrak u(z) +\mathfrak u(\infty) + {\bf f}(\vec p) \ri)}.
 \label{1122}
\ee
Indeed one verifies that the expression \eqref{1122} has no jumps and thus is a rational function; now both the numerator and denominator vanish at $\infty_{1,2}, p_1,\dots, p_{g-1}$. 
Evaluating, for example, at $z=a_1$ where $\mathfrak u(a_1)=0$ shows that \eqref{1121} and \eqref{1122} are the same constant.
\end{proof}

According to Proposition \ref{exacteigenfunctions},  we can now proceed to the evaluation of the approximation to the eigenfunctions of the operator $\wh K$.

\bth \label{asympteigf}
Let $\k_n$ be the $n$-th approximate eigenvalue and $n$ sufficiently large. Then, uniformly in each closed subset of $I=I_i\cup I_e$ not containing any endpoints $a_l,~l=1,2\dots,2g+2$, we have  the approximation of the eigenfunctions of $\wh K$ given by either of  the expressions
\begin{small}
\begin{equation}
\begin{split}
&\phi_{n,j}(z)  = 
{(-1)^{j}} 2  {\rm e}^{(-1)^{j+1}(d_\infty + \k_n\gg_{\infty}) -\frac {\k_n} 2 }\\
&\times \le\{  \Re \le[
\frac{
\pi C_0 \Theta\le(\mathfrak u_+(z) +(-1)^j \mathfrak u(\infty) + {\bf f}(\vec p_n) \ri) h_+(z)}{
{i}
\vec \tau_1 \!\!\cdot\!\!\nabla \Theta({\bf f}(\vec p_n)) \Theta\le(\mathfrak u_+(z) + (-1)^j\mathfrak u(\infty)+W_0\ri)}  {\rm e}^{-i\Im(d_+(z))-i \k_n\Im (\gg _+(z)) }  \ri]\chi_e \ri.\\
&\hspace{10pt}
\le.
 +i\Im  \le[
\frac{
\pi C_0 \Theta\le(\mathfrak u_+(z) +(-1)^j\mathfrak u(\infty) +{\bf f}(\vec p_n) \ri) h_+(z)}{
{i}
\vec \tau_1 \!\!\cdot\!\!\nabla \Theta({\bf f}(\vec p_n)) \Theta\le(\mathfrak u_+(z) +(-1)^j\mathfrak u(\infty)+W_0\ri)}{\rm e}^{-i\Im(d_+(z))-i \k_n\Im (\gg _+(z))} \ri]\chi_i  + \mathcal O(\k_n^{-1})\ri\},
\label{phin}
\end{split}
\end{equation}
\end{small}
$j=1,2$, where the (independent on $n$) constants $W_0, C_0$ and function $h$ were defined in \eqref{wukappa}, \eqref{defC_0} and \eqref{spinorh}, provided that these expressions  are not identically zero.
In this case, the expressions in the right hand side of \eqref{phin} are  proportional to each other by a nonzero constant in $z$.
If one of the two expressions is identically zero  for a particular value of $n$, then the other expression is not identically zero.
\et
\begin{proof}

Unraveling all the steps that transformed the RHP \ref{RHPGamma} to its final approximation, we have in summary:\\
\begin{itemize}
\item On any closed set in $\C \setminus{I}$
\bea
\Gamma (z;\l) = {\rm e}^{(d_\infty + \k \gg_{\infty})\s_3} \mathcal E \Psi  {\rm e}^{-d \s_3} {\rm e}^{-\k\gg \s_3},
\label{away}
\eea
where $\mathcal E(z)$ is the error matrix, which differs from the identity (uniformly) by an $\mathcal O(\k^{-1})$, $\k\to \infty$, term, and $\Psi$ is the solution to the RHP \ref{modelRHP};
\item The boundary values of $\Gamma$ in the interior/exterior intervals in terms of our approximation are read off from \eqref{422} and \eqref{y-def}:
\begin{small}
\bea
\Gamma_\pm (z;\l) = {\rm e}^{(d_\infty + \k \gg_{\infty})\s_3} \mathcal E \Psi_\pm  {\rm e}^{-d_\pm \s_3} \le[
\begin{matrix} 
1 &\ds  \pm \frac { {\rm e}^{-\k(2\gg_\pm +1)}}{iw} \chi_{_i}\\
\ds \pm i w {\rm e}^{\k(2\gg_\pm -1)}  \chi_{_e} &1
\end{matrix}
\ri] {\rm e}^{-\k\gg_\pm  \s_3},\, z\in I.
\label{733}
\eea
\end{small}
\end{itemize}

Using formula \eqref{phi_j}  (with $\l= e^{-\k}$), we can approximate $\varphi_j(z;\l)$. Referring now to \eqref{phi_j}, we point out that $A_j(z;\l)$ is analytic on $I_e$, and $B_j(z;\l)$ is analytic on $I_i$.
Therefore we can use either the $+$ or $-$ boundary values of $\Psi$ when we approximate $\varphi_j(z;\l)$ in terms of $\Psi$. Using, for example, the $+$ boundary value, we find that for $z\in I$ one has:
\begin{equation}
\begin{split}
&\varphi_j(z;\l) = 
\frac{{\rm e}^{(-1)^{j+1}(d_\infty + \k\gg_{\infty}) } }{\sqrt{w(z)}}\le[\Psi_{j1+} {\rm e}^{-d_+-\k\gg _+}  + 
i w {\rm e}^{\k \gg_+ -\k +d_+}\Psi_{j2+}  + \mathcal O(\k^{-1}) \ri]\chi_e
\\
&\quad + i \sqrt{w(z)}{\rm e}^{(-1)^{j+1}(d_\infty + \k\gg_{\infty})}\le[
\frac{\Psi_{j1+} }{iw}{\rm e}^{-\k \gg_+ -\k -d_+ } + \Psi_{j 2+}  {\rm e}^{ \k \gg_+ + d_+}+ \mathcal O(\k^{-1}) 
\ri]\chi_i\\
=&
{\rm e}^{(-1)^{j+1}(d_\infty + \k\gg_{\infty}) -\frac \k 2 }
 \le[
 \Psi_{j1+}  {\rm e}^{-\frac {d_+ -d_-}2-\k\frac{\gg _+  - \gg_-}2 } 
 + {\rm e}^{\k \frac{\gg_+-\gg_-}2 +\frac{d_+-d_-}2 }\Psi_{j 1-}+ \mathcal O(\k^{-1})  \ri]\chi_e
\\
  + &{{\rm e}^{(-1)^{j+1}(d_\infty + \k\gg_{\infty}) - \frac \k 2}} \le[
 \Psi_{j 1+} {\rm e}^{-\k \frac{\gg_+-\gg_-}2  -\frac {d_+ -d_-}2 } -  \Psi_{j1-}  {\rm e}^{ \k\frac { \gg_+ -\gg_-}2  +\frac{ d_+-d_-}2}
+ \mathcal O(\k^{-1}) \ri]\chi_i.
\label{739}
\end{split}
\end{equation}
We have used the jump conditions in RHP \ref{modelRHP} in \eqref{739} and the properties \eqref{geqm} of $\gg$. 
In order to compute the eigenfunctions of $\wh K$ we need to take the residue of \eqref{739} according to Proposition \ref{exacteigenfunctions} at the eigenvalues $\l_n = {\rm e}^{-\k_n^{exact}}$ with respect to $ \k$ (since $\frac {d\l}\l = -d\k$) so that we need to use \eqref{res-Psi}. Given that  the approximation \eqref{739} is uniform in a strip around the positive real $\k$--axis, the process of taking residues does not affect the error term. 
We also recall that $\Psi_{j1-} =\ov {\Psi_{j1+}}$, see Proposition \ref{symmetry},  and the same applies to the residues at $\k_n\in \R$. Moreover, 
it follows from  \eqref{defg} and   \eqref{dz}, respectively, that
$\gg_+-\gg_-=2i\Im \gg_+$  
and $d_+-d_-=2i\Im d_+$ on the main arcs. Thus we obtain \eqref{phin}.
Note also that, according to assertion {\bf (2)} of Lemma \ref{lemmarespsi},  at least one of $\phi_{n,j}$, $j=1,2$, is not identically zero. Moreover, assertion {\bf (3)} of Lemma \ref{lemmarespsi} implies that  if both  $\phi_{n,j}$, $j=1,2$ are nonzero functions, then they are multiple of each other.
\end{proof}
In order to provide the approximation of the {\em normalized} eigenfunctions for $\wh K$ we need to compute the norms of the approximate eigenfunctions from Theorem \ref{asympteigf}.

\bl
\label{technical}
{\bf(1)} The following identities hold for $j=1,2$:
\bea
\label{psipsi}
  N_{j}(\vec p_n)&&:= -\frac i{\pi^2}\oint_{B_1}\res{\k=\k_n} \!\!\! \Psi_{j1}(z;\k)\!\!\! \res{\k=\k_n} \!\!\! \Psi_{j2}(z;\k)d z
=\cr
&& =\frac {\Theta({\mathbf f}(\vec p_n)+(-1)^j2\mathfrak u(\infty))} {\Theta(W_0+(-1)^j 2\mathfrak u(\infty) )}\frac{ [\mathbb A^{-1} \nabla \Theta(W_0)]_g }
{i \vec \tau_1\!\!\cdot \!\! \nabla \Theta({\mathbf f}(\vec p_n))}.
\eea
{\bf (2)} The function $N_j(\vec p)$ is a real analytic function of $\vec p\in A_2\times \dots \times A_g$. It vanishes to second order at $p_{g-1} =\infty_{l}$, where $\infty_l$ is the point at $z=\infty$ on the sheet $l=1,2$, and has no other zeroes.
\el
\begin{proof}
{\bf (1)} 
First notice that the ratios in \eqref{psipsi} are
well-defined because $\Theta(W_0+(-1)^j2\mathfrak u(\infty) )\neq 0$ 
(third bullet point of Lemma \ref{lemma1}) and ${\nabla}\Theta(\mathbf f(\vec p_n))\neq 0$ ( Lemma \ref{lemmatransversal}).
Using the expressions for $\Psi_{ij}$ from \eqref{res-Psi}, we obtain
\bea\label{839}
&& \ \ \ \ \ \ - i\oint_{B_1}\res{\k=\k_n}\!\!\! \Psi_{j1}(z;\k)\!\!\! \res{\k=\k_n} \!\!\! \Psi_{j2}(z;\k)d z =\frac{ -iC_0^2 \pi^2}{ (\vec \tau_1\!\!\cdot \!\! \nabla \Theta(\mathbf f(\vec p_n)))^2}\times \\
&& \times \oint_{B_1} 
\!\!\!\! \frac{\Theta(\mathfrak u(z) +(-1)^j \mathfrak u(\infty) + \mathbf f(\vec p_n))\Theta(-\mathfrak u(z) +(-1)^j \mathfrak u(\infty) + \mathbf f (\vec p_n)) h^2(z) d z }
{\Theta(\mathfrak u(z) +(-1)^j \mathfrak u(\infty) + W_0)\Theta(-\mathfrak u(z) +(-1)^j \mathfrak u(\infty) +W_0)}.
\nonumber
\eea
Now the integrand is a differential with no poles (the finite zeroes in the denominator cancel with those of $h^2$ by construction of $W_0$, and the zeroes at $\infty$ cancel with those in the numerator). 
The numerator also vanishes at the points $z= p_j^{(n)}$ of the components of  ${\vec p_n} = (p_1^{(n)} , \dots, p_{g-1}^{(n)})$ and their reflections on the opposite sheet; in other words it is a holomorphic differential of the form $\propto \frac{\prod_{j=1}^{g-1}(z-z_j^{(n)})}{R(z)}dz$. Here $z_j^{(n)}$ are the $z$--coordinates of the points $p_j^{(n)}$. 
As such, the differential must be proportional to $\vec \omega(z) \cdot \nabla \Theta(\mathbf f(\vec p_n))dz$ (Lemma \ref{lemmafay}). To decide the proportionality constant it suffices to evaluate the two expressions at $\infty$, so one finds  
\beas
\frac{\Theta(\mathfrak u(z) +(-1)^j \mathfrak u(\infty) + \mathbf f_n)\Theta(-\mathfrak u(z)+( -)^j \mathfrak u(\infty) + \mathbf f_n) h^2(z) }
{\Theta(\mathfrak u(z) +(-1)^j \mathfrak u(\infty) + W_0)\Theta(-\mathfrak u(z) +(-1)^j \mathfrak u(\infty) +W_0)} = \\
=
\frac {\Theta(\mathbf f_n+(-1)^j2\mathfrak u(\infty))} {C_0\Theta(W_0+(-1)^j2\mathfrak u(\infty) )} \vec \omega(z) \cdot \nabla \Theta(\mathbf f_n), 
\eeas
and the right-hand side of \eqref{839} becomes
\bea
 \frac {-i\Theta(\mathbf f_n+(-1)^j2\mathfrak u(\infty))} {\Theta(W_0+(-1)^j2\mathfrak u(\infty) )}\frac{ C_0 \pi^2}{ (\vec \tau_1\!\!\cdot \!\! \nabla \Theta(\mathbf f_n))^2} \oint_{B_1} 
 \vec \omega(z) \cdot \nabla \Theta(\mathbf f_n) 
= \cr
=  \frac {\Theta(\mathbf f_n+(-1)^j2\mathfrak u(\infty))} {\Theta(W_0+(-1)^j2\mathfrak u(\infty) )}\frac{ C_0 \pi^2}
{i \vec \tau_1\!\!\cdot \!\! \nabla \Theta(\mathbf f_n)}. 
\label{841}
 \eea
Thus, recalling formula \eqref{defC_0} for $C_0$  completes the proof.
 {\bf (2)} 
It can be shown using the periodicity of $\Theta$ and the Schwarz symmetry that $N_j(\vec p)$ is real--valued 
when $\vec p\in A_2\times \dots \times A_g$.
According to Corollary \ref{generalThetadiv}, the term $\Theta({\bf f}(\vec p) +(-1)^j 2\u(\infty))$ in \eqref{psipsi}  vanishes 
on $ A_2 \times \dots A_g$ 
if {and only if}  the point $p_{g-1}\in A_g$ coincides with $\infty_j$, 
see Figure \ref{homology}. To prove that the vanishing is  of second order, consider the case $j=1$ for definiteness:
\bea\label{fvec_p}
{\bf f}(\vec p) - 2\u(\infty) = \u(p_{g-1}) + \sum_{j=1}^{g-2} \u(p_j) - 2\u(\infty)  -\K = \u(p_{g-1}) - \le(\sum_{j=1}^{g-2} \u(\wh p_j) + 2\u(\infty)\ri)  +\K,
\eea
where the equalities are in the Jacobian and we have used the fact  that $\K = -\K$ in the Jacobian
Then, by Theorem \ref{generalTheta},   $\Theta({{\bf f}(\vec p)} - 2\u(\infty))$
has a second order zero at $p_{g-1}=\infty$ on the main sheet.
\end{proof}

\bc
\label{doublecover}
The functions $N_j(\vec p)$ have constant sign on the torus $A_2\times \dots A_g$.
 The function $\sqrt{N_j(\vec p)}$ can be defined analytically on the double cover of  $A_2\times \dots A_g$.
\ec
\begin{proof}
By Lemma \ref{technical}, {(2)}, $N_j(\vec p)$ are real and vanish quadratically at $p_{g-1} =\infty_j$ and thus the square root can be analytically continued across the hypersurface $p_{g-1}  = \infty_j$. Since this is the only sign change on the cycle $p_{g-1} \in A_g$, we conclude that $\sqrt{N_j(\vec p)}$ gains the opposite sign under the analytic continuation in the variable $p_{g-1}\in A_g$. Thus, it is well-defined on the  double cover of the $A_g$ component.
\end{proof}
\br
Although it is nontrivial  to prove it directly using  properties of the Theta functions, it will follow from the  
Proposition \ref{asymptnorms} below  that $N_j(\vec p)$ is actually {\em positive} since it appears in \eqref{748} as an approximation of a positive quantity (norm).
\er

 \bp
 \label{asymptnorms}
 {\bf (1)}
The norms in $L^2({I}) $ of the eigenfunctions $\phi_{n,j}$ appearing in Theorem \ref{asympteigf}, are
 \be
 \label{748}
\|\phi_{n,j}\|^2 =     2{\rm e}^{(-1)^{j+1}2(d_\infty + \k_{n}\gg_{\infty}) - \k_n} \le(\pi^2 N_j(\vec p_n)
 + \mathcal O(\k_n^{-1})\ri),~~~j=1,2.
 \ee
{\bf (2)} There exists $\nu>0$ such that for all $n\in\N$
\be\label{748a}
\max_{j=1,2}|N_j(\vec p_n)| >\nu.
\ee
 
\ep
\begin{proof}
 {\bf (1)} 
We will, for concreteness, consider $j=1$ (i.e., the first row in \eqref{res-Psi}).
First of all, recall that by Proposition \ref{exacteigenfunctions}, $\phi_{n,1}(z)$ is related to $\Gamma(z;\l )$. Thus its squared norm in $L^2(I_e \sqcup I_i)$ is the sum of the two pieces on $I_i$ and $I_i$. 
Since the eigenfunctions are the residues of $\Gamma(z;\l)$ as per Proposition \ref{exacteigenfunctions}, we need to compute the expressions below (where $J(H):= H_+-H_-$):
\bea
\begin{split}
\int_{I_e} |\phi_{n,1}|^2  d z  &= 
\int_{I_e} \le(\res{\l=\l_n} \Gamma_{11}(z;\l)\frac {1}\l\ri)^2\frac{dz}{ w(z)} 
=\cr
&=   \int_{I_e}\le(  \res{\l=\l_n}\Gamma_{11}(z;\l) \frac {1}\l \ri) \le(\res{\l=\l_n} i\l J(\Gamma_{12}(z;\l))\frac {1}\l\ri) dz=\\
& =i   \le(\oint_{B_0} + \oint_{B_g}\ri) 
\res{\l=\l_n}\le(\Gamma_{11}(z;\l) \frac {1}{\l} \ri)\res{\l=\l_n} \le( \Gamma_{12}(z;\l) \ri) d z;
\end{split}
\label{744} 
\eea
\bea
\begin{split}
\int_{I_i} |\phi_{n,1}|^2  d z  &=\int_{I_i}\le(\res{\l=\l_n}  \Gamma_{12}(z;\l) \frac {1}\l\ri)^2 \!\!\!\! {w(z)} dz=\\
&= -i  \int_{I_i} \le(\res{\l=\l_n} \l J(\Gamma_{11}(z;\l))\frac {1}\l\ri)  \le(\res{\l=\l_n}  \Gamma_{12}(z;\l) \frac {1}\l\ri) dz \\
& 
= -i \oint_{B_1} \le(\res{\l=\l_n} \Gamma_{11}(z;\l)\ri) \le(\res{\l=\l_n}  \Gamma_{12}(z;\l)\frac {1}\l\ri) d z.
\end{split}
\label{745} 
\eea
Here we have used Proposition \ref{exacteigenfunctions} (with $j=1$) and Remark \ref{SchwartzGamma}, which  imply that $\Gamma_{11}$ is real on $I_e$, and $\Gamma_{12}$ is real on $I_i$.
 Both residues of the matrix entries of $\Gamma$ in the integrands above  behave as $\mathcal O(z^{-1})$ at infinity: 
this is always true for the element $\Gamma_{12}$, and the residue of  $\Gamma_{11}(z;\l)$ at $\l_n$ is proportional to $h_n(z)$ of \eqref{svd-def}, which has the mentioned behavior. Then
 \be
 \int_{I_e} |\phi_n|^2  d z   =  \int_{I_i} |\phi_n|^2  d z\label{halfnorms},
 \ee
as can be easily seen by deforming the corresponding contours of integration, see Figure \ref{homology}.

Since the estimate of $\mathcal E(z)=\mathcal E(z;\k)$  in \eqref{away} is uniform as long as $\k$ stays away from  $\k_n$, evaluation of the residue (which is a contour integral over a small fixed circle around $\k_n$) is affected by the same error term. Plugging \eqref{away} into the  contour integral in \eqref{745} gives
\bea
 && \ \ \ -i\oint_{B_1}\bigg(\res{\l=\l_n} \Gamma_{11}(z;\l)  \bigg)\le(\res{\l=\l_n}  \Gamma_{12}(z) \frac {1}\l\ri)d z =\\
 &&= -i {\rm e}^{2d_\infty +2 \k_n \gg_\infty -\k_n}
\le( \oint_{B_1} 
\bigg(\res{\k=\k_n} \!\!\! \Psi_{11}(z;\k) \res{\k=\k_n} \Psi_{12} (z;\k)\bigg) dz + \mathcal O(\k_n^{-1})\ri).
\nonumber
\label{838}
\eea
The integral  of the leading contribution  can be computed explicitly  using Lemma \ref{technical}.
 Note that the factor of $2$ in \eqref{748} follows from \eqref{halfnorms}.
The final statement follows from the fact that $C_0=[\mathbb A^{-1} \nabla\Theta]_g\neq 0$, see Theorem \ref{theoremM}.

 {\bf (2)} 
 The two functions  $N_j (\vec p)$ are defined  on the compact torus $A_2\times...\times  A_g$  by Lemma \ref{technical}, {\bf (2)}. They vanish at the two hypersurfaces $p_{g-1}= \infty_{j}$ that have no intersection. Thus a simple compactness argument  implies that 
 $$ \min_{\vec p\in A_2\times...\times  A_g} \max_{j=1,2} |N_j(\vec p)| = \nu >0.$$ 
\end{proof}

\bc\label{cor-Ups12}
Functions 
\bea\label{Upsns}
\Upsilon_j(z;\vec p)=
\sqrt{ \frac {\Theta(W_0\!+\!(-1)^j2\mathfrak u(\infty) )}{\Theta(\mathbf f(\vec p)\!+\! (-1)^j2\mathfrak u(\infty))}
 \frac {[\mathbb A^{-1} \nabla \Theta(W_0)]_g }{i\vec \tau_1\!\!\cdot \!\! \nabla \Theta(\mathbf f(\vec p)) }}
\frac{\Theta\le(\mathfrak u(z)\! +\!\!(-1)^j\!\mathfrak u(\infty) + \mathbf f(\vec p) \ri) h(z)}
{ \Theta\le(\mathfrak u(z) \!+\!(-1)^j \mathfrak u(\infty)+W_0\ri)},
\eea
$j=1,2$, are analytic in $z$ on $Z_0$ and in $\vec p$ on the double covering of the torus  $ A_2 \times \dots A_g$.
Here $Z_0 := \bar\C \setminus [a_1,a_{2g+2}]$, but the boundary points on both sides of the interval $ (a_1,a_{2g+2})$ are also included in $Z_0$. Moreover, they coincide on
$Z_0 \times A_2 \times \dots A_g$ modulo factor $(-1)$. 
\ec
\begin{proof}
 According to \eqref{psipsi}, we have
\be\label{Ups}
\Upsilon_j(z;\vec p) =
\frac{\Theta\le(\mathfrak u(z)\! +\!\!(-1)^j\!\mathfrak u(\infty) + {\mathbf f}(\vec p) \ri) h(z)}
{\sqrt{N_j(\vec p)}  \Theta\le(\mathfrak u(z) \!+\!(-1)^j \mathfrak u(\infty)+W_0\ri)},\ \ j=1,2.
\ee
Then the analyticity of $\Upsilon_j(z;\vec p)$ everywhere on $Z_0 \times A_2 \times \dots A_g$ except possibly zeroes
of the denominator 
follows from Corollary \ref{doublecover}. According to  Lemma \ref{lemma1},
$\Theta\le(\mathfrak u(z) \!+\!(-1)^j \mathfrak u(\infty)+W_0\ri)\neq 0$ on $Z_0\setminus \{\infty\}$.
Note that  at $z=\infty$ the ratio  $\frac{h(z)}{\Theta\le(\mathfrak u(z) \!+\!(-1)^j \mathfrak u(\infty)+W_0\ri)}$ has a removable singularity. Thus, it remains to consider a simple zero of  $\sqrt{N_j(\vec p)}$ at $p_{g-1}=\infty_j$, $j=1,2$,
see Corollary \ref{doublecover}.

The function $\Theta(\u(z) +(-1)^j \u(\infty) + {\bf f}(\vec p))$  vanishes identically in $z$ when  $p_{g-1}=\infty_j$.
For example,  for $j=1$, equation \eqref{fvec_p} implies  
\be
\Theta(\u(z) - \u(\infty) + {\bf f}(\vec p)) = 
\Theta\le(   \u(p_{g-1}) - \le(\sum_{k=1}^{g-2} \u(\wh p_{k}) + \u(\infty) + \u(\wh z)  \ri)  +\K  \ri)=0
\ee
at $p_{g-1}=\infty_1$ ($\infty$ on the main sheet) and this zero is simple according to Theorem \ref{generalTheta}. 
Since $\sqrt{N_1(\vec p)}$ also vanishes linearly, the singularity at   $p_{g-1}=\infty_1$ is removable and we can analytically continue  $\Upsilon_{1}(z;\vec p)$ across the hypersurface $p_{g-1} = \infty_1$. The resulting function lives naturally on the double cover by virtue of  Corollary \ref{doublecover}. Similar arguments work for $\Upsilon_{2}(z;\vec p)$.

We now show that on their common domain of analyticity (i.e., when $p_{g-1}\neq \infty_{1,2}$) $\Upsilon_{1,2}$ coincide  up to a sign and thus, by analytic continuation, they are proportional to each other on the whole double covering of $A_2 \times \dots A_g$.
According to Lemmas \ref{lemmarespsi}, \ref{technical},
\be\label{expr-Ups}
\frac{{(-1)^j}}{\sqrt{N_j(\vec p_n)} }
\frac{\Theta\le(\mathfrak u(z)\! +\!\!(-1)^j\!\mathfrak u(\infty) +{\mathbf f}(\vec p_n) \ri) h(z)}
{ \Theta\le(\mathfrak u(z) \!+\!(-1)^j \mathfrak u(\infty)+W_0\ri)}
 =\frac{\pi \res{\k=\k_n} \!\!\! \Psi_{j1}(z;\k)}
{\sqrt{- i\oint_{B_1}\res{\k=\k_n} \!\!\! \Psi_{j1}(z;\k)
\!\!\! \res{\k=\k_n} \!\!\! \Psi_{j2}(z;\k)d z}}.
\ee
But the rows  of \eqref{res-Psi} are proportional, see  Lemma \ref{lemmarespsi}, {\bf (3)}, so
the right hand side of \eqref{expr-Ups} does not depend on $j$ (modulo sign).
\end{proof}

\br \label{rem-otherUps}
The statement of Corollary \ref{cor-Ups12} remains valid if the last fraction in \eqref{Upsns}
is replaced by 
\be
\frac{\Theta\le(-\mathfrak u(z)\! +\!\!(-1)^j\!\mathfrak u(\infty) + \mathbf f(\vec p) \ri) h(z)}
{ \Theta\le(-\mathfrak u(z) \!+\!(-1)^j \mathfrak u(\infty)+W_0\ri)}.
\nonumber
\ee
\er

In view of Corollary \ref{cor-Ups12}, we denote by $\Upsilon(z;\vec p)$ a function on $Z_0 \times A_2 \times \dots A_g$
that coincides (modulo sign) with both  $\Upsilon_{1,2}(z;\vec p)$. The normalized singular functions  $\wh f_n$, $\wh h_n$
from \eqref{svd-def2} are expressed through $\Upsilon(z;\vec p_n)$ in Theorem \ref{cor-first} below.
Since these normalized  (real) singular functions are defined up to a sign, the ambiguity in the sign of 
$\Upsilon(z;\vec p)$ will not affect our result. Moreover, for a given $n\in\N$ we can choose either
of $\Upsilon_{1,2}(z;\vec p_n)$ for $\Upsilon(z;\vec p_n)$, as convenient.  We remind that
 $\Upsilon_{1,2}(z;\vec p_n)$ is given by \eqref{Upsns}, where $\mathbf f(\vec p_n)=W(\k_n)-W_0$.

\bth \label{cor-first}
{\bf (1)} The  singular functions $\wh f_n(z)$ and $\wh h_n(z)$ of the 
system in \eqref{svd-def2} {\em normalized} in $L^2(I_i)$ and $L^2(I_e)$, respectively,  are asymptotically given by
\bea
\begin{split}
\wh f_n(z) =
{i} \Im  \le[ 2\Upsilon_+(z;\vec p_n){\rm e}^{-i \k_n \Im (\gg_+(z))  -i\Im (d_+(z)) } \ri] + 
\mathcal O({\k}_n^{-1}), \ z\in I_i,\cr
\wh  h_n(z) = 
\Re  \le[ 2\Upsilon_+ (z;\vec p_n) {\rm e}^{-i \k_n \Im (\gg_+(z))  -i\Im (d_+(z)) } \ri]  + 
\mathcal O({\k}_n^{-1}),\ z\in I_e,
\end{split}
\label{843}
\eea
where the approximation is uniform in any compact subset of the interior of $I_i,  I_e$, respectively.  
The subscript `$+$' in $\Upsilon_+ (z;\vec p_n)$ denotes that $z$ is taken on the upper shore of $\C\setminus [a_1,a_{2g+2}]$.
\et
\begin{proof}
Equations \eqref{843} follow from
Proposition \ref{asymptnorms}, Theorem \ref{asympteigf}, Corollary \ref{cor-Ups12},  Theorem \ref{theo-ass} and \eqref{svd-K}. 
The error estimate in \eqref{843}
follows from  Proposition \ref{asymptnorms}, {\bf (2)}.  
\end{proof}

\br\label{signchange2}
The  shape of  $\wh h_n, \wh f_n$ 
is that of rapid oscillations driven by ${\rm e}^{-{i\k_n}\Im \gg_+(z)}$ and modulated by  the prefactors,
 see Figure \ref{asympteignum} for numerical simulations of some particular singular functions.
The phase of  $\exp (-i\k_n \Im g_+)$ has a total increment of $ \k_n \Im (g_+(a_{2g}) - g_+(a_3))$ {when $z$ changes from $a_3$ to $a_{2g}$}, which, by \eqref{defg} (see also Figure \ref{homology}), is $\k_n \oint_{B_1} \omega_1 = \k_n \tau_{11}$. Given 
the estimate  $\k_n \sim \frac {n \pi}{\tau_{11}}$ from \eqref{asskap-exact}, we see that the total increment of the argument is $\Delta \arg = n\pi$, which is perfectly consistent with the expected $n$ sign changes discussed in Remark \ref{signchange}.
We point out that the prefactor has  $g$ zeroes: at $\infty$ and at the $g-1$ points $p_j=(z_j, R_j)$ described in Theorem \ref{ThetaDiveig} (all of which are in the gaps between the intervals (main arcs)). The prefactor is not a fast oscillating function, since its $n$ dependence is only via the shift ${\bf f}_n$, which is constant in $z$.
\er
The following corollary for the singular functions  $ f_n=\sqrt{w}\wh f,~ h_n = \sqrt{w}\wh h_n$  is an immediate consequence of Theorem  \ref{cor-first} and \eqref{svd-def2}.
\bc \label{cor-sing-func}
 The {singular functions} $f_n(z)$ and $h_n(z)$  of the system \eqref{svd-def} {\em normalized} in $L^2(I_i,\frac 1{w(z)})$ and $L^2(I_e,\frac 1{w(z)})$, respectively,  are asymptotically given by
\bea
\begin{split}
 f_n(z) =
\sqrt{w(z)} \Im  \le[ 2\Upsilon_+(z;\vec p_n){\rm e}^{-i \k_n \Im (\gg_+(z))  -i\Im (d_+(z)) } \ri] + 
\mathcal O({\k}_n^{-1}), \  \ z\in I_i,\cr
  h_n(z) = 
\sqrt{w(z)}\Re  \le[ 2\Upsilon_+ (z;\vec p_n) {\rm e}^{-i \k_n \Im (\gg_+(z))  -i\Im (d_+(z)) } \ri]  + 
\mathcal O({\k}_n^{-1}),\  \ z\in I_e,
\end{split}\label{844}
\eea
where the approximation is uniform in any compact subset of the interior of $I_i,  I_e$, respectively.
\ec

\appendix
\section{Some basic facts about Theta functions and divisors}
\label{thetaapp}
The reference for all the following theorems are  \cite{FarkasKra} or, occasionally, \cite {Faybook}.
We quote here certain results about general Riemann surfaces of genus $g\in \N$.

The {\bf Riemann Theta function} associated to a symmetric matrix $\tau$ with strictly positive imaginary part  
is the function of the vector argument $\vec z\in\C^g$ given by
\be
\Theta(\vec z,\tau):= \sum_{\vec n\in \Z^g} \exp\bigg(i\pi   \vec n^t \cdot \tau \cdot \vec n +2i\pi \vec n^t \vec z\bigg).
\ee
Often the dependence on $\tau$ is  omitted from the notation.

\bp
\label{thetaproperties}
For any $\lambda, \mu \in \Z^g$, the Theta function has the following properties:
\bea
&\& \Theta(\vec z,\tau)  = \Theta(-\vec z,\tau);\\
&\& \Theta(\vec z + \mu + \tau\lambda, \tau) = \exp \bigg( -2i\pi \lambda^t\vec z - i\pi  \lambda^t \tau \lambda \bigg)
\Theta(\vec z,\tau).\label{thetaperiods}
\eea
\ep

We shall denote by $\Lambda_\tau = \Z^g + \tau \Z^g\subset \C^g$ the {\em lattice of periods}. 
The {\bf Jacobian}  $\mathbb J_\tau$ is the quotient $\mathbb J_\tau = \C^g\mod \Lambda_\tau$. 
It is a compact torus of real dimension $2g$ on account that the imaginary part of $\tau$ is a positive 
definite matrix (Riemann's theorem).

The general definition of the {\it vector of Riemann constants} $\mathcal K$ can be found in \cite{FarkasKra}.
For the case of a hyperelliptic Riemann surface (our case) the  following proposition
can be considered as the definition 
of $\mathcal K$. 
\bp[\cite{FarkasKra}, p. 324]\label{prop-K}
Let $a_1$ be a base-point of the Abel map $\mathfrak u(z)$  (see \eqref{Abelmap})  on the hyperelliptic Riemann surface $\Rscr$ 
of $\sqrt{\prod_{j=1}^{2g+2} (z-a_j)}$. Then the vector of Riemann constants is 
\be
\mathcal K = \sum_{j=1}^g \mathfrak u(a_{2j+1}).
\ee
\ep
\bth[\cite{FarkasKra}, p. 308]
\label{generalTheta}
Let ${\bf  f}\in \C^g$ be arbitrary, and denote by $\mathfrak u(p)$ the Abel map (extended to the whole Riemann surface).
The (multi-valued) function $\Theta(\mathfrak u(z) - {\bf  f})$ on the Riemann surface either vanishes identically or 
vanishes at $g$ points ${p}_1,\dots, {p}_g$ (counted with multiplicity).
In the latter case we have 
\be
{\bf  f} =\sum_{j=1}^{g} \mathfrak u(p_j) + \mathcal K.
\ee
\et

\br
Description of the vectors ${\bf f}$  that lead to identically vanishing $\Theta(\mathfrak u(z) - {\bf  f})$
is more involved and will not be discussed here.
\er
An immediate consequence of Theorem \ref{generalTheta} is the following statement.
\bc
\label{generalThetadiv}
The function $\Theta$ vanishes at ${\bf  e}\in \mathbb J_\tau$ if and only if  there exist $g-1$ points $p_1,\dots, p_{g-1}$ on the Riemann surface such that
\be\label{vece}
{\bf e} =\sum_{j=1}^{g-1} \mathfrak u(p_j) + \mathcal K.
\ee
\ec

\bd\label{def-thetadiv}
The {\bf Theta divisor} is the locus ${\bf  e}\in\mathbb J_\tau$ such that
$\Theta({\bf e})=0$. It will be denoted by the symbol $(\Theta)$.
\ed

\subsection {Divisors}
On a manifold the name ``divisor'' refers to a formal union of subsets of co-dimension one. On a Riemann surface of genus $g$ a divisor is a collection of points (counted with multiplicity). We are going to consider here only {\em positive} divisors, namely, with positive multiplicities. 
The following facts and definitions  are standard and can be found, for example, in \cite{FarkasKra}.

\bd
\label{defspecial}
A (positive) divisor of degree $k\leq g$ is  called {\em special} if the vector space of meromorphic functions with poles at the points of order not exceeding the given multiplicities has dimension strictly greater than $1$. (Note that the constant function is always in this space).
\ed
As the definition suggests, generic divisors of degree $\leq g$ do not admit other than the constant function in the above-mentioned vector space.
The other fact that we have used is that a divisor $\mathcal D = p_1 + \dots +p_k$ ($k\leq g$)  on a {\em hyperelliptic} Riemann 
surface $R^2 = \prod_{j=1}^{2g+2} (z-a_j)$ is special if and only if at least one pair of points is of the 
form $(z, \pm R)$ (i.e. the points are on the two sheets and with the same $z$ value).
\subsection{Derivatives of \texorpdfstring{$\Theta(u(z))$}{Tu}}
\bl[Fay]
\label{lemmafay}
{\bf (1)} Let $\mathbf f = \sum_{j=1}^{g-1} \mathfrak u(p_j) + \mathcal K$, with the divisor consisting of the points $p_j = (z_j, R_j)$ being 
non-special, see Definition \ref{defspecial}.
Then the unique differential (up to a multiplicative constant) that vanishes at  {all the} points of $p_j$  is given by
\be
\vec \omega^t(z) \cdot \nabla \Theta(\mathbf f)dz = \frac {d}{d z'} \Theta(\mathfrak u(z)-\mathfrak u(z') -\mathbf f)\bigg|_{z'=z}dz  
=\omega_{\mathcal D}(z),
\label{B1}
\ee
where the vector $\vec \o(z) dz$ of normalized holomorphic differentials is defined in \eqref{normalization}. 
In the case of a hyperelliptic curve with the Abel map as in Proposition \ref{prop-K}, the remaining $g-1$ zeroes are $\wh p_j$, where the hat denotes the exchange of sheets.
{\bf (2)} In particular, we have 
\be
\vec \omega^t(z) \cdot \nabla \Theta(W_0)dz = 
\frac {d}{d z'} \Theta(\mathfrak u(z)-\mathfrak u(z') -W_0)\bigg|_{z'=z} dz = C_0 h^2(z)dz,
\label{B2}
\ee
where $W_0$, $C_0$, and $h(z)$, are given by \eqref{W0long}, \eqref{defC_0},  and \eqref{spinorh}, respectively.
\el
\begin{proof}
{\bf (1)} 
Consider the function $F(z,z'):= \Theta(\mathfrak u(z)-\mathfrak u(z') -\mathbf f)$.  We know that, as a function of $z$, $F$ vanishes at $z=z'$ and at each $p_j$, $j=1,\dots,g-1$ (cf. Theorem \ref{generalTheta}). By the chain rule (and using the parity of $\Theta$),
\bea
\frac {d}{d z'} \Theta(\mathfrak u(z)-\mathfrak u(z') -\mathbf f)\bigg|_{z'=z}dz  = 
\sum_{j=1}^g\omega_j(z) \nabla_j \Theta(\mathbf f)dz=\omega_{\mathcal D}(z),
\eea
and thus it is a first kind differential of the form $P_{g-1}(z)dz/R(z)$, where $P_{g-1}$ is some polynomial of degree $g-1$. 
According to Theorem \ref{generalTheta}, the differential $\omega_{\mathcal D}(z)$ is not a zero differential 
since  $\mathbf f  = \sum_{j=1}^{g-1} \mathfrak u(p_j) +\mathcal K$ corresponds to a
non-special divisor (Definition \ref{defspecial}) of degree $g-1$ (see also \cite{Faybook}, page 13).
Thus, by Riemann's vanishing theorem (\cite{FarkasKra}, page 317), at least one partial  derivative of $\Theta$ is nonzero.
Since $F(p_j,z)\equiv 0$ (as a function of $z$),  then $\pa_z F(p_j, z)\big|_{z=z_j}=0$. Thus $P_{g-1}(z_j)=0$ and, 
since there are $g-1$ points, this fixes $P_{g-1}$ up to a nonzero constant. Finally, it is clear that the $2g-2$ zeroes of this differential are symmetric under the exchange of sheets. \\
{\bf (2)} Applying the previous assertion to $W_0 =\sum_{j\in J} \mathfrak u(a_j) + \mathcal K$ we see that the expression must be proportional to $[h(z)]^2$ that was defined in \eqref{spinorh}. Indeed, $[h(z)]^2 d z$ is a differential with  $g-1$  zeroes at branchpoints $a_j,\ j\in J$. Hence, by the first part of the proposition, all of the zeroes are of multiplicity two since they are invariant under the exchange of sheets.
The proportionality constant is computed by inspecting the asymptotic behavior as $z\to \infty$ on the first sheet: 
by the definition of first kind differentials \eqref{1stkind} we have  
\be
\frac {d}{d z'} \Theta(\mathfrak u(z)-\mathfrak u(z') -W_0 )\bigg|_{z'=z}  =\frac{[1,z,\dots, z^{g-1}] }{R(z)}\mathbb A^{-1} \cdot \nabla\Theta(W_0). \label{A9}
\ee
The leading behavior  of \eqref{A9} at $z=\infty$ only comes from the term multiplying the $z^{g-1}$ and thus 
\be
\frac {d}{d z'} \Theta(\mathfrak u(z)-\mathfrak u(z') -W_0 )\bigg|_{z'=z} = \frac 1 {z^2}  \le[\mathbb A^{-1} \nabla \Theta(W_0)\ri]_{g} + \mathcal O(z^{-3}).
\ee
On the other hand $h^2  = z^{-2} + \mathcal O(z^{-3})$ and thus, by comparison, we deduce that the expression for the proportionality constant $C_0$ is as in \eqref{defC_0}.
\end{proof}
\bl
\label{oddchar}
Let $W_0 = \frac 12 \le(\vec m +  \tau\vec n\ri)$ with $\vec m, \vec n\in \Z^g$. Suppose that $\vec m \cdot \vec n$ is an odd number. 
Then 
\be\label{djdkth}
 \pa_j \pa_k \Theta(W_0)=-  i\pi \le(n_j \pa_k + n_k \pa_j \ri) \Theta(W_0).  \ee
\el
\begin{proof}
Consider the function (called ``theta function with characteristics $\vec n, \vec m$)
\be
\Theta\le[{\vec n\atop \vec m}\ri](\vec z) := \exp \le[ \frac {i\pi}4 \vec n ^t \tau \vec n  -i\pi  
\vec n ^t \vec z + \frac {i\pi}2 \vec n^t \vec m \ri] \Theta\le( \vec z  - W_0\ri). 
\ee
Then one verifies by the periodicities of $\Theta$ that this is an odd function 
$
\Theta\le[{\vec n\atop \vec m}\ri](- \vec z)  ={\rm e}^{i\pi \vec n \vec m} \Theta\le[{\vec n\atop \vec m}\ri](\vec z).$
Thus $\Theta\le[{\vec n\atop \vec m}\ri](0) = 0$   vanishes at $z=0$, namely $\Theta(-W_0)=0=\Theta(W_0)$. The even derivatives of $\Theta\le[{\vec n\atop \vec m}\ri]$ must also vanish at $\vec z =0$.  Thus 
\bea
\begin{split} 0=&\pa_{j} \pa_{k} \Theta\le[{\vec n\atop \vec m}\ri]( \vec z)\bigg|_{\vec z=0} \ \ \Rightarrow\\
  \Rightarrow -\pi^2 n_j & n_k \overbrace{\Theta(-W_0)}^{=0}  - i\pi \le(n_j \pa_k + n_k \pa_j \ri) \Theta(-W_0)  + \pa_j \pa_k \Theta(-W_0)=0.
\end{split}
\eea
Using the parity of $\Theta$ we obtain the claim.
\end{proof}
\section{On strictly totally positive  kernels}
\label{PinkusApp}
 In \cite{Pinkus-rev} it is stated and proved  (adapting the notation to ours)
\begin{theorem}[\cite{Pinkus-rev}, Theorem 3.2]
Let $L:[0,1]\times [0,1]$ be a continuous, symmetric function satisfying 
\be
\det\le[L(x_j,y_\ell)\ri]_{1\leq j,\ell\leq n} \geq 0\ ,\ \ n=1,2,\dots
\label{TPL}
\ee
for any pair of ordered n-tuples and 
\be
\det\le[L(x_j,x_\ell)\ri]_{1\leq j,\ell\leq n} >0\ .
\ee
Then the spectrum of the (compact) integral operator with kernel $L$ is positive and simple.
\end{theorem}
\br
In the proof  of the auxiliary Proposition 3.3 in \cite{Pinkus-rev}  there is a separation of cases after equation (3.2) of \cite{Pinkus-rev}. We point out that in our case the inequality in \eqref{TPL}  is strict and hence only the simplest first case considered in \cite{Pinkus-rev} applies. Indeed, \cite{Pinkus-rev} introduces $F(x,y)$,  the kernel of the  projector onto the (finite dimensional) eigenspace of the largest eigenvalue $\lambda_0>0$. Then 
$$
F(x,y) = \lambda_0 \int_0^1 L(x,z)F(z,y) dz =  \lambda_0 \int_0^1 F(x,z) L(z,y)dz
$$
and, since $F(x,y) = \lim_{r\to\infty}\lambda_0^{-r} (L^r)(x,y)\geq 0$  (because it is the projector on the eigenspace of the  largest eigenvalue, which is simple), and  $L(x,y)>0$ (strictly), and $\lambda_0>0$, it follows at once that actually $F(x,y)>0$ (strictly) on the closure $[0,1]$  and thus we fall in the first case of Proposition 3.3. of \cite{Pinkus-rev}.
\er

\bibliographystyle{amsalpha}
\bibliography{bibexport}

\end{document}

%% file: IntroIeIi.pdf_t
\begin{picture}(0,0)%
\includegraphics{IntroIeIi.pdf}%
\end{picture}%
\setlength{\unitlength}{3947sp}%
\begingroup\makeatletter\ifx\SetFigFont\undefined%
\gdef\SetFigFont#1#2#3#4#5{%
  \reset@font\fontsize{#1}{#2pt}%
  \fontfamily{#3}\fontseries{#4}\fontshape{#5}%
  \selectfont}%
\fi\endgroup%
\begin{picture}(6423,753)(2311,-2305)
\put(8551,-1711){\makebox(0,0)[lb]{\smash{{\SetFigFont{12}{14.4}{\familydefault}{\mddefault}{\updefault}{\color[rgb]{0,0,0}$a_8$}%
}}}}
\put(4951,-1711){\makebox(0,0)[lb]{\smash{{\SetFigFont{12}{14.4}{\familydefault}{\mddefault}{\updefault}{\color[rgb]{0,0,1}$I_i$}%
}}}}
\put(4876,-2236){\makebox(0,0)[lb]{\smash{{\SetFigFont{12}{14.4}{\familydefault}{\mddefault}{\updefault}{\color[rgb]{0,0,0}$I_e$}%
}}}}
\put(2326,-1786){\makebox(0,0)[lb]{\smash{{\SetFigFont{12}{14.4}{\familydefault}{\mddefault}{\updefault}{\color[rgb]{0,0,0}$a_1$}%
}}}}
\put(3076,-1786){\makebox(0,0)[lb]{\smash{{\SetFigFont{12}{14.4}{\familydefault}{\mddefault}{\updefault}{\color[rgb]{0,0,0}$a_2$}%
}}}}
\put(7051,-1711){\makebox(0,0)[lb]{\smash{{\SetFigFont{12}{14.4}{\familydefault}{\mddefault}{\updefault}{\color[rgb]{0,0,0}$a_7$}%
}}}}
\put(4576,-2086){\makebox(0,0)[lb]{\smash{{\SetFigFont{12}{14.4}{\familydefault}{\mddefault}{\updefault}{\color[rgb]{0,0,1}$a_4$}%
}}}}
\put(3826,-2086){\makebox(0,0)[lb]{\smash{{\SetFigFont{12}{14.4}{\familydefault}{\mddefault}{\updefault}{\color[rgb]{0,0,1}$a_3$}%
}}}}
\put(5251,-2086){\makebox(0,0)[lb]{\smash{{\SetFigFont{12}{14.4}{\familydefault}{\mddefault}{\updefault}{\color[rgb]{0,0,1}$a_5$}%
}}}}
\put(6076,-2086){\makebox(0,0)[lb]{\smash{{\SetFigFont{12}{14.4}{\familydefault}{\mddefault}{\updefault}{\color[rgb]{0,0,1}$a_6$}%
}}}}
\end{picture}%

%% file: Homology2.pdf_t
\begin{picture}(0,0)%
\includegraphics{Homology2.pdf}%
\end{picture}%
\setlength{\unitlength}{3947sp}%
\begingroup\makeatletter\ifx\SetFigFont\undefined%
\gdef\SetFigFont#1#2#3#4#5{%
  \reset@font\fontsize{#1}{#2pt}%
  \fontfamily{#3}\fontseries{#4}\fontshape{#5}%
  \selectfont}%
\fi\endgroup%
\begin{picture}(18494,5369)(54,-7433)
\put(6151,-7111){\makebox(0,0)[lb]{\smash{{\SetFigFont{14}{16.8}{\familydefault}{\mddefault}{\updefault}{\color[rgb]{0,0,0}$a_4$}%
}}}}
\put(1651,-7111){\makebox(0,0)[lb]{\smash{{\SetFigFont{14}{16.8}{\familydefault}{\mddefault}{\updefault}{\color[rgb]{0,0,0}$a_1$}%
}}}}
\put(3376,-7111){\makebox(0,0)[lb]{\smash{{\SetFigFont{14}{16.8}{\familydefault}{\mddefault}{\updefault}{\color[rgb]{0,0,0}$a_2$}%
}}}}
\put(4651,-7111){\makebox(0,0)[lb]{\smash{{\SetFigFont{14}{16.8}{\familydefault}{\mddefault}{\updefault}{\color[rgb]{0,0,0}$a_3$}%
}}}}
\put(3826,-3736){\makebox(0,0)[lb]{\smash{{\SetFigFont{20}{24.0}{\familydefault}{\mddefault}{\updefault}{\color[rgb]{0,0,0}$A_1$}%
}}}}
\put(6676,-3736){\makebox(0,0)[lb]{\smash{{\SetFigFont{20}{24.0}{\familydefault}{\mddefault}{\updefault}{\color[rgb]{0,0,0}$A_2$}%
}}}}
\put(5026,-2986){\makebox(0,0)[lb]{\smash{{\SetFigFont{20}{24.0}{\familydefault}{\mddefault}{\updefault}{\color[rgb]{0,0,0}$B_1$}%
}}}}
\put(8251,-3436){\makebox(0,0)[lb]{\smash{{\SetFigFont{20}{24.0}{\familydefault}{\mddefault}{\updefault}{\color[rgb]{0,0,0}$B_2$}%
}}}}
\put(14851,-3511){\makebox(0,0)[lb]{\smash{{\SetFigFont{20}{24.0}{\familydefault}{\mddefault}{\updefault}{\color[rgb]{0,0,0}$B_g$}%
}}}}
\put(16726,-3661){\makebox(0,0)[lb]{\smash{{\SetFigFont{20}{24.0}{\familydefault}{\mddefault}{\updefault}{\color[rgb]{0,0,1}$A_g$}%
}}}}
\put(15751,-7111){\makebox(0,0)[lb]{\smash{{\SetFigFont{14}{16.8}{\familydefault}{\mddefault}{\updefault}{\color[rgb]{0,0,0}$a_{2g+2}$}%
}}}}
\put(14026,-7111){\makebox(0,0)[lb]{\smash{{\SetFigFont{14}{16.8}{\familydefault}{\mddefault}{\updefault}{\color[rgb]{0,0,0}$a_{2g+1}$}%
}}}}
\put(13126,-3811){\makebox(0,0)[lb]{\smash{{\SetFigFont{20}{24.0}{\familydefault}{\mddefault}{\updefault}{\color[rgb]{0,0,0}$A_0$}%
}}}}
\put(2401,-3436){\makebox(0,0)[lb]{\smash{{\SetFigFont{20}{24.0}{\familydefault}{\mddefault}{\updefault}{\color[rgb]{0,0,0}$B_0$}%
}}}}
\put(12526,-7186){\makebox(0,0)[lb]{\smash{{\SetFigFont{14}{16.8}{\familydefault}{\mddefault}{\updefault}{\color[rgb]{0,0,0}$a_{2g}$}%
}}}}
\end{picture}%

%% file: Lenses.pdf_t
\begin{picture}(0,0)%
\includegraphics{Lenses.pdf}%
\end{picture}%
\setlength{\unitlength}{3947sp}%
\begingroup\makeatletter\ifx\SetFigFont\undefined%
\gdef\SetFigFont#1#2#3#4#5{%
  \reset@font\fontsize{#1}{#2pt}%
  \fontfamily{#3}\fontseries{#4}\fontshape{#5}%
  \selectfont}%
\fi\endgroup%
\begin{picture}(15330,4871)(361,-7737)
\put(4651,-4336){\makebox(0,0)[lb]{\smash{{\SetFigFont{14}{16.8}{\familydefault}{\mddefault}{\updefault}{\color[rgb]{0,0,0}$a_{3}$}%
}}}}
\put(601,-4336){\makebox(0,0)[lb]{\smash{{\SetFigFont{14}{16.8}{\familydefault}{\mddefault}{\updefault}{\color[rgb]{0,0,0}$a_{1}$}%
}}}}
\put(7276,-4336){\makebox(0,0)[lb]{\smash{{\SetFigFont{14}{16.8}{\familydefault}{\mddefault}{\updefault}{\color[rgb]{0,0,0}$a_4$}%
}}}}
\put(15151,-4336){\makebox(0,0)[lb]{\smash{{\SetFigFont{14}{16.8}{\familydefault}{\mddefault}{\updefault}{\color[rgb]{0,0,0}$a_{2g+2}$}%
}}}}
\put(3301,-4336){\makebox(0,0)[lb]{\smash{{\SetFigFont{14}{16.8}{\familydefault}{\mddefault}{\updefault}{\color[rgb]{0,0,0}$a_{2}$}%
}}}}
\put(12376,-4336){\makebox(0,0)[lb]{\smash{{\SetFigFont{14}{16.8}{\familydefault}{\mddefault}{\updefault}{\color[rgb]{0,0,0}$a_{2g+1}$}%
}}}}
\put(8251,-4336){\makebox(0,0)[lb]{\smash{{\SetFigFont{14}{16.8}{\familydefault}{\mddefault}{\updefault}{\color[rgb]{0,0,0}$a_{2g-1}$}%
}}}}
\put(11101,-4336){\makebox(0,0)[lb]{\smash{{\SetFigFont{14}{16.8}{\familydefault}{\mddefault}{\updefault}{\color[rgb]{0,0,0}$a_{2g}$}%
}}}}
\put(1651,-4486){\makebox(0,0)[lb]{\smash{{\SetFigFont{20}{24.0}{\familydefault}{\mddefault}{\updefault}{\color[rgb]{0,0,0}$\mathcal L_e^{(-)}$}%
}}}}
\put(1651,-3661){\makebox(0,0)[lb]{\smash{{\SetFigFont{20}{24.0}{\familydefault}{\mddefault}{\updefault}{\color[rgb]{0,0,0}$\mathcal L_e^{(+)}$}%
}}}}
\put(5851,-3661){\makebox(0,0)[lb]{\smash{{\SetFigFont{20}{24.0}{\familydefault}{\mddefault}{\updefault}{\color[rgb]{0,0,0}$\mathcal L_i^{(+)}$}%
}}}}
\put(5776,-4486){\makebox(0,0)[lb]{\smash{{\SetFigFont{20}{24.0}{\familydefault}{\mddefault}{\updefault}{\color[rgb]{0,0,0}$\mathcal L_i^{(-)}$}%
}}}}
\put(9526,-3661){\makebox(0,0)[lb]{\smash{{\SetFigFont{20}{24.0}{\familydefault}{\mddefault}{\updefault}{\color[rgb]{0,0,0}$\mathcal L_i^{(+)}$}%
}}}}
\put(13576,-4486){\makebox(0,0)[lb]{\smash{{\SetFigFont{20}{24.0}{\familydefault}{\mddefault}{\updefault}{\color[rgb]{0,0,0}$\mathcal L_e^{(-)}$}%
}}}}
\put(13576,-3661){\makebox(0,0)[lb]{\smash{{\SetFigFont{20}{24.0}{\familydefault}{\mddefault}{\updefault}{\color[rgb]{0,0,0}$\mathcal L_e^{(+)}$}%
}}}}
\put(9526,-4486){\makebox(0,0)[lb]{\smash{{\SetFigFont{20}{24.0}{\familydefault}{\mddefault}{\updefault}{\color[rgb]{0,0,0}$\mathcal L_i^{(-)}$}%
}}}}
\end{picture}%

%% file: Ptrixa1.pdf_t
\begin{picture}(0,0)%
\includegraphics{Ptrixa1.pdf}%
\end{picture}%
\setlength{\unitlength}{3947sp}%
\begingroup\makeatletter\ifx\SetFigFont\undefined%
\gdef\SetFigFont#1#2#3#4#5{%
  \reset@font\fontsize{#1}{#2pt}%
  \fontfamily{#3}\fontseries{#4}\fontshape{#5}%
  \selectfont}%
\fi\endgroup%
\begin{picture}(6357,718)(6,-6592)
\put(5071,-6154){\makebox(0,0)[lb]{\smash{{\SetFigFont{12}{14.4}{\familydefault}{\mddefault}{\updefault}{\color[rgb]{0,0,0}${-}i\s_1$}%
}}}}
\put(185,-6438){\makebox(0,0)[lb]{\smash{{\SetFigFont{12}{14.4}{\familydefault}{\mddefault}{\updefault}{\color[rgb]{0,0,0}$a_1$}%
}}}}
\put(4415,-6438){\makebox(0,0)[lb]{\smash{{\SetFigFont{12}{14.4}{\familydefault}{\mddefault}{\updefault}{\color[rgb]{0,0,0}$a_{_{2g+1}}$}%
}}}}
\put(5953,-6438){\makebox(0,0)[lb]{\smash{{\SetFigFont{12}{14.4}{\familydefault}{\mddefault}{\updefault}{\color[rgb]{0,0,0}$a_{_{2g+2}}$}%
}}}}
\put(1676,-6438){\makebox(0,0)[lb]{\smash{{\SetFigFont{12}{14.4}{\familydefault}{\mddefault}{\updefault}{\color[rgb]{0,0,0}$a_2$}%
}}}}
\put(2259,-6438){\makebox(0,0)[lb]{\smash{{\SetFigFont{12}{14.4}{\familydefault}{\mddefault}{\updefault}{\color[rgb]{0,0,0}$a_{_{2j+1}}$}%
}}}}
\put(819,-6158){\makebox(0,0)[lb]{\smash{{\SetFigFont{12}{14.4}{\familydefault}{\mddefault}{\updefault}{\color[rgb]{0,0,0}${-}i\s_1$}%
}}}}
\put(3029,-6154){\makebox(0,0)[lb]{\smash{{\SetFigFont{12}{14.4}{\familydefault}{\mddefault}{\updefault}{\color[rgb]{0,0,0}$i\s_1$}%
}}}}
\put(3812,-6438){\makebox(0,0)[lb]{\smash{{\SetFigFont{12}{14.4}{\familydefault}{\mddefault}{\updefault}{\color[rgb]{0,0,0}$a_{_{2j+2}}$}%
}}}}
\end{picture}%